\documentclass[11pt]{article}

\usepackage{booktabs} 
\usepackage{anysize,amsmath,amsthm,amsfonts,amssymb,dsfont,graphics,graphicx,enumerate,multirow}

\graphicspath{{figures/}} 

\usepackage{thmtools}
\usepackage{thm-restate} 

\usepackage[ruled,vlined]{algorithm2e} 

\SetAlFnt{\small}
\SetAlCapFnt{\small}
\SetAlCapNameFnt{\small}
\SetAlCapHSkip{0pt}
\IncMargin{-\parindent}

\usepackage[shortlabels,inline]{enumitem}

\usepackage{tikz}
\usetikzlibrary{arrows,shapes}
\usepackage{subfigure}
\usetikzlibrary{patterns}
\usetikzlibrary{decorations.pathreplacing}
\usepackage{pgfplots}

\usepackage{graphicx}
\usepackage{mwe}

\usepackage{natbib}
\usepackage{commath}
\bibpunct[, ]{(}{)}{,}{a}{}{,}%

\usepackage{bm}
\usepackage{booktabs}
\usepackage{breqn}
\usepackage{tabularx}
\usepackage{array}
\usepackage{color}

\usepackage{endnotes}
\usepackage{enumerate}
\usepackage{rotating}
\usepackage{subfigure}
\usepackage{graphicx}

\usepackage{multirow}
\usepackage{multicol}
\usepackage{hhline}

\usepackage[english]{babel}

\usepackage{xcolor}
\usepackage{float}
\usepackage{graphicx}
\usepackage{soul}
\usepackage{thm-restate}
\usepackage{thmtools}
\usepackage{verbatim}

\allowdisplaybreaks
\newlist{compactitem}{itemize}{3}
\setlist[compactitem]{topsep=0pt,partopsep=0pt,itemsep=0pt,parsep=0pt}
\setlist[compactitem,1]{label=\textbullet}
\setlist[compactitem,2]{label=---}
\setlist[compactitem,3]{label=*}

\newlist{compactdesc}{description}{3}
\setlist[compactdesc]{topsep=0pt,partopsep=0pt,itemsep=0pt,parsep=0pt}

\newlist{compactenum}{enumerate}{3}
\setlist[compactenum]{topsep=0pt,partopsep=0pt,itemsep=0pt,parsep=0pt}
\setlist[compactenum,1]{label=\arabic*}
\setlist[compactenum,2]{label=\alph*}
\setlist[compactenum,3]{label=\roman*}

\newtheorem{thm}{Theorem}
\newtheorem{cor}{Corollary}
\newtheorem{lem}{Lemma}
\newtheorem{asm}{Condition}

\newtheorem{examp}{Example}

\newtheorem{prop}{Proposition}

\newtheorem{defn}{Definition}
\newtheorem{cond}{Condition}

\newtheorem{rem}{Remark}

\newtheorem*{theorem*}{Theorem}

\pgfplotsset{
	compat=1.8,
	tick label style = {font=\scriptsize},
	every axis label = {font=\scriptsize},
	legend style = {font=\scriptsize},
	label style = {font=\scriptsize},
	/pgfplots/xlabel near ticks/.style={
		/pgfplots/every axis x label/.style={
			at={(ticklabel cs:0.5)},anchor=near ticklabel
		}
	},
	/pgfplots/ylabel near ticks/.style={
		
		/pgfplots/every axis y label/.style={
			at={(ticklabel cs:0.5)},rotate=90,anchor=near ticklabel}
	}
	
}

\makeatletter
\renewcommand{\paragraph}{%
	\@startsection{paragraph}{4}%
	{\z@}{1ex \@plus 1ex \@minus .2ex}{-1em}%
	{\normalfont\normalsize\bfseries}%
}
\makeatother

\usepackage{natbib}
 \bibpunct[, ]{(}{)}{,}{a}{}{,}%
\usepackage{setspace}

\newcommand{\dashrule}[1][black]{%
  \color{#1}\rule[\dimexpr.5ex-.2pt]{4pt}{.4pt}\xleaders\hbox{\rule{4pt}{0pt}\rule[\dimexpr.5ex-.2pt]{4pt}{.4pt}}\hfill\kern0pt%
}

\makeatletter
\renewcommand\paragraph{\@startsection{paragraph}{4}{\z@}%
{1.15ex \@plus.3ex \@minus.2ex}%
	{-0.8em}%
	{\normalfont\normalsize\bfseries}}
\makeatother

\input{commands}

\begin{document}

\setlength{\abovedisplayskip}{6pt}
\setlength{\belowdisplayskip}{6pt}
\setlength{\abovedisplayshortskip}{0pt}
\setlength{\belowdisplayshortskip}{6pt}

\title{Blind Dynamic Resource Allocation in Closed Networks \\ via Mirror Backpressure}
%
\author{Yash Kanoria\thanks{Decision, Risk and Operations Division, Columbia Business School, Email: \texttt{ykanoria@columbia.edu}} \and Pengyu Qian\thanks{Krannert School of Management, Purdue University, Email: \texttt{pqian20@gmail.com}}.}
%
\date{}

\maketitle
%


\begin{abstract}
We study the problem of maximizing payoff generated over a period of time in a general class of closed queueing networks with a finite, fixed number of supply units which circulate in the system. Demand arrives stochastically, and serving a demand unit (customer) causes a supply unit to relocate from the ``origin'' to the ``destination'' of the customer. The key challenge is to manage the distribution of supply in the network. We consider general controls including customer entry control, pricing, and assignment. Motivating applications include shared transportation platforms and scrip systems.

Inspired by the mirror descent algorithm for optimization and the backpressure policy for network control, we introduce a rich family of \emph{Mirror Backpressure} (MBP) control policies. The MBP policies are simple and practical, and crucially do not need any statistical knowledge of the demand (customer) arrival rates (these rates are permitted to vary in time). Under mild conditions, we propose MBP policies that are provably near optimal. Specifically, our policies lose at most $O(\frac{K}{T}+\frac{1}{K} + \sqrt{\eta K})$ payoff per customer relative to the optimal policy that knows the demand arrival rates, where $K$ is the number of supply units, $T$ is the total number of customers over the time horizon, and {$\eta$ is the demand process' average rate of change per customer arrival.}
An adaptation of MBP is found to perform well in numerical experiments based on data from ride-hailing.
\end{abstract}



\section{Introduction}\label{sec:intro}
The control of complex systems with circulating resources such as shared transportation platforms and scrip systems has been heavily studied in recent years.
The hallmark of such systems is that serving a demand unit causes a (reusable) supply unit to be relocated.
Closed queueing networks (i.e., networks where a fixed number of supply units circulate in the system) provide a powerful abstraction for these applications \citetext{\citealp*{banerjee2016pricing}, \citealp{braverman2016empty}}. The key challenge is \emph{managing the distribution of supply} in the network.
A widely adopted approach for this problem is to solve the deterministic optimization problem that arises in the continuum limit (often called \emph{the static planning problem}), and show that the resulting control policy is near-optimal in a certain asymptotic regime.
However, this approach only works under the restrictive assumption that \emph{the system parameters (demand arrival rates) are precisely known}, and {most existing works assume \emph{time invariant} parameters}. 
%
%

In this paper, we relax \emph{both} assumptions. 
We propose a family of \emph{simple} control policies that are \emph{blind} in that they use {no} prior knowledge of demand arrival rates, and prove strong transient and steady state performance guarantees for these policies, for time-varying demand arrival rates. Strong  performance in simulations backs up our theoretical findings.

\smallskip
\textbf{Informal description of our model.}
Our main setting is one where the control levers include entry control and flexible assignment of resources, with time-varying demand arrival rates.
Later we allow dynamic pricing control 
and show that our machinery and guarantees extend seamlessly. For simplicity, we introduce here the special case of our main model with entry control only. 
We consider a closed queueing network that consists of a set of nodes (locations) $V$, and a fixed number $K$ of supply units that circulate in the system.
Demand units with different origin-destination pairs $(j,k)$ arrive stochastically over slotted time with some time-varying arrival rates which are unknown to the controller.
The controller dynamically decides whether to admit each incoming demand unit. 
Each control decision to admit a demand unit has two effects: it generates a certain \textit{payoff} $w_{jk}$ depending on the origin-destination pair of the demand unit, and it causes a supply unit to \textit{relocate} from the origin $j$ to the destination $k$ instantaneously.
The goal of the system is to maximize the collected payoff over a period of time.

Notably, the greedy policy, which admits a demand unit if a supply unit is available in its pick-up neighborhood, is generically far from optimal: even as $K \rightarrow \infty$, the optimality gap per demand unit of this policy is $\Omega(1)$ even in steady state; see Appendix~\ref{append:greedy}. The intuition is that some nodes have no available supply an $\Omega(1)$ fraction of the time in steady state under the greedy policy. 
Furthermore, if demand arrival rates are imperfectly known, any state independent policy \citep*[such as that of][]{banerjee2016pricing} generically suffers a steady state optimality gap per demand unit of $\Omega(1)$; see \citet*[][Proposition~4]{banerjee2018SMW}. 

\textbf{Our control policy.} The {system state} at time $t$ is the vector of queue lengths $\bq[t]=[q_1[t],\cdots,q_{|V|}[t]]^{\T}$, which sums up to the total supply $\mathbf{1}^{\T}\bq[t]=K$; we work with a normalized queue length vector $\nq$ satisfying $\mathbf{1}^{\T} \nq[t] = 1$.
Our proposed \emph{Mirror Backpressure} (MBP) policy makes entry control decisions according to the following simple rule: Admit a demand unit with origin  node $j$ and destination node $k$ if and only if the \emph{score} $w_{jk} + f(\bar{q}_j[t]) - f(\bar{q}_k[t]) \geq 0$ and $\bar{q}_j[t] > 0$. Here, $f(\cdot)\triangleq - \sqrt{m}\cdot \bar{q}_j^{-\frac{1}{2}}$ is a suitably chosen \emph{congestion function}, a monotone increasing function which causes the policy to be generous in allowing use of supply from long queues while protecting supply in near-empty queues. See Figures~\ref{fig:reduced_cost} and \ref{fig:mirror_map} for illustrations. Note that the MBP policy is agnostic to demand arrival rates.

\begin{figure}
	\centering
	\begin{minipage}{0.45\textwidth}
		\centering
		\includegraphics[width=\textwidth]{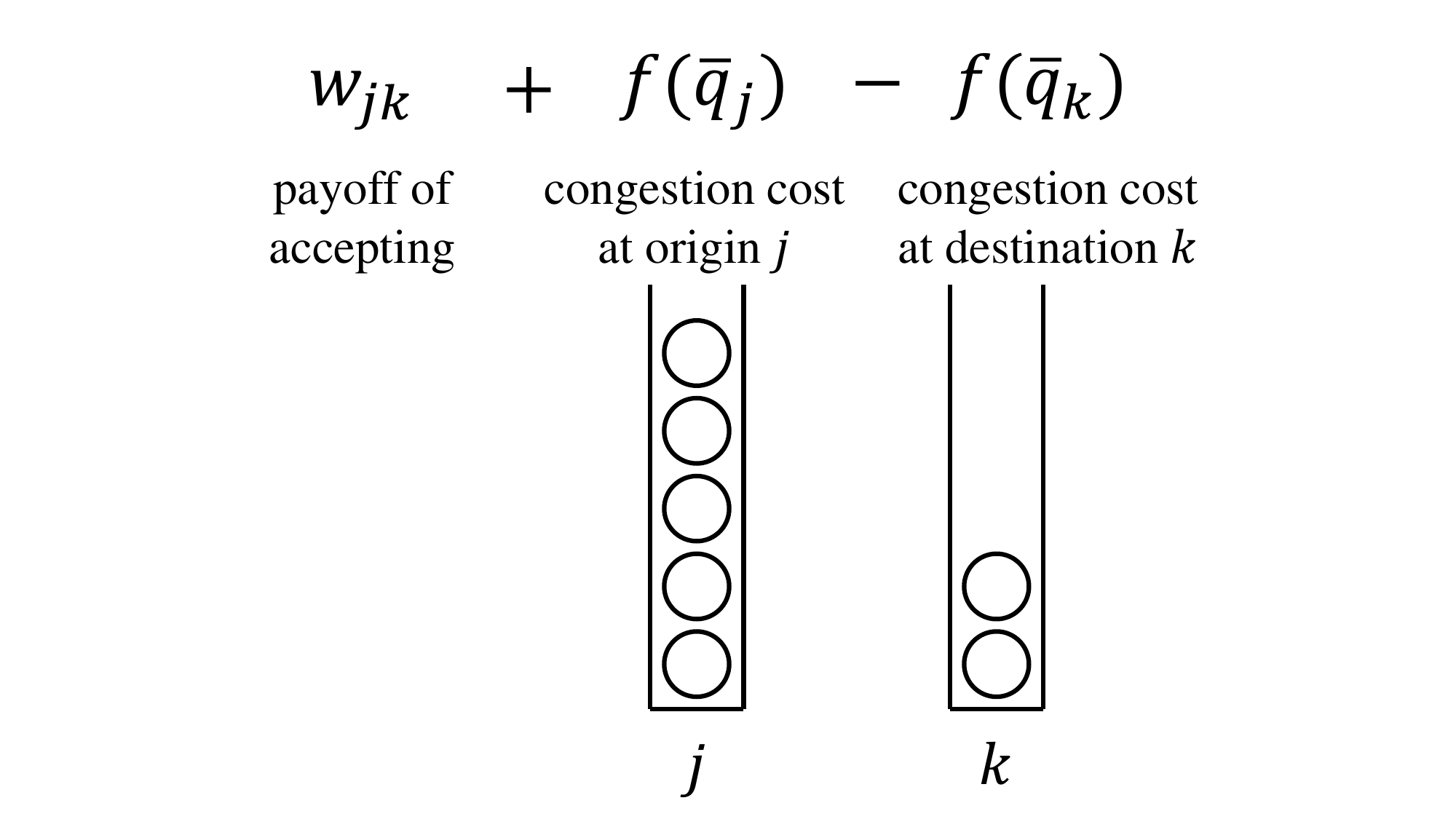} 
		\caption{Our MBP policy admits a demand with origin $j$ and destination $k$ if and only if the score illustrated above is non-nonnegative and $j$ has at least one supply unit.
		}
\label{fig:reduced_cost}
	\end{minipage}\hfill
	\begin{minipage}{0.45\textwidth}
		\centering
\resizebox{\textwidth}{4cm}{\includegraphics{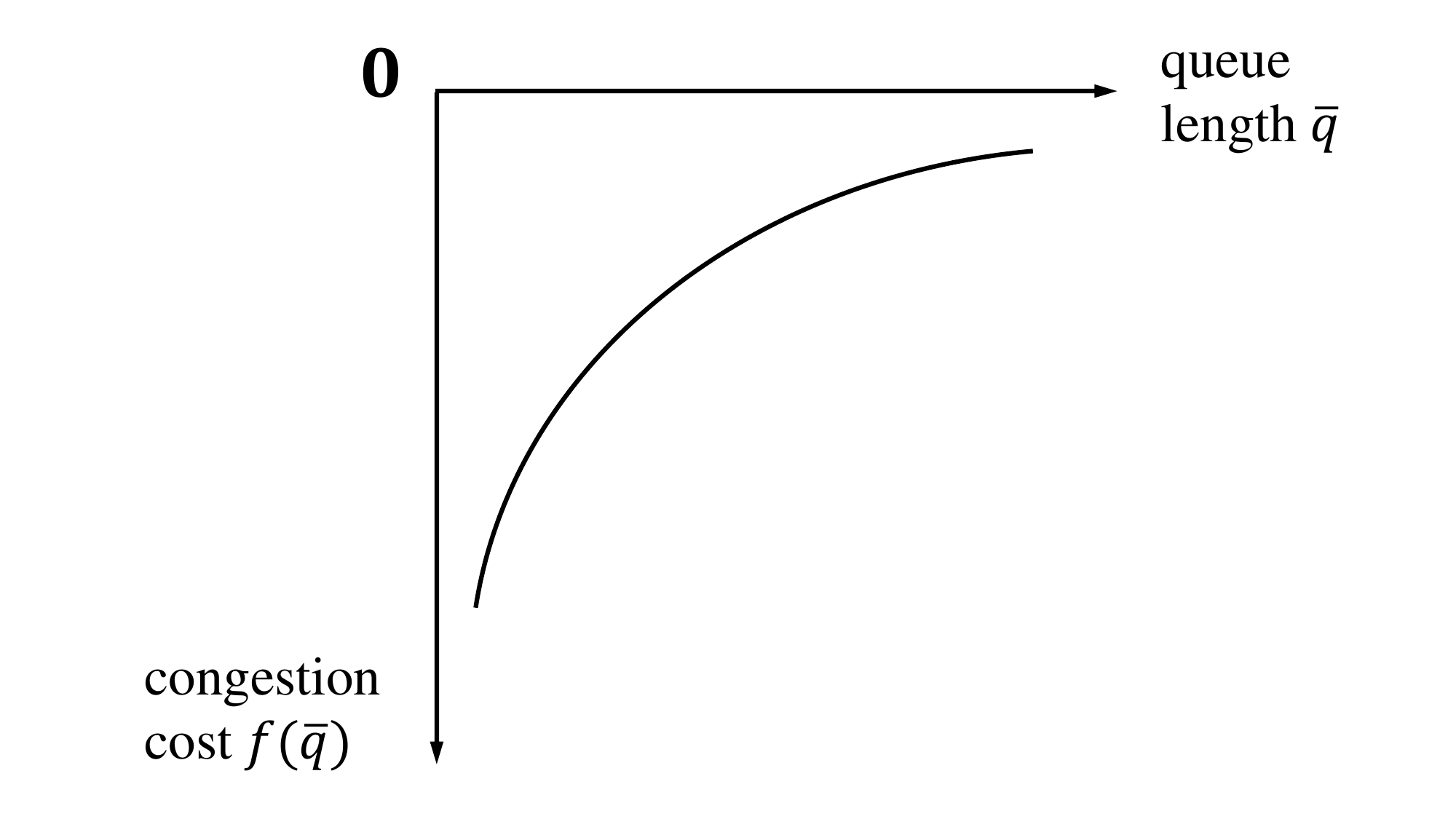}} 
		\caption{An example of a congestion function (a mapping from queue lengths to congestion costs) which aggressively protects supply units in near-empty queues. 
		}\label{fig:mirror_map}
	\end{minipage}
\end{figure}

\smallskip
\textbf{Performance guarantee.}
%
We show that, under a mild connectivity assumption on the network, the MBP policy is near optimal. Specifically, we show that our policies lose payoff (per demand unit) 
at most $O\left(\frac{K}{T} + \frac{1}{K} + \sqrt{\eta K}\right)$ relative to the optimal policy that knows the demand arrival rates, 
where $K$ is the number of supply units, $T$ is the number of demand units that arrive during the period of interest, and {$\eta$ is the demand process' average rate of change per customer arrival}. 
Our result is \textit{non-asymptotic}, i.e., our performance guarantee holds for finite $K$ and $T$,  and thus covers both transient and steady state performance.
In particular, for stationary demand arrivals, taking $T\to\infty$, we obtain a steady state optimality gap of $O(\frac{1}{K})$. 
Our bound further provides a guarantee on transient performance: the horizon-dependent term $K/T$ in our bound on optimality gap is small if the total number of arrivals $T$ over the horizon is large compared to the number of supply units $K$.
Our policies retain their good performance if the {demand arrival rate's average rate of change over $K$ periods}, i.e., the term $\eta K$, is small.
We find that our bound is invariant to system size in a relevant scaling regime (the large market regime) where the number of supply units $K$ increases proportionally to the demand arrival rates (see the discussion after Theorem~\ref{thm:main_regret}).
In this regime, $T^{\textup{real}}\triangleq T/K$ is an invariant as $K\to\infty$, which can be interpreted as the time horizon measured in physical time. Let $\zeta \triangleq \eta K$ be the average rate of change of demand with respect to physical time. 
We can rewrite our bound on the optimality gap as
	$
	O\left(\frac{1}{T^{\textup{real}}} +
	\frac{1}{K} + \sqrt{\zeta} \right ) \xrightarrow{K \to \infty} O\left(\frac{1}{T^{\textup{real}}} + \sqrt{\zeta}\right )
	$.


\textbf{Motivation for our control policy.} 
Our control approach is inspired by the celebrated backpressure methodology of \citet{tassiulas1992stability} for the control of queueing networks. Backpressure simply uses queue lengths as congestion costs (the shadow prices to the flow constraints; the flow constraint for each queue is that the inflow must be equal to the outflow in the long run), and chooses a control decision at each time which maximizes the myopic payoff inclusive of congestion costs. Concretely, in the special case where entry control is the only leverage, backpressure admits a demand if and only if the payoff of serving the demand plus the origin queue length exceeds the destination queue length.
This simple approach has been used very effectively in a range of 
settings arising in cloud computing, networking, etc.; see, e.g., \cite{georgiadis2006resource}.
Backpressure is provably near-optimal (in the large market limit) in many settings where payoffs accrue from serving jobs, because it has the property of executing dual stochastic gradient descent (SGD) on the controller's deterministic (continuum limit) optimization problem.
However, this property breaks down when the so-called ``no-underflow constraint'' binds, making it challenging to use backpressure in our setting. 


The \emph{no-underflow constraint} is that each decision to admit a demand unit needs to be backed by an available supply unit at the pick-up node of the demand. This constraint binds in our setting under backpressure because we model non-zero payoffs from serving a customer, as a result of which the congestion-adjusted myopic payoff can be positive even if the origin queue is empty (see Appendix~\ref{append:bp_fails} for a discussion). 
Moreover, several popular workarounds to this issue fail in our setting; see Section~\ref{subsec:lit-review}.
Our Mirror Backpressure (MBP) policy generalizes the celebrated backpressure (BP) policy. Whereas BP uses the queue lengths as congestion costs, MBP employs a flexibly chosen \emph{congestion function} $f(\cdot)$ to translate from queue lengths to congestion costs.
The mirror map can be flexibly chosen to fit the problem geometry arising from the no-underflow constraints. Roughly, we find better performance with congestion functions which are steep for small queue lengths, the intuition being that this makes MBP more aggressive in protecting the shortest queues (and hence preventing underflow). In case of finite buffers, we use congestion functions which moreover increase  steeply as the queue length approaches buffer capacity, to prevent buffer overflow (Section~\ref{subsec:extend-barrier}).

\textbf{Analytical approach.}
We show that MBP has the property that it executes dual \emph{stochastic mirror descent} \citep{nemirovsky1983problem,beck2003mirror} on the platform's continuum limit optimization problem, which generalizes the SGD property of backpressure.
We develop a general machinery to prove performance guarantees for MBP: 
We use the antiderivative of the chosen congestion function as the Lyapunov function in our analysis,
and adapt 
the Lyapunov drift method from the network control literature to obtain sharp bounds on the suboptimality caused by the no-underflow constraint. 
Our analysis exploits the structure of the platform's continuum limit optimization problem 
(see Section~\ref{sec:proof_sketch}). 
Our work fits into the broad literature on the control of stochastic processing networks \citep{harrison2000brownian}.  

\smallskip
\textbf{Applications.} Our models 
include a number of ingredients which are central in many applications.
We illustrate its versatility by discussing the application to shared transportation systems (Section~\ref{sec:shared-transportation}) and the application to scrip systems (Appendix~\ref{sec:scrip-systems}). These applications and the relevant settings in the paper are summarized in Table~\ref{tab:application_model}.

\begin{table}[htbp!]
	\small
	\centering
	\begin{tabular}{ccc}
		\toprule
		\emph{Application} & \emph{Control lever} & \emph{Corresponding setting in this paper} \\
		\midrule
		Ride-hailing in USA, Europe & Pricing \& Dispatch & Joint pricing-assignment\\
		Ride-hailing in China & Admission \& Dispatch & Joint entry-assignment\\
		%
		%
		Bike sharing & Reward points & Pricing (finite buffer queues) \\
		Scrip systems & Admission \& Provider selection   & Joint entry-assignment\\
		\bottomrule
	\end{tabular}
	\caption{Summary of applications of our model. 
	}
	\label{tab:application_model}
\end{table}

Shared transportation systems include ride-hailing and bike sharing systems. Here the nodes in our model correspond to geographical locations, while supply units and demand units correspond to vehicles and customers, respectively. Bike sharing systems dynamically incentivize certain trips using point systems to minimize out-of-bike and out-of-dock events caused by demand imbalance. Our pricing setting is relevant for the design of a dynamic incentive program for bike sharing; in particular, it allows for a limited number of docks. Ride-hailing platforms make dynamic decisions to optimize their objectives (e.g., revenue, welfare, etc.). For ride-hailing, our pricing-assignment model is relevant  in regions such as North America, and our entry-assignment control model is relevant in regions where dynamic pricing is undesirable like in China.
We perform simulations of ride-hailing and find that our MBP policy, suitably adapted to account for positive travel times, performs well (Section~\ref{subsec:numerics-ridehailing}).

	A scrip system is a nonmonetary trade economy where agents use scrips (tokens, coupons, artificial currency) to exchange services (because
monetary transfer is undesirable or impractical), e.g., for babysitting or kidney exchange.
	A key challenge in these markets is the design of the
admission-and-provider-selection rule:
	If an agent is running low on scrip balance, should they be allowed to request services? If yes, and if there are several possible providers for a trade, who should be selected as the service provider?
	In Appendix~\ref{sec:scrip-systems}, we show that a natural model of a scrip system is a special case of our entry-assignment control setting, yielding a near optimal admission-and-provider-selection control rule.

\subsection{Literature Review}
\label{subsec:lit-review}

\textbf{MaxWeight/backpressure policy.}
Backpressure \citep[also known as MaxWeight, see][]{tassiulas1992stability,georgiadis2006resource} are well-studied dynamic control policies in constrained queueing networks for workload minimization \citep{stolyar2004maxweight,dai2008asymptotic}, queue length minimization \citep{eryilmaz2012asymptotically} and utility maximization \citep{eryilmaz2007fair},  etc.
Attractive features of MaxWeight/backpressure policies include their simplicity and provably good performance, and that arrival/service rate information is not required beforehand. 

The main challenge in using backpressure in settings with payoffs is the no-underflow constraints, as described earlier. Several works make strong assumptions to ensure the constraint does not bind: For example, \cite{dai2005maximum} assume that the network satisfies a so-called Extreme Allocation Available (EAA) condition; \cite{stolyar2005maximizing} assumes that payoffs are generated only by the source nodes, which have infinite queue lengths. 	\cite{huang2011utility} consider networks where the no-underflow constraint does bind, but the payoffs are generated only by the output nodes.
{In our setting, payoffs are essential (there is value generated by serving
a customer) and can be generated by any node. Therefore, the no-underflow constraint binds, and none of the aforementioned assumptions hold for our network.}
Another workaround is a machinery that introduces \emph{virtual queues}; see, e.g., \cite{jiang2009stable}. The idea is to introduce a ``fake'' supply unit into the network each time the constraint binds, to preserve the SGD property of backpressure.
{However, in open queueing networks, these fake supply units eventually leave the system, and so have a small effect (under appropriate assumptions). In our closed network setting, these fake supply units, once created, never leave and so would build up in the system, irreparably damaging performance.}
\yka{Mention in a couple of sentences why DMW and PMW do not work in your setting. I prefer not to have a separate section discussing these side points, given that ultimately an approach from the literature, namely \cite{gupta2020interior}, is what we adapt.}

Most of this literature considers the open queueing networks setting, where packets/jobs enter and leave, and there is much less work on closed networks.
An exception is a recent paper on assignment control of closed networks by \citet*{banerjee2018SMW}, which shows the large deviations optimality of ``scaled'' MaxWeight policies. 
%


Similar to MBP, several works use nonlinear functions of queue lengths for decision making to improve upon the performance of Backpressure in a variety of  contexts. 
\cite{walton2015concave} proposes Concave Switching policies that generalize 
Backpressure to address a weakness of Backpressure in fixed route multihop networks, namely, that
the number of queues it needs to maintain grows rapidly in network size. 
\cite{neely2006super} uses exponential functions of queue length as congestion functions to achieve the optimal delay-utility tradeoff. \cite{gupta2020interior} use non-linear functions of the state variables in the context of online stochastic bin-packing to obtain distribution-oblivious algorithms with sublinear additive suboptimality. They also identify the connection to mirror descent as we do for MBP.
\yka{See https://pubsonline.informs.org/doi/suppl/10.1287/opre.2019.1914/suppl_file/opre.2019.1914.sm1.pdf EC.2.2, it's nearly a carbon copy of our section describing the connection to mirror descent. I think the motivation for using non-linearity is quite similar to our motivation too. Hence my edits and why I moved it to the non-linear queue lengths para.}

\textbf{Mirror Descent.} Mirror descent (MD) is a generalization of the gradient descent algorithm for optimization (\citealt{nemirovsky1983problem,beck2003mirror}). 
Recently, there have been several works that use online MD to solve other online decision-making problems, including the $k$-server problem (\citealt{bubeck2018k}), and various online packing and covering problems, e.g.,  \cite{gupta2016experts,agrawal2014fast}.  

\textbf{Applications: shared transportation, scrip systems.}
Most of the ride-hailing literature studied controls that require the exact knowledge of system parameters:
\citet{ozkan2016dynamic} studied payoff maximizing assignment control in an open queueing network model,
\citet{braverman2016empty} derived the optimal state independent routing policy that sends empty vehicles to under-supplied locations,
\citet*{banerjee2016pricing} adopted the Gordon-Newell closed queueing network model and considered various controls that maximize throughput, welfare or revenue.
\citet{balseiro2019dynamic} considered a dynamic programming based approach for dynamic pricing for a specific network of star structure. (\citet{ma2019spatio} studied the somewhat different issue of ensuring that drivers have the incentive to accept dispatches by setting prices which are sufficiently smooth in space and time, in a model with no demand stochasticity.)
\citet*{banerjee2018SMW} which assumes a near balance condition on demands and equal pickup costs may be the only paper in this space that does not require knowledge of system parameters.
%
Comparing with \citet*{banerjee2016pricing} which obtains a steady state optimality gap of $O(\frac{1}{K})$ (in the absence of travel times) assuming \emph{perfect} knowledge of demand arrival rates which are assumed to be \emph{stationary}, our control policy achieves the same steady state optimality gap with \emph{no} knowledge of demand arrival rates, and further achieves a \emph{transient} optimality gap under \emph{time-varying} demand arrival rates of $O(\frac{K}{T} + \frac{1}{K} +\sqrt{\eta K})$ for a finite number of arrivals $T$ and average changes of up to $\eta$ per arrival in demand arrival rates. 
Some of these papers are able to formally handle travel delays: \citet{braverman2016empty}, \citet{banerjee2016pricing}, and \citet*{banerjee2018SMW} prove theoretical results for the setting with i.i.d. geometric/exponential travel delays; \citet{ma2019spatio} consider deterministic travel delays. On the other hand, \citet{balseiro2019dynamic} ignores travel delays in their theory and later heuristically adapt their policy to accommodate travel delay (the present paper follows a similar approach). 
\cite{ozkan2016dynamic} is the only paper among these  which (like the present paper) allows time-varying demand. 

Our model can be applied to the design of dynamic incentive programs for bike sharing systems \citep{chung2018bike} and service provider rules for scrip systems \citep{johnson2014analyzing,agarwal2019market}.
For example, the ``minimum scrip selection rule'' proposed in \cite{johnson2014analyzing} is a special case of our policy, and our methodology leads to control rules in much more general settings as described in Appendix~\ref{sec:scrip-systems}. 

\yka{commented a bunch of stuff about scaling of previous results.}

\textbf{Other related work.}
A related stream of research studies online stochastic bipartite matching, see, e.g., \cite{caldentey2009fcfs,adan2012loss,buyic2015approximate,mairesse2016stability};
the main difference between their setting and ours is that we study a \emph{closed} system where supply units never enter or leave the system. Network revenue management is a classical set of (open network) dynamic resource allocation problems, e.g., see \cite{gallego1994optimal,talluri2006theory}, and recent works, e.g., \cite{bumpensanti2018re}.
\cite{jordan1995principles,desir2016sparse,shi2015process} and others study how process flexibility can facilitate improved performance, analogous to our use
of assignment control to maximize payoff (when all pickup costs are equal), but the focus there is more on network design than on control policies.
\yka{cite process flexibility, no? shi, wei, zhong could be our reviewers. same for weina wang.}
Again, this is an open network setting in that each supply unit can be used only once. 

\subsection{Organization of the Paper}
The remainder of our paper is organized as follows.
Section \ref{sec:model} presents our main model of joint entry-assignment control with time-varying demand arrival rates, and the platform objective.
Section \ref{sec:mbp} introduces the Mirror Backpressure policy and presents our main theoretical result, i.e., a performance guarantee for the MBP policies.
Section \ref{sec:proof_sketch} outlines the proof of our main result.
In Section \ref{sec:extension}, we provide MBP policies for the joint pricing-assignment control setting, demonstrating the versatility of our approach.
In Sections~\ref{sec:shared-transportation} we discuss the applications to shared transportation systems.

\textbf{Notation.} All vectors are column vectors if not specified otherwise.
The transpose of vector or matrix $\bx$ is denoted as $\bx^{\T}$.
We use $\boe_i$ to denote the $i$-th unit column vector with the $i$-th coordinate being $1$ and all other coordinates being $0$, and $\mathbf{1}$ ($\bzero$) to denote the all $1$ ($0$) column vector, where the dimension of the vector will be indicated in the superscript when it is not clear from the context.

\section{The Model: Joint Entry-Assignment Control}\label{sec:model}
In this section, we formally define our model of joint entry-assignment control in closed queueing networks.
%
We consider a finite-state Markov chain model with slotted time $t= 0, 1, 2, \dots$, where a fixed number (denoted by $K$) of identical \textit{supply units} circulate among a set of 
\textit{nodes} $V$ (locations), with $m \triangleq |V|> 1$. In our model, $t$ will capture the number of demand units (customers) who have arrived so far.

\smallskip
\textbf{Queues (system state).}
At each node $j \in V$, there is an infinite-buffer queue of supply units.
(Section~\ref{subsec:extend-barrier} shows how to seamlessly incorporate finite-buffer queues.) 
The \emph{system state} is the vector of queue lengths at time $t$, which we denote by $\bq[t]=[q_1[t],\cdots,q_m[t]]^{\T}$.
Denote the state space of queue lengths by $\Omega_{K} \triangleq \{\bq:\bq\in\mathbb{Z}^{m}_+,\mathbf{1}^{\T}\bq=K\}$, and the normalized state space by $\Omega\triangleq\{\bq:\bq\in\mathbb{R}^{m}_+,\mathbf{1}^{\T}\bq=1\}$.

\smallskip
\textbf{Demand types and time-varying arrival process.}
We assume exactly one demand unit (customer) arrives at each period $t$, and denote her abstract \emph{type} by $\tau[t] \in \cT$, and the type for the demand unit is drawn from distribution $\bphi^t = (\phi_\tau^t)_{\tau \in \cT}$, independent of demands in earlier periods.\footnote{
{Analyses of i.i.d. unit demand arrivals have been shown to generalize easily to more general arrival processes, e.g., Markovian arrivals with bounded demand units per period as in \cite{huang2009delay}, though at a significant notational burden. 
Given the aforementioned precedent, we reason that the cost of carrying the reader through this generalization exceeds the benefit of doing so, and assume i.i.d. unit demand arrivals throughout the paper.
}
}
Note that the demand arrival rate (i.e., type distribution) can be time-varying.
Importantly, the system can observe the type of the arriving demand at the beginning of each time slot, but \emph{the probabilities (arrival rates) $\bphi^t$ are not known.}
Thus we substantially relax the assumption in previous works that the system has exact knowledge of demand arrival rates
\citetext{\citealp{ozkan2016dynamic}, \citealp*{banerjee2016pricing}, \citealp{balseiro2019dynamic}}.

Each demand type $\tau\in \cT$ has a pick-up neighborhood $\cP(\tau)\subset V, \cP(\tau) \neq \emptyset$ and drop-off neighborhood $\cD(\tau)\subset V, \cD(\tau) \neq \emptyset$.
The sets $(\cP(\tau))_{\tau \in V}$ and $(\cD(\tau))_{\tau \in V}$ are model primitives. (In shared transportation systems, each demand type $\tau$ may correspond to an (origin, destination) pair in $V^2$, with $\cP(\tau)$ being nodes close to the origin and $\cD(\tau)$ being nodes close to the destination.) 

\smallskip
\textbf{Temporal uncertainty of demand arrival rates.}
We define the following notion of \emph{$\eta$-slowly varying demand} that characterizes the average amount of change of demand arrival rates over a finite time horizon.
\begin{defn}
	\label{defn:variation-budget}
	We say that demand arrival rates vary \emph{$\eta$-slowly} over a finite horizon $T$ if:
	\begin{align*}
		\frac{1}{T-1}\sum_{t=1}^{T-1}\lVert \bphi^{t+1} - \bphi^t\rVert_1 \leq \eta\, .
	\end{align*}
\end{defn}

\smallskip
\textbf{Control and payoff.}
At time $t$, after observing the demand type $\tau[t]=\tau$, the system makes a decision
\begin{align}\label{eq:jea-control}
	(x_{j\tau k}[t])_{j\in\cP(\tau),k\in\cD(\tau)}\in \{0,1\}^{|\cP(\tau)|\cdot|\cD(\tau)|}\quad
	\textup{such that}
	\quad
	\sum_{j\in\cP(\tau),k\in\cD(\tau)}x_{j\tau k}[t]\leq 1\, .
\end{align}
Here $x_{j\tau k}[t]=1$ stands for the platform choosing pick-up node $j\in \mathcal{P}(\tau)$ and drop-off node $k \in \mathcal{D}(\tau)$, causing a supply unit to be relocated from $j$ to $k$.
The constraint in \eqref{eq:jea-control} captures that each demand unit is either served by one supply unit, or not served.
With $x_{j\tau k}[t]=1$, the system collects payoff $v[t] = w_{j\tau k}$. 
Without loss of generality, we assume the scaling
$
\max_{\tau\in \cT, j \in \cP(\tau), k \in \cD(\tau)} |w_{j\tau k}|\  = \, 1\,  . $
Because the queue lengths are non-negative by definition, we require the following \emph{no-underflow constraint} to be met at any $t$:
\begin{align}\label{eq:no_underflow}
	x_{j\tau k}[t] = 0\quad \textup{if} \quad
	q_{j}[t]=0\, .
\end{align}
As a convention, let $x_{j\tau'k}=0$ if $\tau'\neq \tau$.\yka{Revisit later.}

%
A \emph{feasible policy} specifies, for each time $t \in \{1, 2, \dots\}$, a mapping from the history so far of demand types $\big(\tau[t']\big)_{t' \leq t}$ and states $(\bq[t'])_{t' \leq t}$ to a decision $x_{j\tau k}[t]\in\{0,1\}^{|\cP(\tau)|\cdot|\cD(\tau)|}$ satisfying \eqref{eq:no_underflow}, where $\tau=\tau[t]$ as above.
We allow $x_{j\tau k}[t]$ to be randomized, although our proposed policies will be deterministic.
The set of feasible policies is denoted by $\mathcal{U}$.

\smallskip
\textbf{System dynamics and objective.}
The dynamics of system state $\bq[t]$ is as follows:
\begin{align}\label{eq:system-dynamics-jea}
	\bq[t+1] = \bq[t] + \sum_{\tau\in \cT, j \in \cP(\tau), k \in \cD(\tau)}
	(-\boe_{j} + \boe_{k}) x_{j\tau k}[t]\, .
\end{align}
We use $v^{\pi}[t]$ to denote the payoff collected at time $t$ under control policy $\pi$.
Let $W^{\pi}_T$ denote the average payoff per period (i.e., per customer) collected by policy $\pi$ in the first $T$ periods, and let $W^*_T$ denote the optimal payoff per period in the first $T$ periods over all admissible policies.
Mathematically, they are defined respectively as:\footnote{Note that in \eqref{eq:obj_finite} the expectations are taken over the randomness in arrivals and (possibly) control decisions, and that the supremum is well-defined because the payoffs are bounded from above.}
\begin{align}\label{eq:obj_finite}
W^{\pi}_T \triangleq \;
\min_{\bq\in \Omega_K}\;
\frac{1}{T}\sum_{t=1}^{T}\mathbb{E}[v^{\pi}[t]|\bq[0]=\bq]\, ,
\qquad
W^*_T \triangleq\; \sup_{\pi \in \mathcal{U}}\;\max_{\bq\in \Omega_K}\;
\frac{1}{T}\sum_{t=1}^{T}\mathbb{E}[v^{\pi}[t]|\bq[0]=\bq]\, .
\end{align}
Define the infinite-horizon per period payoff $W^{\pi}$ collected by policy $\pi$ and the optimal per period payoff over all admissible policies $W^{*}$ respectively as:
\begin{align}\label{eq:obj_inf}
W^{\pi}\triangleq\;  \liminf_{T\to\infty}\; W^{\pi}_T\, ,
\qquad
W^* \triangleq \limsup_{T\to\infty}\; W^{*}_T\, .
\end{align}
We measure the performance of a control policy $\pi$ by its finite- and infinite-horizon per-customer \emph{optimality gap} (``loss''), defined respectively as:
\begin{align}
	L^\pi_T = W^*_T - W^{\pi}_T
	\qquad
	\textup{and}
	\qquad
	L^\pi = W^*-W^{\pi}\, .
\label{eq:optimality_gaps_defn}
\end{align}
Note that we consider the worst-case initial system state when evaluating a given policy, and the best initial state for the optimal benchmark; see \eqref{eq:obj_finite}.
Such a definition of optimality gap provides a conservative bound on policy performance and avoids the (unilluminating) discussion of the dependence of performance on initial state.

We make the following mild connectivity assumption on the primitives $(\{\bphi^t\}_{t\leq T},\cP,\cD)$.
\begin{cond}[Strong connectivity of ($\{\bphi^t\}_{t\leq T},\cP,\cD$)]\label{cond:str_connect_jea}
	For any demand arrival rates $\bphi$, define the connectedness of triple ($\bphi$,$\cP$,$\cD$) as
	\begin{align}\label{eq:strong-connectivity-jea}
		\alpha(\bphi,\cP,\cD)
		\triangleq
		\min_{S\subsetneq V,S\neq\emptyset}
		\sum_{\tau\in \cP^{-1}(S)\cap\cD^{-1}(V\backslash S)}
		\phi_{\tau}\, .
	\end{align}
	Here $\mathcal{P}^{-1}(S) \triangleq \{ \tau \in \cT : \mathcal{P}(\tau)\cap S \neq \emptyset\}$ is the set of demand types for which nodes $S$ can serve as a pickup node; and $\mathcal{D}^{-1}(\cdot)$ is defined similarly.
	We assume that for any $1\leq t \leq T$, $(\bphi^t,\cP,\cD)$ is strongly connected, namely, that $\alpha(\bphi^t,\cP,\cD)>0$.
\end{cond}
Note that the strong connectivity of $(\bphi,\cP,\cD)$ is equivalent of requiring that with (stationary) demand arrival rates $\bphi$, for every ordered pair of nodes $(j,k)$, there is a sequence of demand types with positive arrival rates and corresponding pick-up and drop-off nodes that would take a supply unit from $j$ eventually to $k$.

We conclude this section with 
the observation that the main assumption of \citet*{banerjee2018SMW} is automatically violated in our setting.


\begin{rem}
\label{rem:CRP-is-violated-and-greedy-is-suboptimal}
The complete resource pooling (CRP) condition imposed in \citet*[][Assumption 3]{banerjee2018SMW} is automatically violated in the following subclass of our model. 
Consider our setup including Condition~\ref{cond:str_connect_jea}, where each demand type $\tau=(i,j)$ corresponds to an origin-destination pair, and that $\cP(i,j)=\{i\}$, $\cD(i,j)=\{j\}$. The CRP condition can be stated as follows: for each subset of nodes $S \subsetneq V, S \neq \emptyset$, the ``net demand'' $\mu_{S} \triangleq \sum_{i \in S} \sum_{j \in V \backslash S} \phi_{ij}$ is less than the ``net supply'' $\lambda_{S} \triangleq \sum_{j \in V \backslash S} \sum_{i \in S} \phi_{ji}$, i.e., $\mu_{S}< \lambda_{S}$.
Clearly, any demand arrival rates $\bphi$ violate CRP, since if $\mu_{S}< \lambda_{S}$ for some $S \subsetneq V, S \neq \emptyset$ then this means that $\mu_{V\backslash S} > \lambda_{V \backslash S}$ (given that $\mu_{V\backslash S} = \lambda_S$ and $\lambda_{V\backslash S} = \mu_S$ by definition), i.e., CRP is violated.
\end{rem}

\section{The MBP Policies and Main Result}
\label{sec:mbp}
In this section, we propose a family of blind online control policies, and state our main result for these policies, which provides a strong transient and steady state performance guarantee for finite systems.

\subsection{The Mirror Backpressure Policies}
We propose a family of online control policies which we call \emph{Mirror Backpressure} (MBP) policies.
Each member of the MBP family is specified by a mapping of normalized queue lengths (which we will define below) $\mathbf{f}(\nq): \Omega \to \mathbb{R}^m$, where $\mathbf{f}(\nq)\triangleq [f(\bar{q}_1),\cdots,f(\bar{q}_m)]^{\T}$ and $f$ is a monotone increasing function.\footnote{The methodology we will propose will seamlessly accommodate general mappings $\bof(\cdot)$ such that $\bof=\nabla F$ where $F(\cdot):\Omega\to\mathbb{R}$ is a strongly convex function, a special case of which is $\mathbf{f}(\nq)\triangleq [f_1(\bar{q}_1),\cdots,f_m(\bar{q}_m)]^{\T}$ for some  monotone increasing $(f_j)$s. Here it suffices to consider a single congestion function $f(\cdot)$, whereas in Section~\ref{subsec:extend-barrier} we will employ queue-specific congestion functions $f_j(\cdot)$.} %
We will refer to ${f}(\cdot)$ as the \textit{congestion function}, which maps each (normalized) queue length to a congestion cost at that node, based on which MBP will make its decisions.

In this section we will state our main result for the congestion function
\begin{align}
	\label{eq:inv_sqrt_mirror_map}
	f(\bar{q}_j) \triangleq -\sqrt{m}\cdot\bar{q}_j^{-\frac{1}{2}}\, , 
\end{align}
and postpone the results for other choices of congestion functions to Appendix \ref{append:proof-main-regret} (see also Remark~\ref{rem:general_mirror_map}).
%
We will later clarify the precise role of the congestion function and show that it is related to the mirror map in mirror descent \citep{beck2003mirror}.
Similar to the design of effective mirror descent algorithms, the choice of congestion function should depend on the constraints of the setting, leading to an interesting interplay between problem geometry and policy design. For instance, we employ a different congestion function for the setting in Section \ref{subsec:extend-barrier} where there are additional buffer capacity constraints.

For technical reasons, we need to keep $\nq$ in the \textit{interior} of the normalized state space $\Omega$, i.e., we need to ensure that all normalized queue lengths remain positive.
This is achieved by defining\yka{Is this really a \textbf{re}definition, or is it the first time that $\nq$ is defined.} the normalized queue lengths $\nq$ as
\begin{align}\label{eq:normalized_queue}
	\bar{q}_i\triangleq \frac{q_i + \delta_K}{\tilde{K}}
	\quad
	\textup{for}
	\quad
	\delta_K\triangleq \sqrt{K} \quad \textup{and} \quad \tK \triangleq K + m\delta_K\, .
\end{align}
\yka{Replaced $\delta^K$ with $\delta_K$.} Note that this definition leads to $\mathbf{1}^{\T}\nq = 1$ and therefore $\nq \in \Omega$.

Our proposed MBP policy for the joint entry-assignment control problem is given in Algorithm \ref{alg:mirror_bp}.
MBP serves a demand of type $\tau$ using a supply unit at $j^*$ and relocate it to $k^*$ if and only if
\begin{align}\label{eq:bp_score}
	(j^*,k^*)=\textup{argmax}_{j\in \cP(\tau),k\in\cD(\tau)}w_{j\tau k} + f(\bar{q}_j) - f(\bar{q}_k)\, ,
\end{align}
\emph{and} that $w_{j^*\tau k^*} + f(\bar{q}_{j^*}) - f(\bar{q}_{k^*})$ is nonnegative, \emph{and} the origin node $j^*$ has at least one supply unit (see Figure \ref{fig:reduced_cost} for illustration of the score in Section \ref{sec:intro}).
The score in \eqref{eq:bp_score} is nonnegative if and only if the payoff $w_{j\tau k}$ of serving the demand outweighs the difference of congestion costs (given by $f(\bar{q}_k)$ and $f(\bar{q}_j)$) between the dropoff node $k$ and the pickup node $j$.
Roughly speaking, MBP is more willing to take a supply unit from a long queue and add it to a short queue, than vice versa; {see Figures \ref{fig:reduced_cost} and \ref{fig:mirror_map} in Section \ref{sec:intro}}.
The policy is not only completely blind, but also semi-local, i.e., it only uses the queue lengths at the origin and destination.
%
Note that the congestion cost \eqref{eq:inv_sqrt_mirror_map} increases with queue length (as required), and furthermore decreases sharply as queue length approaches zero.
%
Observe that such a choice of congestion function makes MBP very reluctant to take supply units from short queues and helps to enforce the no-underflow constraint~\eqref{eq:no_underflow}. 

\smallskip
\begin{algorithm}[H]\label{alg:mirror_bp}
	\SetAlgoLined
	{At the start of period $t$, the system observes demand type $\tau[t]=\tau$.}
	{$(j^*,k^*)\leftarrow \textup{argmax}_{j\in\cP(\tau),k\in\cD(\tau)} w_{j\tau k} + f(\bar{q}_j[t]) - f(\bar{q}_k[t])$}
	
	\eIf{$ w_{j^*\tau k^*} + f(\bar{q}_{j^*}[t]) - f(\bar{q}_{k^*}[t])\geq 0$ \textup{\textbf{and}} $q_{j^*}[t] > 0$}{
		
		$x_{j^* \tau k^*}[t] \leftarrow 1$, i.e., serve the incoming demand using a supply unit from $j^*$ and relocate it to $k^*$ \;
	}{
		$x_{j^*\tau k^*}[t] \leftarrow 0$, 
		i.e., drop the incoming demand\;
	}
	
	The queue lengths update as $\nq[t+1] = \nq[t] - \frac{1}{\tK} 
	x_{j^*\tau k^*}[t](\boe_{j^*} - \boe_{k^*})$.
	\caption{Mirror Backpressure (MBP) Policy for Joint Entry-Assignment Control}
\end{algorithm}

\subsection{Performance Guarantee for MBP Policies}\label{sec:main-result}
We now formally state the main performance guarantee of our paper for the joint entry-assignment control model introduced in Section \ref{sec:model}.
We will outline the proof in Section \ref{sec:proof_sketch}, and extend the result to the dynamic pricing setting in Section \ref{sec:extension}.

\pqa{Might need to reformulate Theorem 1 and Remark 1 to accommodate more congestion functions.}
\begin{thm}\label{thm:main_regret}
	Consider a set of $m$ nodes and any sequence of demand arrival rates $\{\bphi^t\}_{t\leq T}$ that satisfy Condition \ref{cond:str_connect_jea} and vary $\eta$-slowly (Definition \ref{defn:variation-budget}).
	Define $\alpha_{\min}\triangleq\min_{1\leq t\leq T}\alpha(\bphi^t)>0$.
	Then there exists $K_1 = \textup{poly}\left(m,\frac{1}{\alpha_{\min}}\right)$, and a universal constant $C< \infty$, such that the following holds.\footnote{Here ``poly'' indicates a polynomial. The constant $C$ is universal in the sense that it does not depend on $K$, $m$ or $\alpha_{\min}$.}
    For the congestion function $f(\cdot)$ defined in \eqref{eq:inv_sqrt_mirror_map},
	for any  $K\geq K_1$, the following finite-horizon guarantee holds for Algorithm~\ref{alg:mirror_bp}
	%
	\begin{align*}
	L^{\textup{MBP}}_T\leq
	M_1 \left(\frac{K}{T} +\sqrt{\eta K} \right)
	+
	M_2	\frac{1}{K}\, ,
	 \qquad \textup{for} \ M_1 \triangleq Cm \ \textup{and} \ M_2 \triangleq Cm^2 \,.
	\end{align*}
\end{thm}
\begin{cor}\label{cor:stationary-main-result}
	When the demand arrivals are stationary ($\eta=0$), for any $K\geq K_1$, the following infinite-horizon guarantee holds for Algorithm~\ref{alg:mirror_bp}
	%
	\begin{align*}
		L^{\textup{MBP}} \leq
		M_2
		\frac{1}{K}\, , \qquad \textup{for} \ M_2 = Cm^2 \, .
	\end{align*}
\end{cor}
\begin{rem}\label{rem:general_mirror_map}
	In Section \ref{sec:extension} we obtain results similar to Theorem \ref{thm:main_regret} for the dynamic pricing setting (Theorem \ref{thm:main_regret_jpa}).
	In Appendix \ref{append:proof-main-regret} (Theorem~\ref{thm:general_mirror_map}), we generalize Theorem \ref{thm:main_regret} 
by showing similar performance guarantees for a whole class of congestion functions that satisfy certain growth conditions.
Informally, the congestion function needs to be steep enough near zero to protect the nodes from being drained of supply units.
	%
\end{rem}

There are several attractive features of the performance guarantee provided by Theorem~\ref{thm:main_regret} for the simple and practically appealing Mirror Backpressure policy:
\begin{compactenum}[(1),wide,labelwidth=!,labelindent=0pt]
	\item {\bf The policy is completely blind.}
	In practice, the platform operator at best has access to an imperfect estimate of the demand arrival rates $\{\bphi^t\}$, so it is a very attractive feature of the policy that it does not need any estimate of $\{\bphi^t\}$ whatsoever.
	It is worth noting that the consequent bound of $O\left(\frac{1}{K}\right)$ on the steady state optimality gap remarkably matches that provided by \citet*{banerjee2016pricing} even though MBP requires \emph{no} knowledge of $\bphi$, whereas the policy of  \citet*{banerjee2016pricing} requires \emph{exact} knowledge of $\bphi$. (However, our constant is quadratic in the number of nodes $m$, whereas the constant in the other paper is linear in  $m$.)
As shown in \citet*[][Proposition~4]{banerjee2018SMW}, if the estimate of demand arrival rates is imperfect, any state independent policy \citep*[such as that of][]{banerjee2016pricing} generically suffers a long run (steady state) per customer optimality gap of $\Omega(1)$ (as $K \rightarrow \infty$).
	\item {\bf Guarantee on transient performance.}
	In contrast with \citet*{banerjee2016pricing} which provides only a steady state bound for finite $K$, we are able to provide a performance guarantee for finite horizon and finite (large enough) $K$. The horizon-dependent term $K/T$ in our bound on optimality gap is small if the total number of arrivals $T$ is large compared to the number of supply units $K$.
	
	
	It is worth noting that our bound \emph{does not} deteriorate as the system size increases in the ``large market regime'', where the number of supply units $K$ increases proportionally to the demand arrival rates
\citep[this regime is natural in ride-hailing settings, taking the trip duration to be of order $1$ in physical time, and where a non-trivial fraction of cars are busy at any time, see, e.g.,][]{braverman2016empty}.
	Let $T^{\textup{real}}$ denote the horizon in physical time.
	As $K$ increases in the large market regime, the primitive $\bphi$ remains unchanged, while
	$T = \Theta(K\cdot T^{\textup{real}})$ since there are $\Theta(K)$ arrivals per unit of physical time, and $\zeta \triangleq \eta K$ is average rate of change of $\bphi$ with respect to physical time
	Hence, we can rewrite our performance guarantee as
	\begin{align*}
	W^*_T-W^{\textup{MBP}}_T
	\leq M\left(\frac{1}{T^{\textup{real}}} +
	\frac{1}{K} + \sqrt{\zeta} \right) \xrightarrow{K \to \infty} \frac{M}{T^{\textup{real}}} + \sqrt{\zeta} \, .
	\end{align*}

\item {\bf Guarantee for time-varying arrivals.} 
Our bound shows that MBP is near-optimal when the demand's average rate of change is small ($\eta K = o(1)$), and that the performance guarantee of MBP degrades gracefully as $\eta K$ increases.
We note that if the demand arrival rates remain stationary for blocks of time, 
e.g., the first half of the horizon has one stationary arrival rate matrix and the second half of the horizon has another stationary arrival rate matrix, then applying Corollary \ref{cor:stationary-main-result} 
to each contiguous block of time with stationary demand could yield a
better guarantee than directly applying Theorem \ref{thm:main_regret} to the entire horizon.

	\item {\bf Flexibility in the choice of congestion function.}
	Because of the richness of the class of congestion functions covered in Appendix \ref{append:proof-main-regret} which generalizes Theorem \ref{thm:main_regret}, the system controller now has the additional flexibility to choose a suitable congestion function $f(\cdot)$.
	%
	From a practical perspective, this flexibility can allow significant performance gains to be unlocked by making an appropriate choice of $f(\cdot)$, as evidenced by our numerical experiments in Section~\ref{subsec:numerics-ridehailing}.
	%
\end{compactenum}

\section{Proof of Theorem \ref{thm:main_regret}}
\label{sec:proof_sketch}
In this section we provide the key propositions and lemmas that lead to a proof of Theorem~\ref{thm:main_regret}.
Our analysis generalizes and refines the so-called \emph{Lyapunov drift method} in the network control literature \citep[see, e.g.,][]{neely2010stochastic}.

We first define a sequence of deterministic optimization problems which arise in the continuum limit: the \emph{static planning problem} ({SPP}) \citep[see, e.g.,][]{harrison2000brownian,dai2005maximum}, whose values we use to upper bound the optimal finite (and infinite) horizon per period $W^*_T$ (and $W^*$) defined in \eqref{eq:obj_finite} and \eqref{eq:obj_inf}.
The SPP  is a linear program (LP) defined for any demand arrival rates $\bphi$:
\begin{align}
\hspace{-0.2cm}\textup{SPP($\bphi$)}:\,\textup{maximize}_{\bx} \; &
\sum_{\tau\in \cT, j\in\mathcal{P}(\tau),k\in\mathcal{D}(\tau)}\ w_{j\tau k} \cdot \phi_{\tau} \cdot  x_{j\tau k}
\label{eq:fluid_obj}\\
\textup{s.t.}\; &
\sum_{\tau\in \cT, j\in\mathcal{P}(\tau),k\in\mathcal{D}(\tau)}\ \phi_{\tau}\cdot x_{j\tau k}(\boe_j - \boe_k) =\bzero
\hspace{2.8cm} \textup{(flow balance)} \label{eq:fluid_flow_bal}\\
&\sum_{j\in \mathcal{P}(\tau), k\in \mathcal{D}(\tau)}x_{j\tau k}\leq 1, \ x_{j\tau k} \geq 0\, ,  \, \forall j,k \in V,\ \tau\in\cT .
\hspace{0.1cm} \textup{(demand constraint)} \label{eq:fluid_dmd_constr}
\end{align}
One interprets $x_{j\tau k}$ as the fraction of type $\tau$ demand which is served by pickup location $j$ and dropoff location $k$, and the objective \eqref{eq:fluid_obj} as the rate at which payoff is generated under the fractions $\bx$.
In the SPP \eqref{eq:fluid_obj}-\eqref{eq:fluid_dmd_constr}, one maximizes the rate of payoff generation subject to the requirement that the average inflow of supply units to each node in $V$ must equal the outflow (constraint \eqref{eq:fluid_flow_bal}), and that $\bx$ are indeed fractions (constraint \eqref{eq:fluid_dmd_constr}).
%
Let $W^{\textup{SPP($\bphi$)}}$ be the optimal value of SPP($\bphi$).
The following proposition formalizes that, the optimal finite horizon per customer payoff $W^*_T$ cannot be much larger than $W^{\aspp}$ where $\bar{\bphi}\triangleq \frac{1}{T}\sum_{t=0}^{T-1}\bphi^t$.
%
\begin{prop}\label{prop:fluid_ub}
	For any horizon $T<\infty$, any $K$ and any starting state $\bq[0]$, the finite horizon and steady state average payoff $W^*_T$, $W^*$ are upper bounded as
	\begin{align}\label{eq:fluid_ub}
		W^*_T \leq W^{\aspp} + m \frac{K}{T}\, ,\quad
		W^* \leq W^{\aspp}\, .
	\end{align}
\end{prop}
We obtain the finite horizon upper bound to $W^*_T$ in \eqref{eq:fluid_ub} by slightly relaxing the flow constraint \eqref{eq:fluid_flow_bal} to accommodate the fact that flow balance need not be exactly satisfied over a finite horizon. The proof is in Appendix \ref{appen:proof-of-fluid-ub}.

Our proposed MBP policy and its analysis is closely related to the (partial) dual of the SPP($\bphi^t$):
\begin{align}
	\textup{minimize}_{\by}\, g^t(\by),\quad
	&\textup{for }
	g^t(\by) \triangleq\;
	\sum_{\tau\in \cT}
	\phi_{\tau}^t \max_{j \in \cP(\tau), k \in \cD(\tau)}
	\left(
	w_{j\tau k} + y_j - y_k
	\right)^+\, .\label{eq:partial_dual}
\end{align}
where $(x)^+\triangleq\max\{0,x\}$.
Here $\by$ are the dual variables corresponding to the flow balance constraints \eqref{eq:fluid_flow_bal}, and have the interpretation of ``congestion costs'' \citep{neely2010stochastic}, i.e., $y_j$ can be thought of as the ``cost'' of having one extra supply unit at node $j$. In fact, as we describe in Appendix \ref{appen:MBP-as-MD}, MBP has the attractive property that it executes stochastic mirror descent \citep{beck2003mirror} on \eqref{eq:partial_dual} (where the inverse mirror map is the antiderivative of the congestion function $\mathbf{f}(\cdot)$).

Our main result, Theorem \ref{thm:main_regret}, is directly driven by the proposition below that connects the performance of MBP and the benchmark in the non-stationary environment.
\begin{prop}\label{prop:main_regret_DeltaT}
	Consider the setting in Theorem \ref{thm:main_regret}. Then there exists $K_1 = \textup{poly}\left(m,\frac{1}{\alpha_{\min}}\right)$, and a universal constant $C< \infty$,  such that the following holds.
	For the congestion function $f(\cdot)$ defined in \eqref{eq:inv_sqrt_mirror_map},
	for any  $K\geq K_1$, and any $0<\Delta_T<T$ the following guarantees hold for Algorithm~\ref{alg:mirror_bp}
	%
	\begin{align*}
		L^{\textup{MBP}}_T\leq
		M_1 \frac{K}{\Delta_T} + M_2\frac{1}{K}
		+
		\Delta_T m\eta
				\, , \qquad \textup{for} \ M_1 \triangleq Cm \ \textup{and} \ M_2 \triangleq Cm^2 \,.
	\end{align*}
\end{prop}
We illustrate the high-level structure of the  proof in Figure \ref{fig:gap_decomp}.
We introduce a ``batch size'' quantity $\Delta_T$;
note that $\Delta_T$ is only used for the purpose of analysis and is \emph{not} part of our algorithms.
The loss of MBP policy in $\Delta_T$ periods can be decomposed into three terms:

\begin{compactenum}[(1),leftmargin=*]
	\item The first term (from left to right in Figure \ref{fig:gap_decomp}), is $\frac{1}{T}\sum_{t=0}^{T-1}W^{\textup{SPP($\bphi^t$)}} - W_T^{\textup{MBP}}$. 
	We call it the \emph{policy gap}.
	We bound the policy gap using a Lyapunov analysis in Proposition \ref{prop:policy-gap} in Section \ref{subsec:policy-gap}. The antiderivative of the congestion function $\mathbf{f}(\cdot)$ serves as the Lyapunov function. Our choice of  $\mathbf{f}(\cdot)$ ensures that the policy gap is (provably) small despite the no-underflow constraints.
	\item The second term, $W^{\aspp} - \frac{1}{\Delta_T}\sum_{t=0}^{\Delta_T-1}W^{\textup{SPP($\bphi^t$)}}$ arises from the temporal variation of demand, therefore we call it the \emph{variation gap} (note that it is zero for stationary demand).
	We bound the variation gap in Proposition \ref{lem:SPPavg-vs-avgofSPPt} in Section \ref{sec:bound-variation-gap} using graph-theoretic analysis of the sensitivity of the SPP to $\bphi$.
	\item The third term, $W^*_T - W^{\aspp}$, has already been bounded in Proposition \ref{prop:fluid_ub}.
\end{compactenum}

Intuitively, the quantities in Figure \ref{fig:gap_decomp} highlights the following trade-off.
When $\Delta_T$ is large, the policy gap is small; when $\Delta_T$ is small, the variation gap is small.
Theorem \ref{thm:main_regret} follows by balancing this trade-off and setting $\Delta_T=\Theta(\min(\sqrt{K/\eta}, T))$, which we prove in Appendix \ref{append:proof-main-regret}.

\begin{figure}[!t]
	\centering
	\includegraphics[width=0.8\linewidth]{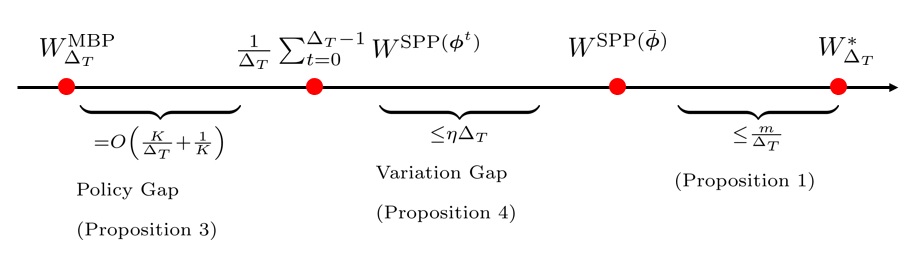}
	\caption{Roadmap of the proof of Proposition \ref{prop:main_regret_DeltaT}
	}\label{fig:gap_decomp}
\end{figure}

\subsection{Bounding the Policy Gap}\label{subsec:policy-gap}
In this section, we prove the following proposition that bounds the policy gap under MBP:
\begin{prop}\label{prop:policy-gap}
	Consider a set of $m$ nodes and any sequence of demand arrival rates $\{\bphi^t\}_{t\leq T}$ that satisfy Condition  \ref{cond:str_connect_jea}.
	\yka{Do you really need $\eta$-slowly varying for this prop. If not, please remove this requirement and moreover highlight in the text that this prop holds for \emph{arbitrary} sequences of demand arrival rates satisfying Condition 1.}
	Recall that $\alpha_{\min}= \min_{0\leq t\leq T}\alpha(\bphi^t)>0$.
	Then there exists $K_1 = \textup{poly}\left(m,\frac{1}{\alpha_{\min}}\right)$, and a universal constant $C< \infty$, such that the following holds.
	For the congestion function $f(\cdot)$ defined in \eqref{eq:inv_sqrt_mirror_map},
	for any $K\geq K_1$, the following guarantees hold for Algorithm~\ref{alg:mirror_bp}:
	\begin{align*}
		\frac{1}{T}\sum_{t=0}^{T-1}W^{\textup{SPP($\bphi^t$)}} - W_T^{\textup{MBP}}
		\leq
		M_1 \frac{K}{T}
		+
		M_2	\frac{1}{K}\, ,
		\qquad \textup{for} \ M_1 \triangleq Cm \ \textup{and} \ M_2 \triangleq Cm^2 \,.
	\end{align*}
\end{prop}

Notably, this proposition holds for \emph{arbitrary} sequences of demand arrival rates satisfying Condition 1.
The proof is based on Lyapunov analysis.
A Lyapunov function will allow to decompose the expected payoff from the next arrival into a combination of the objective, change in potential, and a per-period loss.
We use the antiderivative of $\bof(\cdot)$ as our Lyapunov function; for the congestion function $f$
in \eqref{eq:inv_sqrt_mirror_map}, this is
\begin{align}\label{eq:defn-lyap-func}
	F(\nq) \triangleq -2\sqrt{m}\sum_{j \in V}\sqrt{\bar{q}_j}\, .
\end{align}
The motivation for our choice of Lyapunov function come from the ``drift-plus-penalty'' framework in the network control literature {\citep[see, e.g., ][]{neely2006super,neely2010stochastic,gupta2020interior}. 
We generalize and refine the framework by adding a new term in the analysis, which allows us to bound the suboptimality contributed by underflow.
}
\yka{Deleted a para here. I suspect that, as in \cite{gupta2020interior}, the argument does not in fact leverage the connection with mirror descent. It's basically a Lyapunov drift argument right? If so, please say that and include some citations. And you can mention your work to bound the suboptimality contributed by underflow.}
Recall that $W^{\textup{SPP($\bphi^t$)}}$ is the optimal value of SPP \eqref{eq:fluid_obj}-\eqref{eq:fluid_dmd_constr} with demand arrival rate $\bphi^t$, $v^{\textup{MBP}}[t]$ denotes the payoff collected under the MBP policy in the $t$-th period, and $g^t(\cdot)$ is the dual problem \eqref{eq:partial_dual}.
We have the following key lemma (proved in Appendix~\ref{append:proof-one-step-lyapunov}):
\begin{restatable}[Suboptimality of MBP in one period]{lem}{onesteplyap}\label{lem:one_step_lyap}
	Consider congestion functions $f(\cdot)$s that are strictly increasing and continuously differentiable.
	We have the following decomposition (recall that $\tK = K + m\sqrt{K}$): 
	\begin{align}
	W^{\textup{SPP($\bphi^t$)}} -
	\mathbb{E}[v^{\textup{MBP}}[t]|\nq[t]] \nonumber
	\leq&\
	\underbrace{\tK\left(F (\nq[t]) - \mathbb{E}[F (\nq[t+1])|\nq[t]]\right)}_{\cV_1\textup{ change in potential}}
	+
	\underbrace{\frac{1}{2 \tK} \max_{j\in V} \max_{\bar{q}\in \left[\bar{q}_j[t]-\frac{1}{\tK},\bar{q}_j[t]+\frac{1}{\tK}\right]} \left|f'(\bar{q})\right|}_{\cV_2\textup{ loss due to stochasticity}}\\
	&\ +
	\underbrace{\left(W^{\textup{SPP($\bphi^t$)}}
		-g^t(\bof(\nq[t]))
		\right)}_{\cV_3\textup{ dual optimality gap}}
	+
	\underbrace{\mathds{1}\left\{q_{j}[t]=0,\exists j\in V\right\}}_{\cV_4\textup{ loss due to underflow}}\, . \label{eq:one_step_lyap}
	\end{align}
\end{restatable}
In Lemma \ref{lem:one_step_lyap}, the LHS of \eqref{eq:one_step_lyap} is the suboptimality incurred by MBP (benchmark against the value of SPP($\bphi^t$)) in a single period.
On the RHS of \eqref{eq:one_step_lyap}, $\cV_1$ and $\cV_2$ come from the standard Lyapunov drift argument \yka{Please eliminate mention of mirror descent here and instead refer to the Lyapunov drift argument if these terms are a standard component of that argument (I would imagine so).}; $\cV_3$ is the negative of the dual suboptimality at $\by = \bof(\nq[t])$, hence it is always non-positive; $\cV_4$ is the payoff loss because of underflow.

For the change in potential $\cV_1$, note that it forms a telescoping series when summing over many periods, and will therefore remain bounded.
Hence as we average over many periods, we have that $\cV_1$ tends to zero.

We now proceed to upper bound $\cV_2 + \cV_3 + \cV_4$ on the RHS of \eqref{eq:one_step_lyap}.
We outline our analysis in the following.
Observe that the terms $\cV_2$ and $\cV_4$ are non-negative, while $\cV_3$ is non-positive, thus the goal is to show that $\cV_3$ compensates for $\cV_2 + \cV_4$.
First notice that $\cV_2$ is large when there exist very short queues (because the congestion function \eqref{eq:inv_sqrt_mirror_map} changes rapidly only for short queue lengths), and $\cV_4$ is non-zero only when some queues are empty.
Helpfully, it turns out that $\cV_3$ is more negative in these same cases; we show this by exploiting the structure of the dual problem \eqref{eq:partial_dual}.
In Lemma \ref{lem:dual_subopt} we provide an upper bound for $\cV_3$ that becomes more negative as the shortest queue length decreases.
We prove Lemma \ref{lem:dual_subopt} in Appendix \ref{append:proof-dual-subopt} by utilizing complementary slackness for the SPP \eqref{eq:fluid_obj}-\eqref{eq:fluid_dmd_constr}.
\begin{restatable}{lem}{dualsubopt}\label{lem:dual_subopt}
	Consider congestion functions $f(\cdot)$s that are strictly increasing and continuously differentiable, and any $\bphi$ with connectedness $\alpha(\bphi)>0$.
	We have
	\begin{align*}
		\cV_3 \leq\
		-\alpha(\bphi)\cdot \left[\max_{j\in V} f(\bar{q}_j) - \min_{j\in V} f(\bar{q}_j) - 2 m\right]^+\, .
	\end{align*}
\end{restatable}


%

%

%
%

The following lemma bounds $\cV_2 + \cV_3 + \cV_4$. The proof is in Appendix~\ref{append:proof-bound-rhs}.
(In fact we prove a general version of the lemma which applies to all congestion functions that satisfy certain growth conditions formalized in Condition~\ref{cond:mirror_map} in Appendix~\ref{append:proof-main-theorems}. The growth conditions serve to ensure that $\cV_3$ compensates for $\cV_2 + \cV_4$.)
\begin{lem}\label{lem:bound_rhs}
	Consider the congestion function \eqref{eq:inv_sqrt_mirror_map}, and any $\bphi$ with connectedness $\alpha(\bphi)>0$.
	Then there exists $K_1 = \textup{poly}\left(m,\frac{1}{\alpha(\bphi)}\right)$ 
such that for $K\geq K_1$ and a universal constant $C>0$,
	\begin{align}\label{eq:bound_rhs}
		\cV_2 + \cV_3 + \cV_4 \leq  M_2\frac{1}{\tK}
		\qquad \textup{for} \ M_2 \triangleq Cm^2 \,.
	\end{align}
	Recall that $\tK = K + m\sqrt{K}$. 
\end{lem}

Putting Lemma \ref{lem:one_step_lyap} and Lemma \ref{lem:bound_rhs} together leads to the following proof of Proposition \ref{prop:policy-gap}.
The main
idea is to use the Lyapunov drift argument of \cite{neely2010stochastic}\yka{replace with proper citation}, namely, to sum the expectation of \eqref{eq:one_step_lyap} (the bound in Lemma~\ref{lem:one_step_lyap}) over the first $T$ time steps.
\begin{proof}[Proof of Proposition \ref{prop:policy-gap}]
	Plugging in Lemma \ref{lem:bound_rhs} into \eqref{eq:one_step_lyap} in Lemma~\ref{lem:one_step_lyap} and taking expectation, we obtain that there exists $K_1 = \textup{poly}\left(m,\frac{1}{\alpha_{\min}}\right)$ such that for any $0\leq t\leq T$,
	\begin{align}
		W^{\textup{SPP($\bphi^t$)}}-
		\mathbb{E}[v^{\textup{MBP}}[t]]
		\leq&\
		\tK\left(\mathbb{E}[F(\nq[t])] - \mathbb{E}[F (\nq[t+1])]\right)
		+
		M_2 \frac{1}{\tK}\, \qquad \textup{for } K \geq K_1 \, , \label{eq:one_step_lyap_1}
	\end{align}
	here $\tK = K + m\sqrt{K}$. Take the sum of both sides of the inequality \eqref{eq:one_step_lyap_1} from $t=0$ to $t=T-1$, and divide the sum by $T$.
	This yields
	\begin{align*}
		\frac{1}{T}\sum_{t=0}^{T-1}W^{\textup{SPP($\bphi^t$)}}-
		W^{\textup{MBP}}_T
		\leq&\
		\frac{\tK}{T}\left(\mathbb{E}[F(\nq[0])] - \mathbb{E}[F (\nq[T])]\right)
		+
		M_2 \frac{1}{\tK}\\
		\leq&\
		\frac{\tK}{T}\sup_{\nq_1,\nq_2\in\Omega}\left(F(\nq_1) - F(\nq_2)\right)
		+
		M_2 \frac{1}{\tK}
		\, \qquad \textup{for } K \geq K_1 \, .
	\end{align*}
	Let $M_1 \triangleq \sup_{\nq_1,\nq_2\in\Omega}\left(F(\nq_1) - F(\nq_2)\right) $.
	Observe that the function $F(\nq)$ given in \eqref{eq:defn-lyap-func} is negative $F(\nq)\leq 0$ for all $\nq\in\Omega$, and  is a convex function which achieves its minimum at $\nq = \frac{1}{m}\mathbf{1}$. Hence
	\begin{align*}
		M_1
		\leq
		0-\inf_{\nq\in\Omega}\ F(\nq)
		=
		-F\left(\frac{1}{m}\mathbf{1}\right) =2 m\, .
	\end{align*}
	Therefore the policy gap of MBP is upper bounded by $M_1\frac{\tK}{T} + M_2 \frac{1}{\tK}$ where $M_1 = C m$, $M_2 = C m^2$ and $C$ does not depend on $m$, $K$, or $\alpha_{\min}$. Moreover, $\tK = K + m \sqrt{K} \in [K, 2K]$ taking $K_1 \geq m^2$. 
	This concludes the proof.
\end{proof}

We conclude with some informal intuition as to why MBP with congestion function $\bof(\cdot)$ given in \eqref{eq:inv_sqrt_mirror_map} and normalized queue lengths $\nq$ defined in \eqref{eq:normalized_queue} ensures a small policy gap. MBP could run into two issues due to the no-underflow constraints: (i) The queue lengths corresponding to the optimal dual variables lie outside of the state space; (ii) The Lyapunov drift could be positive at certain ``boundary states'', i.e., states where some of the queues are empty.
%
For issue (i), note that while the range of normalized queue length $\bar{\bq}$ belongs to $[0,1]$, the range of $f(\nq)$ goes to $(-\infty,-\sqrt{m})$ as $K\to\infty$. As a result, for large enough $K$, there exists $\nq\in\Omega$ such that $\bof(\nq)$ corresponds to the optimal dual variables.\footnote{Note that optimal dual variables for \eqref{eq:partial_dual} is non-unique, since if $\by^*$ is optimal, $\by^* + \theta\mathbf{1}$ is also optimal for any $\theta\in\mathbb{R}$. Therefore we can always find an optimal dual variable that corresponds to $\bof(\nq)$ where $\nq\in\Omega$.}
For issue (ii), first note that this problem only occurs when there exists empty queues. 
At these states, the dual-suboptimality at $\bof(\nq)$ is large because $f(0)\to -\infty$ as $K\to\infty$, which creates a negative Lyapunov drift that ``pushes'' $\bof(\nq)$ towards the optimal dual variable. 
This corresponds to the intuition that MBP is aggressive in preserving supply units in near-empty queues. 
In contrast, we show in Appendix \ref{append:bp_fails} that regular backpressure (i.e., linear $\mathbf{f}(\cdot)$) may fail to address the two issues mentioned above, leading to a large policy gap.

\subsection{Bounding the Variation Gap}\label{sec:bound-variation-gap}
We have the following result that bounds the variation gap.
\begin{prop}\label{lem:SPPavg-vs-avgofSPPt}
	Suppose the demand arrival rates vary $\eta$-slowly (Definition~\ref{defn:variation-budget}) for some $\eta > 0$.
	Fix a horizon $T$. For any $0 \leq t \leq T-1$ we have
	\begin{align}
		\frac{1}{T}\sum_{t=1}^{T} W^{\textup{SPP}(\bphi_t)} \geq  W^{\textup{SPP}(\bar{\bphi)}} - m \eta T\, .
		\label{eq:SPPavg-vs-avgofSPPt}
	\end{align}
\end{prop}

The proof is based on sensitivity analysis of the linear program SPP($\bar{\bphi}$) via flow decomposition.
We prove Proposition \ref{lem:SPPavg-vs-avgofSPPt} in Appendix \ref{append:proof-main-theorems}.

\subsection{Proof of Proposition \ref{prop:main_regret_DeltaT}}
Using Proposition \ref{prop:policy-gap} and considering the first $\Delta_T$ periods, we obtain that
there exists $K_2 = \textup{poly}\left(m,\frac{1}{\alpha_{\min}}\right)$, and a universal constant $C< \infty$, such that the following holds.
For the congestion function $f(\cdot)$ defined in \eqref{eq:inv_sqrt_mirror_map}, for any $K\geq K_2$, the following guarantees hold for Algorithm~\ref{alg:mirror_bp}:
\begin{align}
	\frac{1}{\Delta_T}\sum_{t=0}^{\Delta_T-1}W^{\textup{SPP($\bphi^t$)}} - W_{\Delta_T}^{\textup{MBP}}
	\leq
	M_1 \frac{K}{\Delta_T}
	+
	M_2	\frac{1}{K}\, ,
	\qquad \textup{for} \ M_1 \triangleq Cm \ \textup{and} \ M_2 \triangleq Cm^2 \,.\label{eq:payoff-loss-jea-1}
\end{align}
%
Suppose the demand varies $\eta_1$-slowly on $[0,\Delta_T]$.
Using Proposition~\ref{prop:fluid_ub} and then Proposition~\ref{lem:SPPavg-vs-avgofSPPt}, we have
\begin{align}
	L^{\textup{MBP}}_{\Delta_T}
	=
	W^*_{\Delta_T}-
	W_{\Delta_T}^{\textup{MBP}}
	\leq&\
	\left(W^\aspp + m\frac{K}{\Delta_T}\right)  -  W^{\textup{MBP}}_{\Delta_T}  \nonumber\\
	\leq&\  \left( m\eta_1\Delta_T + \tfrac{1}{{\Delta_T}} \sum_{t=0}^{{\Delta_T}-1} W^{\textup{SPP($\bphi^t$)}}\right)  \ +   m\frac{K}{{\Delta_T}} -  W^{\textup{MBP}}_{\Delta_T} \nonumber\\
	\leq&\ \left( Cm\frac{K}{\Delta_T}  +
	M_2 \frac{1}{K}\right)  +  m\eta_1\Delta_T + m\frac{K}{\Delta_T}  \nonumber \\
	\leq&\ \frac{K}{\Delta_T}  m (C +1) +
	M_2 \frac{1}{K}  +   m\eta_1 \Delta_T\, . \label{eq:LT-bound}
\end{align}
where we used \eqref{eq:payoff-loss-jea-1} in the third inequality.
Suppose the demand varies $\eta_{\ell}$-slowly in batch $\ell$, and note that since the demand varies $\eta$-slowly over the whole horizon $[0,T]$, we must have $\eta=\frac{1}{\lfloor \frac{T}{\Delta_T} \rfloor}\sum_{\ell=1}^{\lfloor \frac{T}{\Delta_T} \rfloor}\eta_{\ell}$.
Now take average on both sides of \eqref{eq:LT-bound} over $\lfloor \frac{T}{\Delta_T} \rfloor$ batches, we obtain the result.

\section{Generalizations and Extensions}\label{sec:extension}
In this section, we
consider two general settings, namely, one with finite buffer queues and one allows joint pricing-assignment control (JPA).
We show that the extended models enjoy similar performance guarantees to that in Theorem \ref{thm:main_regret} under mild conditions on the model primitives.

\subsection{Congestion Functions for Finite Buffer Queue}\label{subsec:extend-barrier}

{Suppose the queues at a subset of nodes $\Vb \subset V$ have a finite buffer constraint.
For $j\in\Vb$, denote the buffer size by $d_{j} = \bar{d}_j K$ for some scaled buffer size $\bar{d}_j \in (0,1)$. (If $\bar{d}_j \geq 1$, the buffer size exceeds the number of supply units $d_j \geq K$ and there is no constraint as a result, i.e., $j \notin \Vb$.)
We will find it convenient to define $\bar{d}_j=1$ for each $j\in V \backslash \Vb$.
To avoid the infeasible case where the buffers are too small to accommodate all supply units, we assume that $\sum_{j\in V}\bar{d}_j>1$.}
%
The normalized state space will be
\begin{align*}
\Omega \triangleq
\left\{
\nq: \mathbf{1}^{\T} \nq = 1,\
\bzero \leq \nq \leq \bar{\bd}
\right\}\, ,\quad
\textup{where }
\bar{d}_j \triangleq d_j/K\, .
\end{align*}

Similar to the case of entry control, we need to keep $\nq$ in the interior of $\Omega$, which is achieved by defining the normalized queue lengths $\nq$ as
\begin{align}\label{eq:normalized-queue-general}
\bar{q}_j \triangleq \frac{q_j + \bar{d}_j\delta_K}{\tK}
\quad \textup{for}\quad
\delta_K = \sqrt{K}
\quad \textup{and}\quad
\tK \triangleq K + \left(\sum_{j\in V}\bar{d}_j\right)\delta_K\, .
\end{align}
%
One can verify that $\nq \in \Omega$ for any feasible state $\bq$. {When $\bar{d}_j=1$ for all $j\in V$, the definition of  $\bar{q}_j$ in \eqref{eq:normalized-queue-general} reduces to the one in \eqref{eq:normalized_queue}.}
The congestion functions $(f_j(\cdot))_{j\in V}$ are monotone increasing functions that map (normalized) queue lengths to congestion costs.
%
Here we will state our main results for the congestion function vector
\begin{align}
f_j(\bar{q}_j) \triangleq \left \{
\begin{array}{ll}
\sqrt{m} \cdot
\Cb\cdot\left(
\left(1 - \frac{\bar{q}_j}{\bar{d}_j}\right)^{-\frac{1}{2}}
- \left(\frac{\bar{q}_j}{\bar{d}_j}\right)^{-\frac{1}{2}}
- \Db\right)\, , &\forall j\in \Vb \, ,\\
 - \sqrt{m}\cdot{\bar{q}_j}^{-\frac{1}{2}}  &\forall j\in V\backslash \Vb\, .
\end{array}
\right .
\label{eq:inv_sqrt_mirror_map_gen}
\end{align}
Here $\Cb$ and $\Db$ are normalizing constants\footnote{\yka{Please update this footnote.}
Define $\epsilon\triangleq \frac{\delta_K}{\tK}$.
Let $\hb(\bar{q})\triangleq (1-\bar{q})^{-\frac{1}{2}} - {\bar{q}}^{-\frac{1}{2}}$ and
$h(\bar{q})\triangleq - {\bar{q}}^{-\frac{1}{2}}$.
Define $\Cb\triangleq \frac{h(\epsilon)-h(1/\sum_{j\in V}\bar{d}_j)}{\hb(\epsilon)-\hb(1/\sum_{j\in V}\bar{d}_j)}$ and $\Db \triangleq \hb(1/\sum_{j\in V}\bar{d}_j)-\Cb^{-1}h(1/\sum_{j\in V}\bar{d}_j)$.
%
%
In addition to the properties listed in the main text, we also have that $f_j(\bar{d}_j/\sum_{\ell\in V}\bar{d}_{\ell})$ has the same value for all $j\in V$.
These properties are useful in the following analysis.}
%
chosen to ensure that (i)~for all $j,k\in V$, we have that  $f_j(\bar{q}_j)=f_k(\bar{q}_k)$ when both queues are empty $q_j=q_k=0$; (ii)~for all $j,k\in \Vb$, we have that $f_j(\bar{q}_j)=f_k(\bar{q}_k)$ when both queues are full  $q_j=d_j$, $q_k=d_k$. (We state the results for other choices of congestion functions in Appendix \ref{append:proof-main-theorems}.)

Note that $f_j(\cdot)$ in \eqref{eq:inv_sqrt_mirror_map_gen} is identical to $f(\cdot)$ in \eqref{eq:inv_sqrt_mirror_map} for $j \notin V_b$, i.e., \eqref{eq:inv_sqrt_mirror_map_gen} is a generalization of \eqref{eq:inv_sqrt_mirror_map} to the case where some queues have buffer constraints.
 The intuitive reason \eqref{eq:inv_sqrt_mirror_map_gen} is a suitable congestion function is that it enables MBP to focus on queues which are currently either almost empty or almost full (the congestion function values for those queues take on their smallest and largest values, respectively), and use the control levers available to make the queue lengths for those queues trend strongly away from the boundary they are close to.

We have the following result for the finite-buffer setting. The proof is in Appendix \ref{append:proof-main-regret}.
\begin{thm}\label{thm:main-result-finite-buffer}
	Consider a set $V$ of $m \! \triangleq \! |V|>1$ nodes, a subset $\Vb \subseteq V$ of buffer-constrained nodes with scaled buffer sizes $\bar{d}_j \in (0,1) \ \forall j \in \Vb$ satisfying\footnote{Recall that we define $\bar{d}_j \triangleq 1$ for all $j \in V\backslash \Vb$.} $\sum_{j\in V}\bar{d}_j>1$.
	Consider any sequence of demand arrival rates $(\bphi^t)_{t\leq T}$ that satisfy Condition \ref{cond:str_connect_jea} and vary $\eta$-slowly (Definition \ref{defn:variation-budget}).
	Recall that $\alpha_{\min}=\min_{0\leq t\leq T}\alpha(\bphi^t)$. 
	Then there exists $K_1=\textup{poly}\left(m,\bar{\bd},\frac{1}{\alpha_{\min}}\right)$, and a universal constant $C<\infty$, such that the following holds.
	For the congestion function $\bof(\cdot)$ defined in \eqref{eq:inv_sqrt_mirror_map_gen}, for any $K\geq K_1$, the following guarantees hold for Algorithm \ref{alg:mirror_bp}
	\begin{align*}
		&L^{\textup{MBP}}_T\leq
		M_1 \left( \frac{K}{T} + \sqrt{\eta K} \right )
		+
		M_2
		\frac{1}{K}\, ,\\
\textup{for}\quad  &M_1=C m\, ,\quad M_2=C \frac{1}{\min_{j\in V}\bar{d}_j}\left(\frac{\sum_{j\in V}\bar{d}_j}{\min\{\sum_{j\in V}\bar{d}_j-1,1\}}\right)^{3/2}
\sqrt{m}\, .
\end{align*}
\end{thm}

\subsection{Joint Pricing-Assignment Setting}\label{subsec:jpa}
In this section, we consider the joint pricing-assignment (JPA) setting and design the corresponding MBP policy.
The platform's control problem is to set a price for each demand origin-destination pair, and decide an assignment at each period to maximize payoff. Our model here will be similar to that of \citet*{banerjee2016pricing}, except that the platform does \emph{not} know demand arrival rates, and we allow a finite horizon.
{The demand types $\tau$, pick-up neighborhood $\cP(\tau)$ and drop-off neighborhood $\cD(\tau)$ are defined in the same way as in Section \ref{sec:model}.} For simplicity, we assume that the demand type distribution $\bphi = (\phi_\tau)_{\tau \in \cT}$ is time invariant, and that all buffers have infinite capacity in this subsection.

The platform control and payoff in this setting are as follows.
At time $t$, after observing the demand type 
$\tau[t]=\tau$, the system chooses a \emph{price} $p_{\tau}[t]\in [p_{\tau}^{\min},p_{\tau}^{\max}]$ 
and a decision
\begin{align}\label{eq:jpa-assignment}
(x_{j\tau k}[t])_{j\in\cP(\tau),k\in\cD(\tau)}\in \{0,1\}^{|\cP(\tau)|\cdot |\cD(\tau)|}\quad
\textup{such that}
\quad
\sum_{j\in\cP(\tau),k\in\cD(\tau)}x_{j\tau k}[t]\leq 1\, .
\end{align}
As before we require
$
	x_{j\tau k}[t] = 0
$
if $q_{j}[t]=0$.

The result of the platform control is as follows:
%
\begin{compactenum}[(1),wide, labelwidth=!, labelindent=0pt]
	\item Upon seeing the price, the arriving demand unit will decline (to buy) with probability $F_{\tau}(p_{\tau}[t])$, where $F_{\tau}(\cdot)$ is the cumulative distribution function of type $\tau$ demand's willingness-to-pay.
	\item If the demand accepts (i.e., buys), 
then a supply unit relocates based on $x_{j\tau k}[t]$.
	Meanwhile, the platform collects payoff $v[t] =  p_{\tau}[t] - c_{j\tau k}$ where $c_{j\tau k}$ is the ``cost'' of serving a demand unit of type $\tau$ using pick-up node $j$ and drop-off node $k$.
	%
	If the demand unit declines, 
the supply units do not move and $v[t]=0$.
\end{compactenum}
%

We assume the following regularity conditions to hold for demand functions $\big (F_{\tau}(p_{\tau})\big )_{\tau}$. These assumptions are quite standard in the revenue management literature, \citep[see, e.g.,][]{gallego1994optimal}. 
%
\begin{cond}\label{cond:pricing_elas}
	\begin{compactenum}[(1),wide, labelwidth=!,labelindent=0pt]
		\item Assume\footnote{The assumption $F_{\tau}(p_{\tau}^{\min})=0$ is without loss of generality, since if a fraction of demand is unwilling to pay $p^{\min}_{\tau}$, that demand can be excluded from $\bphi$ itself.} $F_{\tau}(p_{\tau}^{\min})=0$  and that $F_{\tau}(p_{\tau}^{\max})=1$.
		%
		\item Each demand type's willingness-to-pay is non-atomic with support $[p_{\tau}^{\min},p_{\tau}^{\max}]$ and positive density everywhere on the support;  hence $F_{\tau}(p_{\tau})$ is differentiable and strictly increasing  on $(p_{\tau}^{\min},p_{\tau}^{\max})$. (If the support is a subinterval of
$[p_{\tau}^{\min},p_{\tau}^{\max}]$, we redefine $p_{\tau}^{\min}$ and $p_{\tau}^{\max}$ to be the boundaries of this subinterval.)
		%
		
		%
		%
		\item The revenue functions
		$
		r_{\tau}(\mu_{\tau})\triangleq \mu_{\tau}\cdot p_{\tau}(\mu_{\tau})
		$
		are concave and twice continuously differentiable, where $\mu_{\tau}$ denotes the fraction of demand of type $\tau$ which is realized (i.e., willing to pay the price offered).
	\end{compactenum}
\end{cond}
As a consequence of Condition~\ref{cond:pricing_elas} parts 1 and 2, the willingness to pay distribution $F_{\tau}(\cdot)$ has an inverse denoted as $p_{\tau}(\mu_{\tau}):[0,1] \rightarrow [p_{\tau}^{\min},p_{\tau}^{\max}]$ which gives the price which will cause any desired fraction $\mu_{\tau} \in [0,1]$ of demand to be realized. (The concavity assumption in part 3 of the condition is stated in terms of this function $p_\tau(\cdot)$.) 
Without loss of generality, let $\max_{\tau\in \cT}p_{\tau}^{\max} + \max_{j,k\in V,\tau\in\cT}|c_{j\tau k}| = 1$.

In the JPA setting, the net demand $\phi_{\tau} \mu_{\tau}$ plays a role in myopic revenues but also affects the distribution of supply, and the chosen prices need to balance myopic revenues with maintaining a good spatial distribution of supply.
Intuitively, when sufficiently flexible pricing is available as a control lever, the system should modulate the quantity of demand 
through changing the prices (and serving all the demand which is then realized) rather than apply entry control (i.e., dropping some demand proactively). 
Our MBP policy for this setting will have this feature.

The dual problem to the SPP in the JPA setting is (see Appendix \ref{appen:jpa} for the statement of SPP and the derivation of its dual):
\begin{align}
\textup{minimize}_{\by}\, g_{\textup{JPA}}(\by) \quad
\textup{for }
g_{\textup{JPA}}(\by) \triangleq \;
\sum_{\tau\in \cT}
\phi_{\tau}\max_{\{0\leq\mu_{\tau}\leq 1\}}
\left(
r_{\tau}\left(\mu_{\tau}\right)
+
\mu_{\tau}\max_{j\in\cP(\tau),k\in\cD(\tau)}	
\left(-c_{j\tau k} + y_j - y_k\right)
\right)\, .\label{eq:partial_dual_jpa}
\end{align}


Once again, the MBP policy (Algorithm \ref{alg:mirror_bp_jpa} below) is defined to achieve the argmaxes in the definition of the dual objective $\gJPA(\cdot)$ with the $y$s replaced by congestion costs:
MBP dynamically sets prices $p_\tau$ such that mean fraction of demand realized under the policy is the outer argmax in the definition \eqref{eq:partial_dual_jpa} of $\gJPA(\cdot)$,
and the assignment decision of MBP achieves the inner argmax in the definition \eqref{eq:partial_dual_jpa} of $\gJPA(\cdot)$. The policy again has the property that it executes stochastic mirror descent on the dual objective $\gJPA(\cdot)$.
The optimization problem for computing $\mu_{\tau}[t]$ is a one-dimensional concave maximization problem (Condition~\ref{cond:pricing_elas} part 3), hence $\mu_{\tau}[t]$ can be efficiently computed.

The MBP policy retains the advantage that it does not require any prior knowledge of gross demand $\bphi$.  
We assume that the willingness-to-pay distributions $F_{\tau}(\cdot)$s are exactly known to the platform; it may be possible to relax this assumption via a modified policy which ``learns'' the  $F_{\tau}(\cdot)$s, however, pursuing this direction is beyond the scope of the present paper.

\begin{algorithm}[H]\label{alg:mirror_bp_jpa}
	\SetAlgoLined
	{At the start of period $t$, the system observes $\tau[t]=\tau$.}
	
	$(j^*,k^*) \leftarrow \arg \max_{j\in\cP (\tau),k\in\cD(\tau)}\left\{ -c_{j\tau k} + f_j(\bar{q}_j[t]) - f_k(\bar{q}_k[t])  \right\}$\;
	
	\eIf{$q_{j^*}[t] > 0$}{
		
        $\mu_{\tau}[t] \leftarrow \textup{argmax}_{\mu_{\tau} \in [0,1]}\left\{
		r_{\tau}(\mu_{\tau})+
		\mu_{\tau}\cdot
		(-c_{j^* \tau k^*} + f_{j^*}(\bar{q}_{j^*}[t]) - f_{k^*}(\bar{q}_{k^*}[t]))
		\right\}$\;
        $p_{\tau}[t] \leftarrow F^{-1}_{\tau}(\mu_{\tau}[t])$\;
		$x_{j^* \tau k^*}[t] \leftarrow 1$, i.e., if the incoming demand stays, serve it by pick up from $j^*$ and drop off at $k^*$, otherwise do nothing\;
	}{
		$x_{j^* \tau k^*}[t] \leftarrow 0$, i.e., drop the incoming demand\;
	}
	
	The queue lengths update as $\nq[t+1] = \nq[t] - \frac{1}{\tK}
x_{j^* \tau k^*}[t](\boe_{j^*} - \boe_{k^*})$.
	\caption{Mirror Backpressure (MBP) Policy for Joint Pricing-Assignment}
\end{algorithm}

\medskip

We have the following performance guarantee for Algorithm~\ref{alg:mirror_bp_jpa}, analogous to Theorem \ref{thm:main_regret}.
\begin{thm}\label{thm:main_regret_jpa}
Fix a set $V$ of $m=|V|> 1$ nodes, 
minimum and maximum allowed prices $(p_{\tau}^{\min}, p_{\tau}^{\max})_{\tau\in\cT}$,
 any $(\bphi,\cP,\cD)$ that satisfy Condition~\ref{cond:str_connect_jea} (strong connectivity), and willingness-to-pay distributions $(F_{\tau})_{\tau\in\cT}$ that satisfy Condition~\ref{cond:pricing_elas}.  
Then there exist $K_1 < \infty$, $M_1=Cm$, and $M_2=Cm^2$ for universal constant $C>0$ such that
for the congestion function $f(\cdot)$ defined in \eqref{eq:inv_sqrt_mirror_map}, the following guarantee holds for Algorithm \ref{alg:mirror_bp_jpa}. For any horizon $T$ and for any $K \geq K_1$, we have
	\begin{align*}
	L^{\textup{MBP}}_T\leq
M_1 \frac{K}{T}
+
M_2
\frac{1}{K}\, ,
\qquad\textup{and}\qquad
L^{\textup{MBP}} \leq
M_2
\frac{1}{K} \, .
	\end{align*}
\end{thm}
\pqa{Again, check the factor $m^2$ and $\rho$.}
We outline the proof of Theorem \ref{thm:main_regret_jpa} in Appendix \ref{appen:jpa}.

\section{Application to Shared Transportation Systems}\label{sec:shared-transportation}

\yka{Updated this section.}
Our setting can be mapped to shared transportation systems such as bike sharing and ride-hailing systems. In this context,  the nodes in our model correspond to geographical locations, while supply units and demand units correspond to vehicles and customers, respectively.

\emph{Dynamic incentive program for bike sharing systems.} 
\citet{chung2018bike} explain that Citi Bike's Bike Angel incentive program works as follows: there are two types of bike stations {at any time}, the incentivized ones and neutral ones; depending on
the origin and the destination stations of a trip, different amounts of points are awarded to the rider. The points have monetary values. The system objective is to minimize out-of-stock and out-of-bike events.
Therefore, to view it as an application of our JPA model, we can view the amount of points awarded for a certain trip as (the negative of) price of this trip; the customers have a demand function denoting their response to reward points (i.e., negative of prices);
and the value the platform derives from a ride equals customer utility and/or revenue generated (which is a constant) minus the cash value of points awarded.
By using a JPA-based MBP policy, the platform can dynamically set the amount of reward points for each origin-destination pair.
In docked bike sharing systems, there is a constraint on the number of docks available at each location. Such constraints are seamlessly handled in our framework as detailed earlier in Section~\ref{subsec:extend-barrier}. 
One concern may be that our model ignores travel delays. However, in most bike sharing systems, the fraction of bikes in transit at any time is typically quite small (under 10-20\%).\footnote{The report {https://nacto.org/bike-share-statistics-2017/} tells us that U.S. dock-based systems produced an average of 1.7 rides/bike/day, while dockless bike share systems nationally had an average of about 0.3 rides/bike/day. Average trip duration was 12 minutes for pass holders (subscribers) and 28 mins for casual users. In other words, for most systems, each bike was used less than 1 hour per day, which implies that less than 10\% of bikes are in use at any given time during day hours (in fact the utilization is below 20\% even during rush hours).} As a result, we expect our control insights to retain their power despite the presence of delays. (Indeed, we will numerically demonstrate in Section~\ref{subsec:numerics-ridehailing} that this is the case in the ridehailing setting; see the excess supply case where MBP performs well even when the vast majority of supply is in transit at any time.) We leave a detailed study of bike sharing platforms to future work.

\emph{Online control of ride-hailing platforms.} Ride-hailing platforms make dynamic decisions to optimize their objectives (e.g., revenue, welfare, etc.).
	%
	%
For most ride-hailing platforms in North America, pricing is used to modulate demand.
	%
In certain countries such as China, however, pricing is a less acceptable lever, hence admission control of customers is used as a control lever instead.
	%
In both cases, the platform further decides where (near the demand's origin) to dispatch a car from, and where (near the demand's destination) to drop off a customer.
These scenarios are captured, respectively, by the joint entry-assignment (JEA) model\footnote{The JEA setting can be mapped to ride-hailing as follows: there is a demand type $\tau$ corresponding to each (origin, destination) pair $(j,k) = V^2$, with $\cP(\tau)$ being nodes close to the origin $j$ and $\cD(\tau)$ being nodes close to the destination $k$.} studied in Section \ref{sec:model} and joint pricing-assignment (JPA) model studied in Section \ref{sec:extension}. Again, a concern may be that travel delays play a significant role in ride-hailing, whereas delays are ignored in our theory. In the following subsection, we summarize a numerical investigation of ride-hailing focusing on entry and assignment controls only (a full description is provided in Appendix~\ref{appen:subsec:numerics-ridehailing}). We find that MBP performs well despite the presence of travel delays. In order to address the case where the available supply is scarce, we heuristically adapt MBP to incorporate the Little's law constraint (Section~\ref{subsubsec:supply-aware-MBP}). 
	%
	%
	%
	

\subsection{Numerical investigation of the application to {ride-hailing}}
\label{subsec:numerics-ridehailing}
{In Section \ref{subsubsec:supply-aware-MBP}, we examine the performance of MBP policy when there are travel delays using numerical experiments.}
The simulation environment we study is inspired by ride-hailing, and leverages demand estimates deduced from NYC yellow cab data \citep{buchholz2021spatial} and travel times from Google Maps. 
In Section \ref{subsec:sims-large-network-main-paper}, we provide the summary of simulations that study the performance of MBP policy in large networks.
In the interest of space, we provide only the key findings of our simulations here and defer a full description of the simulation environment and various technical details 
to Appendix~\ref{appen:subsec:numerics-ridehailing}.

\subsubsection{Travel Delays and the Supply-Aware MBP Policy}
\label{subsubsec:supply-aware-MBP}
In the following, we investigate the performance of MBP policy when there are travel delays.
Similar to our main setting\footnote{The correspondence between our (ride-hailing) simulation setting and the JEA setting is as follows: In the ride-hailing setting, the type of a demand is its origin-destination pair, i.e. $\cT=V\times V$. For type $(j,k)$ demand, its supply neighborhood is the neighboring locations of $j$, which we denote by (with a slight abuse of notation) $\cP(j)$. We do not consider flexible drop-off, therefore $\cD(j,k)=\{k\}$.}$^{,}$\footnote{In our simulations, we focus on the special case where demand is stationary instead of time-varying, even though MBP policies are expected to work well if demand varies slowly over time. We make this choice because it allows us to compare performance against that of the policy proposed in \cite{banerjee2016pricing} for the stationary demand setting.} in Section~\ref{sec:model}, we allow the platform two control levers: entry control and assignment/dispatch control.
Our theoretical model made the simplifying assumption that pickup and service of demand are \textit{instantaneous}.
We relax this assumption in our numerical experiments by adding realistic travel times. We retain our simplifying assumption that drivers do not relocate in the absence of a passenger.
We consider the following two cases:
\begin{compactenum}[(1),wide, labelwidth=!, labelindent=0pt]
	\item \textit{Excess supply. }The number of cars in the system is slightly (5\%) above the ``fluid requirement'' (see Appendix~\ref{subsec:sims_setup} for details on the ``fluid requirement'') to achieve the value 
of the static planning problem.
	%
	\item \textit{Scarce supply. }The number of cars fall short (by 25\%) of the ``fluid requirement'', i.e., there are not enough cars to realize the optimal solution of static planning problem \eqref{eq:fluid_obj}-\eqref{eq:fluid_dmd_constr} under instantaneous relocation (even if we ignore stochasticity).
\end{compactenum}
We compare our MBP policy to three state-of-the-art policies in literature: the fluid-based policy in \cite{banerjee2016pricing}, the Utility-Delay Optimal Algorithm (UDOA) in \cite{neely2006super}, and the Deficit MaxWeight (DMW) policy in \cite{jiang2009stable}. We note that the UDOA policy is in fact a member of the  MBP family of policies, with exponential congestion function $f(q)=  \omega\cdot (e^{\omega (q-q_0)} - e^{\omega (q_0-q)})$ for suitable $\omega,q_0>0$. 
See Appendix \ref{appen:subsec:numerics-ridehailing} for a detailed description of these benchmark policies. 
\smallskip

{\bf Summary of findings.}
We make a natural modification of the MBP policy (with congestion function \eqref{eq:inv_sqrt_mirror_map}) to account for finite travel times; specifically, we employ a \emph{supply-aware MBP} policy which estimates and uses a shadow price of keeping a vehicle (supply unit) occupied for one unit of time.\footnote{To make the comparison fair, we modify the UDOA and DMW policies using the same heuristic approach, as the original UDOA and DMW policies do not take into account the travel delays.} This policy is described at the end of this section.

\emph{The excess supply case.} We simulate the (stationary) system from 8 a.m. to 12 p.m. with $100$ randomly generated initial states.\footnote{We first uniformly sample $100$ points from the simplex $\{\bq:\sum_{i \in V}q_i = K\}$, which are used as the system's initial states at 6 a.m. (note that all the cars are free).
	Then we ``warm-up'' the system by employing the static policy from 6 a.m. to 8 a.m., assuming the demand arrival process during this period to be stationary (with the average demand arrival rate during this period as mean).
	Finally, we use the system's states at 8 a.m. as the initial states.}
The simulation results on performance are shown in Figure \ref{fig:payoff_excess}.
The results show that MBP policy significantly outperforms both the DMW policy and the fluid-based policy, and consistently outperforms the UDOA policy: the average payoff under MBP over 4 hours is  about $105\%$ of $\Wspp$ (here $\Wspp$ is again an upper bound on the steady state performance\footnote{$\Wspp$ is still an upper bound on stationary performance when pickup and service times are included in our model. However, in this case a transient upper bound 
	is difficult to derive. As a result, we use the ratio of average per period payoff to $\Wspp$ as a performance measure, with the understanding that it may exceed 1 at early times.}),
while UDOA, DMW, and the fluid-based policy achieve $100\%$, $81\%$, and $68\%$ of $\Wspp$, respectively. 
The performance of the static policy converges very slowly to $\Wspp$, leading to poor transient performance.\footnote{For example, the average payoff generated by static policy in the last hour of a $20$-hour period is $0.96\Wspp$.}
The performance of the DMW policy deteriorates over time because the ``fake packets'' it generates accumulate in the system.

\begin{figure}
	\centering
	\begin{minipage}{0.47\textwidth}
		\centering
		\includegraphics[width=0.9\textwidth]{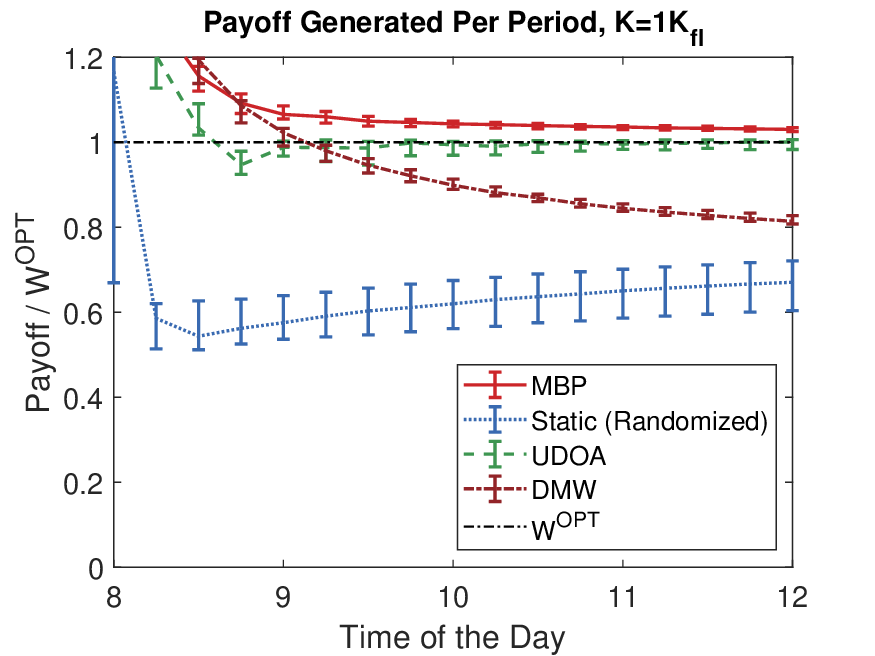} 
		\caption{Per period payoff under the MBP, UDOA, DMW, and the fluid-based policy, relative to $\Wspp$.
		For each data point we run 100 experiments; the error bars represent the 90\% confidence intervals.
		}	
		\label{fig:payoff_excess}
	\end{minipage}\hfill
	\begin{minipage}{0.47\textwidth}
		\centering
		\includegraphics[width=0.9\textwidth]{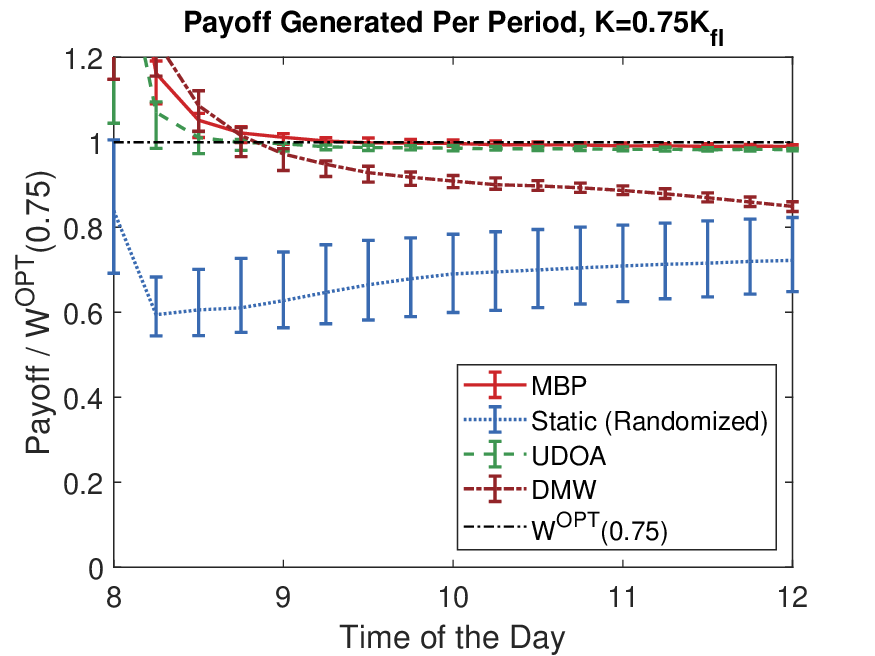} 
		\caption{Per period payoff under the (modified) MBP, (modified) UDOA, (modified) DMW, and the fluid-based policy, relative to $\Wspp({0.75})$, the value of SPP along with constraint \eqref{eq:supply_constraint_matrix} for $K = 0.75 K_{{\rm \tiny fl}}$.
		For each data point we run 100 experiments; the error bars represent the 90\% confidence intervals.
		}
		\label{fig:payoff_scarce}
	\end{minipage}
\end{figure}

\emph{The scarce supply case.} In the scarce supply case, e.g., $K=0.75 K_{\textup{fl}}$, no policy can achieve a stationary performance of $\Wspp$; rather we have a steady-state upper bound of $\Wspp(0.75) \approx 
0.86 \Wspp$ for this $K$, where $\Wspp(0.75)$ is the value of the problem given by \eqref{eq:fluid_obj}-\eqref{eq:fluid_dmd_constr} together with the supply constraint \eqref{eq:supply_constraint_matrix} below. 
Figure \ref{fig:payoff_scarce} shows that the MBP policy again vastly outperforms the DMW policy and the fluid-based policy, and has similar performance to the UDOA policy in the scarce supply case. 
MBP generates average per period payoff that is $99\%$ of the steady-state upper bound  
over 4 hours, while the UDOA, DMW, and the fluid-based policy achieve $98\%$, $85\%$, and $74\%$ resp. of the steady-state upper bound over the same period. Reassuringly, the mean value of $v(t)$ in our simulations of supply-aware MBP is within 10\% of the optimal dual variable to the tightened supply constraint \eqref{eq:tightened_supply_constraint_matrix} in the supply-aware SPP (\eqref{eq:fluid_obj}-\eqref{eq:fluid_dmd_constr} along with \eqref{eq:tightened_supply_constraint_matrix}); both values are close to $0.50$.
Again, we observe that the average performance of static policy improves (slowly) as the time horizon gets longer, while the performance of DMW deteriorates.

\textbf{Supply-aware MBP policy.} In order to heuristically modify MBP to account for travel times, we begin by observing that the SPP must now include a Little's law constraint. (The same observation was previously leveraged by \citealt{braverman2016empty} and \citealt*{banerjee2016pricing} to formally handle travel times, albeit under the assumption that travel times are i.i.d. exponentially distributed.) Our heuristic modification of MBP will maintain an estimate of the shadow price corresponding to the Little's law constraint, and penalize rides appropriately.


Applying Little's Law, if the optimal solution {$\bz^*$ of the SPP 
(here we work with the special case where $\bphi$ does not depend on $t$) 
is realized as the average long run assignment, the mean number of cars which are occupied in picking up or transporting customers is
$
\sum_{j,k\in V}\sum_{i \in \cP(j)}D_{ijk}\cdot z^*_{ijk}\, ,
$ for $D_{ijk} \triangleq \tilde{D}_{ij} + \hat{D}_{jk}$,
where $\tilde{D}_{ij}$ is the pickup time from $i$ to $j$ and $\hat{D}_{jk}$ is the travel time from $j$ to $k$.
We augment the SPP with the additional supply constraint
\begin{align}\label{eq:supply_constraint_matrix}
\sum_{j,k\in V}\sum_{i \in \cP(j)}D_{ijk}\cdot z_{ijk}\leq K\, ,
\end{align}
which simply encodes that the average number of cars occupied at any time cannot exceed $K$. 
We propose and {test in the simulation} the following heuristic policy inspired by MBP, that additionally incorporates the supply constraint. We call it \emph{supply-aware MBP}. Given a demand arrival with origin $j$ and destination $k$, the policy makes its decision as per:
\begin{align*}
&i^* \leftarrow \arg \max_{i\in\cP(j)}\left\{ w_{ijk} +  f(\bar{q}_i[t]) - f(\bar{q}_k[t])  - v[t] D_{ijk} \right \}\\\
&\textbf{If}\ w_{i^*jk} + f(\bar{q}_{i^*}[t]) - f(\bar{q}_{k}[t]) - v[t] D_{i^*jk} \geq 0\  \textup{\textbf{and}} \ q_{i^*}[t] > 0\, , \textup{dispatch from $i^*$, \textbf{else} Drop,}
\end{align*}
We define the tightened supply constraint as
\begin{align}\label{eq:tightened_supply_constraint_matrix}
\sum_{j,k\in V}\sum_{i \in \cP(j)}D_{ijk} \cdot z_{ijk}\leq 0.95K\, ,
\end{align}
where the coefficient of $K$ is the flexible ``utilization'' parameter, that we have set at 0.95, meaning that we are aiming to keep $5\%$ vehicles free on average, systemwide.\footnote{Keeping a small fraction of vehicles free is helpful in managing the stochasticity in the system. Note that the present paper does not study how to systematically choose the utilization parameter.}\pqa{Need to specify which congestion function did we use.}
Here $v[t]$ is the current estimate of the shadow price for the ``tightened'' version of supply constraint \eqref{eq:tightened_supply_constraint_matrix}.
{We use the congestion function given in \eqref{eq:inv_sqrt_mirror_map}, i.e., $f_j(\bar{q}_j)=\sqrt{m} \cdot\bar{q}_j^{-1/2}$,  in our numerical simulations. An important detail here is that the queue lengths are normalized by the estimated number of \emph{free} cars $K - 0.95K = 0.05K$ instead of $K$. 
}
We update $v[t]$ as
\begin{align*}
&v[t+1]
=
\Bigg[v[t] + \\
&\frac{1}
{K} \left (
\sum_{j,k\in V}\sum_{i \in \cP(j)}D_{ijk}\cdot \mathds{1}\{(o[t],d[t])=(j,k),\textup{MBP would dispatch from $i$}\}- 0.95K\right ) \Bigg]^+\, .
\end{align*}
An iteration of supply-aware MBP is equivalent to executing a (dual) stochastic mirror descent step on the supply-aware SPP \eqref{eq:fluid_obj}-\eqref{eq:fluid_dmd_constr} along with \eqref{eq:tightened_supply_constraint_matrix}. 
\yka{The previous sentence was not clear to me.}

\yka{``We use the congestion function $f(q_i) = c \log q_i$ in our numerical simulations, with $c= \max_{i,j,k\in V}w_{ijk}$.'' I suggest to delete this sentence. your summary hasn't said how you specify $w$s, so this detail is not helpful.}



\subsubsection{MBP Policy in Large Networks}\label{subsec:sims-large-network-main-paper}
Recall that in Corollary \ref{cor:stationary-main-result}, the steady-state optimality gap of MBP is shown to be $O\left(\frac{m^2}{K}\right)$ for congestion function \eqref{eq:inv_sqrt_mirror_map}.
Compared with the $O\left(\frac{m}{K}\right)$ bound for the fluid-based policy proved in \cite{banerjee2016pricing}, our bound for MBP has the same dependence on $K$ but worse dependence on $m$.
A natural question is whether the worse dependence on $m$ reflects poorer performance, or if it is a proof artifact.
We conduct numerical experiments in Appendix \ref{subsec:sims-large-network} to study this question.

\textbf{Summary of findings.} We construct a family of instances that has the same total demand rate, but different network sizes $m$.
We compare the performance of our MBP policy with the fluid-based policy in \cite{banerjee2016pricing} for different values of fleet size $K$ and network size (i.e., number of locations) $m$. 
The results demonstrate that MBP consistently outperforms the fluid-based policy in steady state across different choices of $m$ and $K$.
Also, the steady-state suboptimality of MBP appears to scale as $m/K$ (and not $m^2/K$, which was the scaling of our formal upper bound on the optimality gap). 

%
%

\section{Discussion}\label{sec:discuss}
In this paper we considered the payoff maximizing dynamic control of a closed network of resources.
We proposed a novel family of policies called Mirror Backpressure (MBP), which generalize the celebrated backpressure policy such that it executes mirror descent with the desired mirror map, while retaining the simplicity of backpressure.
The MBP policy overcomes the challenge stemming from the no-underflow constraint and it does not require any knowledge of demand arrival rates.
We proved that MBP achieves good transient performance for demand arrival rates which are stationary or vary slowly over time, losing at most $O\left(\frac{K}{T}+\frac{1}{K} +\sqrt{\eta K}\right)$ payoff per customer, where $K$ is the number of supply units, $T$ is the number of customers over the horizon of interest, and $\eta$ is the average rate of change in demand arrival rates per customer arrival. We considered a variety of control levers: entry control, assignment control and pricing, and allowed for finite buffer sizes. We discussed the application of our results to the control of shared transporation systems and scrip systems.

One natural question is whether our bounds capture the right scaling of the per customer optimality gap of MBP with $K$, $T$ and $\eta$, relative to the best policy which is given exact demand arrival rates and horizon length $T$ in advance. 
Consider the joint entry-assignment setting (Section~\ref{sec:model}).  It is not hard to construct examples showing that each of the terms in our bound is unavoidable: a $1/K$ optimality gap arises in steady state (under stationary demand arrival rates) for instance in a two-node entry-control-only example where the two demand arrival rates are exactly equal to each other, the $K/T$ term arises because over a finite horizon the flow balance constraints need not be satisfied exactly and MBP does not exploit this flexibility fully, and the $\sqrt{\eta K}$ term arises in examples where demand arrival rates oscillate (with a period of order $\sqrt{K/\eta}$) but MBP does not take full advantage of the flexibility to allow queue lengths to oscillate alongside. We omit these examples in the interest of space.

We point out some interesting directions that emerge from our work:
\begin{compactenum}[1.,wide, labelwidth=!, labelindent=0pt]
    \yka{any thoughts on whether this may be possible and how?}
	%
    \item \textit{Improved performance via ``centering'' MBP based on demand arrival rates.}
    If the optimal shadow prices $\by^*$ are known (or learned by learning $\bphi$ via observing demand), we can modify the congestion function to 
    $\tilde{f}_j(\bar{q}_j)= y_j^* + f(\bar{q}_j)$.
    For the resulting ``centered'' MBP policy, based on the result of \cite{huang2009delay} and the convergence of mirror descent, we are optimistic that the steady state regret will decay exponentially in $K$. 

    \item Another promising direction is to pursue the viewpoint that there is an MBP policy which (very nearly) maximizes the steady state rate of payoff generation, specifically for the choice of congestion functions $f_j(\cdot)$ that are the discrete derivatives of the relative value function $F(\nq)$ (for the average payoff maximization dynamic programming problem) with respect to $\bar{q}_j$; see Chapter 7.4 of \cite{bertsekas1995dynamic} for background on dynamic programming. Thus, estimates of the relative value function $F(\nq)$ can guide the choice of congestion function.
\end{compactenum}

\setstretch{1.0}

\bibliographystyle{abbrvnat}
\bibliography{bib_mirror_bp}

%

\newpage


\setcounter{page}{1}
\appendix

\begin{center}
	{\LARGE {\textbf{Appendix to ``Blind Dynamic Resource Allocation in Closed Networks via Mirror Backpressure"}}}
\end{center}
\smallskip


The appendix is organized as follows.
\begin{itemize}
  %
  %
      %
  \item In Appendix~\ref{appen:proof-of-fluid-ub} we prove Proposition \ref{prop:fluid_ub}.
  \item In Appendix~\ref{append:key-lemmas} we prove the lemmas in Section \ref{subsec:policy-gap}, including Lemma \ref{lem:one_step_lyap}, Lemma \ref{lem:dual_subopt}, and Lemma \ref{lem:bound_rhs} (and its generalization).
      %
  \item In Appendix~\ref{append:proof-main-theorems} we prove Proposition \ref{lem:SPPavg-vs-avgofSPPt}.
      %
        %
  \item In Appendix~\ref{append:proof-main-regret} we prove Theorem \ref{thm:main_regret}, Theorem \ref{thm:main-result-finite-buffer} and their generalization that holds for a large class of congestion functions which satisfy a certain growth condition.
  \item In Appendix~\ref{appen:jpa} we outline the proof for the JPA setting.
  \item In Appendix~\ref{appen:subsec:numerics-ridehailing}, we provide further details of the simulation setting.
  \item In Appendix~\ref{sec:scrip-systems} we describe the application to scrip systems.
  \item In Appendix~\ref{append:additional-proofs} we provide additional proofs and examples that supports some of the arguments made in the paper.
  \item In Appendix~\ref{appen:MBP-as-MD} we show that MBP executes stochastic mirror descent on the controller's fluid limit optimization problem.
  %
\end{itemize}

\section{Proof of Proposition \ref{prop:fluid_ub}}\label{appen:proof-of-fluid-ub}
For easier reference, we rewrite the static planning problem (SPP) for JEA below:
\begin{align}
	\hspace{-0.3cm}\textup{SPP($\bphi$)}:  \textup{maximize}_{\bz} \; &
	\sum_{\tau\in \cT, j\in\mathcal{P}(\tau),k\in\mathcal{D}(\tau)} 
	w_{j\tau k} z_{j\tau k}
	\label{eq:fluid_linear_obj-z}\\
	\textup{s.t.}\; &
	\sum_{\tau\in \cT,j\in\mathcal{P}(\tau),k\in\mathcal{D}(\tau)} z_{j\tau k}(\boe_j - \boe_k) =\bzero
	\hspace{2.9cm} \textup{(flow balance)}\label{eq:fluid_linear_flow_bal-z}\\
	&\sum_{j\in \mathcal{P}(\tau), k\in \mathcal{D}(\tau)}z_{j\tau k}\leq \phi_\tau \, , \ z_{j\tau k} \geq 0  \quad \forall j,k \in V, \tau\in\cT
	\hspace{0.1cm}\textup{(demand constraint)}.\label{eq:fluid_linear_dmd_constr-z}
\end{align}
The variable $z_{j\tau k}$ can be interpreted as the flow of demand type $\tau$ being served by pickup location $j$ and dropoff location $k$.
{(Note that our LP formulation here has a cosmetic difference from that in \eqref{eq:fluid_obj}-\eqref{eq:fluid_dmd_constr}: here we find that it simplifies our analysis to use the \emph{flows} $z_{j\tau k}$ as the LP variables instead of using the \emph{fractions} $x_{j \tau k}$ of demand of type $\tau$ served by pickup location $j$ and dropoff location $k$ as the variables. The correspondence is simply $z_{j\tau k} \leftrightarrow \phi_\tau x_{j\tau k}$.)}

The idea behind Proposition~\ref{prop:fluid_ub} is as follows.
As is typical in such settings, $W^{\aspp}$ is an upper bound on the payoff if the flow constraints are satisfied in expectation.
However, since the flow constraints can be slightly violated in the finite horizon setting under consideration, we obtain an upper bound by slightly relaxing the flow constraint \eqref{eq:fluid_linear_flow_bal-z} in the SPP($\bar{\bphi}$) to
\begin{align}
	\left|
	\mathbf{1}_S^{\T}
	\left(\sum_{\tau\in \cT,j\in\cP(\tau),k\in\cD(\tau)}
	z_{j\tau k}(\boe_j - \boe_k)\right)
	\right|
	\leq \frac{K}{T} \qquad \forall \; S \subseteq V \, ,
	\label{eq:approx_flow_balance_jea-z}
\end{align}
where $\mathbf{1}_S$ is the vector with $1$s at nodes in $S$ and $0$s at all other nodes.

We establish two key lemmas to facilitate the proof of Proposition~\ref{prop:fluid_ub}. The first lemma (Lemma \ref{lem:fluid_upper_bound_jea-timevar}) shows that the expected payoff cannot exceed the value of the finite horizon demand-averaged SPP($\bar{\bphi}$).

\begin{lem}\label{lem:fluid_upper_bound_jea-timevar}
	For any horizon $T< \infty$, any $K$ and any starting state $\bq[0]$, the expected payoff generated by any feasible joint entry-assignment control policy $\pi$ is upper bounded by the value of the linear program defined by SPP($\bar{\bphi}$) with the flow constraint \eqref{eq:fluid_linear_flow_bal-z} replaced by \eqref{eq:approx_flow_balance_jea-z}.
\end{lem}
\begin{proof}
	Let $\pi$ be any feasible policy.
	For each $\tau\in\cT$ and $j\in\cP(\tau),k\in \cD(\tau)$, define
	\begin{align*}
		\widebar{z}_{j \tau k} \triangleq \frac{1}{T}\sum_{t=0}^{T-1}
		\mathbb{E}[x_{j\tau k}[t] \ind\{ \tau[t]=\tau \} ]\, .
	\end{align*}
	In words, $\widebar{z}_{j \tau k}$ is the average flow over $1 \leq t \leq T$ of the demand type $\tau$ being served by pickup location $j$ and dropoff location $k$. Since for each $t$, $z_{j \tau k}[t] \triangleq  \mathbb{E}[\ind\{\tau[t] = \tau\} x_{j\tau k}]$ satisfies the period-specific demand constraint \eqref{eq:fluid_linear_dmd_constr-z} for all $\tau \in \cT, j \in \cP[\tau], k \in \cD(\tau)$, the averaged constraints \eqref{eq:fluid_linear_dmd_constr-z} must hold for $\widebar{\bz}$.

	We can write the expected per-period payoff collected in the first $T$ periods as:
	\begin{align*}
		W^{\pi}_T
		&=\;
		\frac{1}{T}\sum_{t=0}^{T-1}\mathbb{E}
		\left[
		\sum_{\tau \in \cT,j\in\cP(\tau),k\in\cD(\tau)} w_{j\tau k} x_{j\tau k}[t] \ind \{\tau[t]=\tau\}
		\right]\\
		&=\;
		\sum_{\tau \in \cT,j\in\cP(\tau),k\in\cD(\tau)}
		w_{j\tau k}
		\widebar{z}_{j\tau k}\, ,
	\end{align*}
	where we only used linearity of expectation. In words, the expected per-period payoff is the objective \eqref{eq:fluid_linear_obj-z} evaluated at $\widebar{\bz}$.
	Similarly, for the time-average of the change of queue length we have:
	\begin{align*}
		\frac{1}{T}\mathbb{E}[\bq[T] - \bq[0]]
		=
		\sum_{\tau\in\cT,j\in\cP(\tau),k\in\cD(\tau)}
		\widebar{z}_{j\tau k}(\boe_j - \boe_k)\, ,
	\end{align*}
	which implies that $\widebar{\bz}$ satisfies the approximate flow constraints \eqref{eq:approx_flow_balance_jea-z} since $|\sum_{j \in S}q_j[T] - q_j[0]| \leq K$ for all $S \subset V$.
	(Because there are only $K$ resources circulating in the system, the net outflow from any subset of nodes $S\subseteq V$ should not exceed $K$ in magnitude.)

	We have shown that $\widebar{\bz}$ is feasible for the given linear program with constraints  \eqref{eq:approx_flow_balance_jea-z} and \eqref{eq:fluid_linear_dmd_constr-z}, and the expected payoff earned $W^{\pi}_{T}$  is identical to objective \eqref{eq:fluid_linear_obj-z} evaluated at $\widebar{\bz}$. It follows that $W^{\pi}_{T}$ is upper bounded by the value of the optimization problem defined by \eqref{eq:fluid_linear_obj-z}, \eqref{eq:approx_flow_balance_jea-z} and \eqref{eq:fluid_linear_dmd_constr-z} regardless of the initial configuration $\bq[0]$.
	This concludes the proof.
\end{proof}

In order to facilitate the second key lemma, we first prove a supporting lemma (Lemma \ref{lem:flow-decomp-jea-z}).
We call $\bz$ a (directed)  \emph{acyclic flow} if there is no (directed) cycle
\begin{align*}
	\mathcal{C}=\big(
	\,(j_1,\tau_1,j_2),(j_2,\tau_2,j_3),\cdots,(j_{s},\tau_s,j_{s+1}=j_1)\, \big) \, , \qquad \textup{where } j_r\in V \textup{ and }\tau_r \in \cT \textup{ for } r=1,2,\cdots,s\, ,
\end{align*}
such that
\begin{align*}
	z_{j_r,\tau_r,j_{r+1}}>0 \qquad \textup{for all }r=1,\cdots,s \, .
\end{align*}
In words, there is no cycle $\cC$ such that there is a positive flow along $\cC$.
\begin{lem}\label{lem:flow-decomp-jea-z}
	Any feasible solution $\bz^F$ of SPP($\bar{\bphi}$) satisfying approximate flow balance \eqref{eq:approx_flow_balance_jea-z} and the average demand constraint \eqref{eq:fluid_linear_dmd_constr-z} (with $\bphi=\bar{\bphi}$) can be decomposed as
	\begin{align}
		\bzF = \bzS + \bzDAG \, ,
	\end{align}
	where $\bzS$ is a feasible solution for SPP($\bar{\bphi}$) satisfying exact flow balance \eqref{eq:fluid_linear_flow_bal-z} and \eqref{eq:fluid_linear_dmd_constr-z} (with $\bphi=\bar{\bphi}$), and $\bzDAG$ is an acyclic flow satisfying \eqref{eq:approx_flow_balance_jea-z} and \eqref{eq:fluid_linear_dmd_constr-z} (with $\bphi=\bar{\bphi}$).
\end{lem}

\begin{proof}
	The existence of such a decomposition can be established using a standard flow decomposition argument \citep[see, e.g.,][]{williamson2019network}: Start with $\bzS = \mathbf{0}$ and $\bzDAG = \bzF$. Then, iteratively, if $\bzDAG$ includes a cycle $\cC$ with a positive flow along $\cC$ as above, move a flow of $u(\cC) \triangleq \min_{1\leq r \leq s} z_{j_r,\tau_r,j_{r+1}}$ along $\cC$ from $\bzDAG$ to $\bzS$, via the updates
	\begin{align*}
		\zS_{j_r,\tau_r, j_{r+1}} &\leftarrow \zS_{j_r,\tau_r, j_{r+1} } + u(\cC)\, ,&
		\zDAG_{ j_r,\tau_r, j_{r+1} } &\leftarrow \zDAG_{ j_r,\tau_r, j_{r+1} } - u(\cC)\, ,
	\end{align*}
	for all $r= 1, 2, \ldots, s$.
	This iterative process maintains the following invariants which hold at the end of each iteration:
	\begin{compactitem}[leftmargin=*]
		\item $\bzS$ remains fesible for SPP($\bar{\bphi}$), in particular, it satisfies flow balance \eqref{eq:fluid_linear_flow_bal-z}.
		\item $\bzF = \bzS + \bzDAG$ remains true.
		\item It remains true that
		\begin{align*}
			\sum_{\tau\in \cT,j\in\cP(\tau),k\in\cD(\tau)} \zDAG_{j\tau k}(\boe_j - \boe_k)
			=\
			\sum_{\tau\in \cT,j\in\cP(\tau),k\in\cD(\tau)}
			\zF_{j\tau k}(\boe_j - \boe_k)\, .
		\end{align*}
		i.e., $\bzDAG$ has the same net inflow/outflow from each supply node as $\bzF$.  In particular, $\bzDAG$ satisfies approximate flow balance \eqref{eq:approx_flow_balance_jea-z}.
	\end{compactitem}
	Moreover, the iterative process progresses monotonically: Observe that $\bzS$ coordinate-wise (weakly) increases monotonically, whereas $\bzDAG$ coordinate-wise (weakly) decreases monotonically (but preserves $\bzDAG \geq \mathbf{0}$).
	Since we also know that $\bzS$ is bounded, it follows that this iterative process converges. Moreover, in the limit
	it must be that there is no remaining cycle with positive flow in $\bzDAG$ (else we observe a contradiction with the fact that the process has converged). Hence, $\bzS$ and $\bzDAG$ at the end of the process provide the claimed decomposition.
\end{proof}

Using this supporting lemma, we now establish the second key lemma which shows that the value of SPP($\bar{\bphi}$) with approximate flow balance constraints \eqref{eq:approx_flow_balance_jea-z} cannot be much larger than the value of SPP($\bar{\bphi}$) which imposes exact flow balance constraints \eqref{eq:fluid_linear_flow_bal-z}.
\begin{lem}\label{lem:WoptT-ub-Wopt-jea-z}
	The value of the linear program defined by \eqref{eq:fluid_linear_obj-z}, the approximate flow balance constraints \eqref{eq:approx_flow_balance_jea-z} and time-averaged demand constraints \eqref{eq:fluid_linear_dmd_constr-z} is bounded above by
	\begin{align*}
		W^\aspp + m\frac{K}{T} \, .
	\end{align*}
	where $W^\aspp$ is the value of the linear program $\aspp$ which imposes exact flow balance constraints \eqref{eq:fluid_linear_flow_bal-z}.
\end{lem}

\begin{proof}
	We appeal to the decomposition from Lemma~\ref{lem:flow-decomp-jea-z} to decompose any feasible solution $\bzF$ to the finite horizon fluid problem as
	\begin{align*}
		\bzF = \bzS + \bzDAG \, ,
	\end{align*}
	where $\bzS$ is feasible for $\aspp$ and $\bzDAG$ is a directed acyclic flow that satisfies approximate flow balance \eqref{eq:approx_flow_balance_jea-z} and the averaged demand constraints \eqref{eq:fluid_linear_dmd_constr-z}.
	Hence, the objective \eqref{eq:fluid_linear_obj-z} can be written as the sum of two terms
	\begin{align}\label{eq:obj-decomp-dag-jea}
		\sum_{\tau \in \cT,j\in\cP(\tau),k\in\cD(\tau)}
		w_{j\tau k}
		 \zF_{j\tau k}
		= \
		\sum_{\tau \in \cT , j\in\cP(\tau),k\in\cD(\tau)}
		w_{j\tau k}
		(\zS_{j\tau k}
		+
		\zDAG_{j\tau k})\, ,
	\end{align}
	and each of the terms can be bounded from above.
	By definition of $W^\aspp$ we know that
	\begin{align*}
		\sum_{\tau \in \cT,j\in\cP(\tau),k\in\cD(\tau)}
		w_{j\tau k}
		\zS_{j\tau k}
		\leq
		W^\aspp\, .
	\end{align*}
	We will now show that
	\begin{align*}
		\sum_{\tau \in \cT,j\in\cP(\tau),k\in\cD(\tau)}
		w_{j\tau k}
		\zDAG_{j\tau k}
		\leq (m-1)\frac{K}{T} <
		m\frac{K}{T}\, .
	\end{align*}
	The lemma will follow, since this will imply an upper bound of $W^\aspp + m\frac{K}{T}$ on the objective for any $\bzF$ satisfying \eqref{eq:approx_flow_balance_jea-z} and \eqref{eq:fluid_linear_dmd_constr-z}.
	
	Consider $\bzDAG$. Since it is an acyclic flow, there is an ordering $(j_1, j_2, \ldots, j_m)$ of the nodes in $V$ such that all positive flows move supply from an earlier node to a later node in this ordering. More precisely, it holds that for any $\tau\in\cT$,
	\begin{align}
		\zDAG_{j_l,\tau, j_r} = 0 \qquad \forall \; l > r \ \textup{ s.t. } j_l \in \cP(\tau), j_r \in \cD(\tau)\, .
		\label{eq:DAG-order-jea}
	\end{align}
	Now consider the subsets $A_{\ell} \triangleq \{j_1, j_2, \ldots, j_{\ell}\} \subset V$ for $\ell = 1, 2, \ldots, m-1$.
	Note that from \eqref{eq:DAG-order-jea}, $\bzDAG$ does not move any supply from $V \backslash A_{\ell}$ to $A_{\ell}$. Hence we have
	\begin{align}
		\mathbf{1}_{A_{\ell}}^{\T}
		\left(\sum_{\tau\in\cT,j\in\cP(\tau),k\in\cD(\tau)}
		\zDAG_{j\tau k}(\boe_j - \boe_k)\right)
		= &\sum_{\tau \in \cT, j\in\cP(\tau)\cap A_{\ell},k\in\cD(\tau)\cap (V\backslash A_{\ell})} \zDAG_{j\tau k} \; \nonumber\\
		\leq &\; \frac{K}{T}
		\qquad \forall \; l = 1, 2, \ldots, m-1\, , \label{eq:cut-bound-z}
	\end{align}
	We made use of \eqref{eq:approx_flow_balance_jea-z} to obtain the upper bound.
	Further, note that for each $\zDAG_{j_l,\tau, j_r}$ with $l < r$, the term $ \zDAG_{j_l,\tau, j_r}$ is part of the above sum for $\ell = l$. Motivated by this observation, we bound the expected payoff of $\bzDAG$ by first using our assumption $\max_{j,k\in V,\tau\in\cT}|w_{j\tau k}|\leq 1$  to bound the payoff by the sum of $\zDAG$s (the first inequality below), and then bounding the sum of $\zDAG$s by ``allocating'' $\zDAG_{j_l,\tau, j_r}$ to the left-hand side of \eqref{eq:cut-bound-z} with $\ell = l$ and summing over $\ell$  (the second inequality below):
	\begin{align*}
		\sum_{\tau\in \cT,j\in\cP(\tau),k\in\cD(\tau)} w_{j\tau k} \zDAG_{j\tau k} 
		\leq\; &
		\sum_{\tau\in\cT, j\in\cP(\tau),k\in\cD(\tau)}
		\zDAG_{j\tau k}\\
		\leq\; &
		\sum_{1 \leq \ell < m}\
		\sum_{\tau \in \cT, j\in\cP(\tau)\cap A_{\ell},k\in\cD(\tau)\cap (V\backslash A_{\ell})} \zDAG_{j\tau k}\\
		\leq\; & (m-1) \frac{K}{T}\, .
	\end{align*}
	The last inequality uses \eqref{eq:cut-bound-z} summed over $\ell$. This completes the proof.
\end{proof}

\begin{proof}[Proof of Proposition~\ref{prop:fluid_ub}]
	The proposition follows immediately from Lemmas~\ref{lem:fluid_upper_bound_jea-timevar} and \ref{lem:WoptT-ub-Wopt-jea-z}.
\end{proof}

\section{Bounding the Policy Gap: Proofs in Section \ref{subsec:policy-gap}}
\label{append:key-lemmas}

\subsection{Proof of Lemma \ref{lem:one_step_lyap}}\label{append:proof-one-step-lyapunov}
For ease of reference, we repeat Lemma \ref{lem:one_step_lyap} below.
\onesteplyap*
%

\begin{proof}
	For congestion function $f_j(\bar{q}_j)$ that are strictly increasing and continuous for each $j$, we consider the Lyapunov function $F(\nq)$ which is the antiderivative of $\bof(\nq)$.
	The Bregman divergence associated with $\bof(\nq)$ is defined as:
	\begin{align}\label{eq:bregman}
	D_{ F }(\nq_1,\nq_2)
	=
	F (\nq_1) - F (\nq_2) - \langle \bof(\nq_1), \nq_1 - \nq_2 \rangle\, .
	\end{align}
	Plugging $\nq_1 = \nq[t+1]$, $\nq_2 = \nq[t]$ into \eqref{eq:bregman} and rearranging the terms, we have:
	\begin{align*}
	F (\nq[t+1]) - F (\nq[t])
	=\
	\langle
	\bof (\nq[t]), \nq[t+1]-\nq[t] \rangle
	+
	D_{ F }(\nq[t+1],\nq[t])\, .
	\end{align*}
	Subtracting $\frac{1}{\tK}\sum_{\tau\in\cT}\phi_{\tau}^t\sum_{j\in\cP(\tau),k\in\cD(\tau)}w_{j\tau k}\cdot x_{j\tau k}[t]$ on both sides and taking conditional expectation given $\nq[t]$, we have:
	\begin{align}
	& \mathbb{E}[F (\nq[t+1])|\nq[t]] - F (\nq[t]) -
	\frac{1}{\tK}\sum_{\tau\in\cT}\phi_{\tau}^t\sum_{j\in\cP(\tau),k\in\cD(\tau)}w_{j\tau k} \mathbb{E}[x_{j\tau k}[t]|\nq[t]] \nonumber\\
	=\ &
	\underbrace{
		-\frac{1}{\tK}\sum_{\tau\in\cT}\phi_{\tau}^t\sum_{j\in\cP(\tau),k\in\cD(\tau)}w_{j\tau k} \mathbb{E}[x_{j\tau k}[t]|\nq[t]]
		+
		\langle \bof(\nq[t]), \mathbb{E}[\nq[t+1]|\nq[t]]-\nq[t] \rangle
	}_{\textup{(I)}}\nonumber\\
	&+
	\underbrace{\mathbb{E}\big[D_{ F }(\nq[t+1],\nq[t])|\nq[t]\big ]}_{\textup{(II)} }\, .\label{eq:regret_decomp}
	\end{align}
	
	Let $x_{j\tau k}^{\textup{NOM}}[t]$ be the ``nominal'' control that ignores the no-underflow constraint, i.e.
	\begin{align}
	(x_{j\tau k}^{\textup{NOM}})[t]=
	\left\{
	\begin{array}{ll}
	1 & \textup{ if } w_{j\tau k} +
	f_j(\bar{q}_j[t]) - f_k(\bar{q}_k[t])\geq 0\\
	0 & \textup{ otherwise.}
	\end{array}
	\right.
	\label{eq:nom-defn}
	\end{align}
	It immediately follows that
	\begin{align}
	(x_{j\tau k}^{\textup{MBP}})[t]
	=
	(x_{j\tau k}^{\textup{NOM}})[t]
	\cdot
	\mathds{1}\{q_j[t] > 0\} \, .
	\label{eq:nom-vbp}
	\end{align}
	
	With a slight abuse of notation, denote $\bx^{\textup{NOM}}$ as $\tilde{\bx}$, $\bx^{\textup{MBP}}$ as $\bx$.
	Rearranging the terms in (I) and plugging in \eqref{eq:nom-vbp}, we have
	\begin{align*}
	\textup{(I)} =&\
	-\frac{1}{\tK}
	\sum_{\tau\in\cT}
	\phi_{\tau}^t
	\sum_{j\in\cP(\tau),k\in\cD(\tau)}
	\left(
	w_{j\tau k}
	+
	f_j(\bar{q}_{j}[t]) - f_k(\bar{q}_{k}[t])
	\right)
	\cdot
	\mathbb{E}[x_{j\tau k}[t]|\nq[t]]\\
	=&\
	-\frac{1}{\tK}
	\sum_{\tau\in\cT}
	\phi_{\tau}^t
	\sum_{j\in\cP(\tau),k\in\cD(\tau)}
	\left(
	w_{j\tau k}
	+
	f_j(\bar{q}_{j}[t]) - f_k(\bar{q}_{k}[t])
	\right)
	\cdot
	\mathbb{E}[\tx_{j\tau k}[t]|\nq[t]]\\
	&\ +
	\frac{1}{\tK}
	\sum_{\tau\in\cT}
	\phi_{\tau}^t
	\sum_{j\in\cP(\tau),k\in\cD(\tau)}
	\left(
	w_{j\tau k}
	+
	f_j(\bar{q}_{j}[t]) - f_k(\bar{q}_{k}[t])
	\right)
	\cdot
	\mathbb{E}[\tx_{j\tau k}[t]|\nq[t]]
\cdot
\mathds{1}\left\{q_{j}[t]=0\right\}\, .
	\end{align*}
	
	By definition of the 
nominal control $\tilde{\bx}$ and \eqref{eq:partial_dual}, we have:
	\begin{align*}
	&\ -\frac{1}{\tK}
	\sum_{\tau\in\cT}
	\phi_{\tau}^t
	\sum_{j\in\cP(\tau),k\in\cD(\tau)}
	\left(
	w_{j\tau k}
	+
	f_j(\bar{q}_{j}[t]) - f_k(\bar{q}_{k}[t])
	\right)
	\cdot
	\mathbb{E}[\tx_{j\tau k}[t]|\nq[t]] \\
	=&\
	-\frac{1}{\tK}
	\sum_{\tau\in\cT}
	\phi_{\tau}^t
	\sum_{j\in\cP(\tau),k\in\cD(\tau)}
	\left(
	w_{j\tau k}
	+
	f_j(\bar{q}_{j}[t]) - f_k(\bar{q}_{k}[t])
	\right)^{+} \\
	=&\
	-\frac{1}{\tK}\cdot g^t(\bof(\nq[t]))\, .
	\end{align*}
	
	Using the fact that $\max_{j,k\in V,\tau\in\cT}|w_{j\tau k}|=1$, we have
	\begin{align*}
	&\frac{1}{\tK}
	\sum_{\tau\in\cT}
	\phi_{\tau}^t
	\sum_{j\in\cP(\tau),k\in\cD(\tau)}
	\left(
	w_{j\tau k}
	+
	f_j(\bar{q}_{j}[t]) - f_k(\bar{q}_{k}[t])
	\right)
	\cdot
\mathbb{E}[\tx_{j\tau k}[t]|\nq[t]]
\cdot
	\mathds{1}\left\{q_{j}[t]=0\right\}\\
	\leq\ &
	\frac{1}{\tK}
	\sum_{\tau\in\cT}
	\phi_{\tau}^t
	\cdot
	\mathds{1}\left\{q_{j}[t]=0,\ \exists j\right\}+
	\frac{1}{\tK}
	\sum_{\tau\in\cT}
	\phi_{\tau}^t
	\sum_{j\in\cP(\tau),k\in\cD(\tau)}
	\left(
	f_j(\bar{q}_{j}[t]) - f_k(\bar{q}_{k}[t])
	\right)^+
	\cdot
	\mathds{1}\left\{q_{j}[t]=0\right\} \\
	\leq\ &
	\frac{1}{\tK}\cdot \mathds{1}\left\{q_{j}[t]=0,\ \exists j\right\}\, .
	\end{align*}
Here the last inequality follows from the assumption that $f_j(\bar{q}_{j}[t]) \leq f_k(\bar{q}_{k}[t])$ for any $j,k\in V$ when $q_j[t]=0$. 
Note that such assumption is satisfied by any congestion function such that $f_j(\bar{q}_j)=f(\bar{q}_j)$ for all $j\in V$ where $f(\cdot)$ is a monotonically increasing function.
%
	
	Combining the above inequality and equality yields
	\begin{align*}
	\textup{(I)}
	\leq\
	-\frac{1}{\tK}\cdot g^t(\bof(\nq[t]))
	+
	\frac{1}{\tK}\cdot \mathds{1}\left\{q_{j}[t]=0,\ \exists j\right\}\, .
	\end{align*}
	
	Now we proceed to bound (II).
	By definition of Bregman divergence, (II) is the second order remainder of the Taylor series of $F(\cdot)$.
	Using the fact that $f(\cdot)$ is increasing, we have
	\begin{align*}
	\textup{(II)}
	\leq\
	&\frac{1}{2}\sum_{j\in V}
	\mathbb{E}\left[ \left(\max_{\bar{q}\in\left[\bar{q}_{j}[t]-\frac{1}{\tK},\bar{q}_{j}[t]+\frac{1}{\tK}\right]}f'_j(\bar{q})
	\right)
	(\bar{q}_j[t]-\bar{q}_j[t+1])^2 |\nq[t] \right]\\
	\leq\ &
	\ \frac{1}{2 \tK^2}\cdot \max_{j\in V} \max_{\bar{q}\in\left[\bar{q}_{j}[t]-\frac{1}{\tK},\bar{q}_{j}[t]+\frac{1}{\tK}\right]}f'_j(\bar{q})
	\, .
	\end{align*}
	
	Plugging the above bounds on (I) and (II) into \eqref{eq:regret_decomp}, we have
	\begin{align*}
	&\ \mathbb{E}[F (\nq[t+1])|\nq[t]] - F (\nq[t]) - \frac{1}{\tK}
	\mathbb{E}[v^{\textup{MBP}}[t]|\nq[t]]\\
	\leq&\
	-\frac{1}{\tK}\cdot g^t(\bof(\nq[t]))
	+
	\frac{1}{2 \tK^2}\cdot \max_{j\in V} \max_{\bar{q}\in\left[\bar{q}_{j}[t]-\frac{1}{\tK},\bar{q}_{j}[t]+\frac{1}{\tK}\right]}f'_j(\bar{q})
	+
	\frac{1}{\tK}\cdot \mathds{1}\left\{q_{j}[t]=0,\exists j\right\}\, .
	\end{align*}
	
	Rearranging the terms yields:
	\begin{align*}
	-\mathbb{E}[v^{\textup{MBP}}[t]|\nq[t]]
	\leq&\
	\tK\left(F (\nq[t]) - \mathbb{E}[F (\nq[t+1])|\nq[t]]\right)
	+
	\frac{1}{2 \tK}\cdot \max_{j\in V} \max_{\bar{q}\in\left[\bar{q}_{j}[t]-\frac{1}{\tK},\bar{q}_{j}[t]+\frac{1}{\tK}\right]}f'_j(\bar{q})\\
	&\ -g^t(\bof(\nq[t]))
	+
	 \mathds{1}\left\{q_{j}[t]=0,\exists j\right\}\, .
	\end{align*}
Adding $W^{\textup{SPP($\bphi^t$)}}$ to both sides concludes the proof.
\end{proof}

\subsection{Proof of Lemma \ref{lem:dual_subopt}}\label{append:proof-dual-subopt}
For ease of reference, we repeat Lemma \ref{lem:dual_subopt} below.
\dualsubopt*

\begin{proof}
	Consider $\by \triangleq (f_j(\bar{q}_j[t])_{j\in V}$ and order the nodes in $V$ in decreasing order of $y_j$ as $y_{i_1}\geq y_{i_2}\geq \cdots y_{i_m}$.
	We will prove the desired result by updating the dual variable $(m-1)$ times and show that the dual objective decreases by at least $\alpha(\bphi)\cdot [y_{i_{r}} - y_{i_{r+1}} - 2]^+$ after the $r$-th update.
	For $r=1$ to $r=m-1$, we repeat the following procedure: if $y_{i_r}-y_{i_{r+1}}\leq 2$, then do nothing and move on to $r+1$; if otherwise, perform the following update:
	\begin{align*}
	y_{i_k} \leftarrow\ y_{i_k} - \left(y_{i_r}-y_{i_{r+1}}- 2\right)\qquad
	\forall 1\leq k \leq r \, .
	\end{align*}
	Recall that $g(\by) = \sum_{\tau\in\cT}\phi_{\tau}\sum_{j\in\cP(\tau),k\in\cD(\tau)}[w_{j\tau k}+y_j-y_k]^+$.
	For the terms where $j,k\in \{i_1,\cdots,i_r\}$ or $j,k\in \{i_{r+1},\cdots,i_m\}$, their value are not affected by the update.
	Consider the terms where $j\in \{i_1,\cdots,i_r\}$, $k\in \{i_{r+1},\cdots,i_m\}$:
	If $y_{i_r}-y_{i_{r+1}}> 2$, then after the update, for $\tau\in\cP^{-1}(j)\cap\cD^{-1}(k)$,
	\begin{align*}
	w_{j\tau k} + y_j - y_k
	\geq\
	w_{j\tau k}
	+
	y_{i_r} - \left(y_{i_r}-y_{i_{r+1}}- 2 \right)
	-
	y_{i_{r+1}}
	\geq
	w_{j\tau k} + 2
	>
	0\, ,
	\end{align*}
	hence the update decrease these terms each by $y_{i_r}-y_{i_{r+1}}- 2$.
	Finally, for the terms where $j\in \{i_{r+1},\cdots,i_m\}, k\in \{i_1,\cdots,i_r\}$, it is easy to verify that their value stay at zero after the update.
	To sum up, such an update decreases $g(\by)$ by at least
	\begin{align*}
	\left(\sum_{\tau\in\cP^{-1}(\{i_1,\cdots,i_r\})\cap\cD^{-1}(\{i_{r+1},\cdots,i_m\})}\phi_{\tau}\right)
	\cdot
	\left[y_{i_r}-y_{i_{r+1}}- 2\right]^+\, .
	\end{align*}
	Note that the first term is lower bounded by $\alpha(\bphi)$ defined in \eqref{eq:strong-connectivity-jea}.
	As a result, after the finishing the procedure, $g(\by)$ decreased by at least:
	\begin{align*}
	\alpha(\bphi)\cdot \sum_{r=1}^{m-1}\left[y_{i_r}-y_{i_{r+1}}- 2 \right]^{+}
	\geq\
	\alpha(\bphi)\cdot \left[y_{i_1} - y_{i_m} - 2 m  \right]^+\, .
	\end{align*}
	By strong duality we have $\min_{\by}g(\by)=W^{\textup{SPP($\bphi$)}}$, hence
	\begin{align*}
	g(\by) -W^{\textup{SPP($\bphi$)}}
	\geq\
	\alpha(\bphi)\cdot \left[\max_{j\in V}y_j - \min_{k\in V}y_k - 2 m \right]^+\, .
	\end{align*}
	This concludes the proof.
\end{proof}

\subsection{Proof of Lemma \ref{lem:bound_rhs} and Its Generalization}\label{append:proof-bound-rhs}
In this section, we first prove a generalized version of Lemma \ref{lem:bound_rhs} (Lemma \ref{lem:rhs-bound-general}), then we show that Lemma \ref{lem:bound_rhs} follows from Lemma \ref{lem:rhs-bound-general} as a corollary.

\subsubsection{A General Result}\label{subsec:growth-condition-proof}
The lemma below generalizes Lemma \ref{lem:bound_rhs} in two aspects: 
(i) it holds for the more general setting (described in Section \ref{subsec:extend-barrier}) where a subset of the nodes can have finite capacity buffers;
(ii) it holds not only for the congestion function \eqref{eq:inv_sqrt_mirror_map}, but also for any congestion function that meet a certain growth condition defined below.
We then verify that the inverse square root congestion function \eqref{eq:inv_sqrt_mirror_map} indeed satisfy the growth condition.


Recall the terms in one-period Lyapunov analysis for the general setting in Section \ref{subsec:extend-barrier} (for the main setting in Section \ref{sec:model}, simply set $\bar{\mathbf{d}}=\mathbf{1}$ and replace the text in $\mathcal{V}_4$ by $q_j[t]=0$, $\exists j\in V$ below):
\begin{align}
	W^{\textup{SPP($\bphi^t$)}} -
	\mathbb{E}[v^{\textup{MBP}}[t]|\nq[t]] \nonumber
	\leq&\
	\underbrace{\tK\left(F (\nq[t]) - \mathbb{E}[F (\nq[t+1])|\nq[t]]\right)}_{\cV_1\textup{ change in potential}}
	+
	\underbrace{\frac{1}{2 \tK}
	\max_{j\in V} \max_{\bar{q}\in\left[\bar{q}_{j}[t]-\frac{1}{\tK},\bar{q}_{j}[t]+\frac{1}{\tK}\right]}f'_j(\bar{q})
	}_{\cV_2\textup{ loss due to stochasticity}}\\
	&\ +
	\underbrace{\left(W^{\textup{SPP($\bphi^t$)}}
		-g^t(\bof(\nq[t]))
		\right)}_{\cV_3\textup{ dual optimality gap}}
	+
	\underbrace{\mathds{1}\left\{q_{j}[t]=0\textup{ or }d_j,\exists j\in V\right\}}_{\cV_4\textup{ loss due to underflow}}\, , \label{eq:one-step-lyap-finite-buffer}
\end{align}
and the definition of $\bar{\bq}$ and $\tilde{K}$:
\begin{align*}
	\bar{q}_j \triangleq \frac{q_j + \bar{d}_j\delta_K}{\tK}
	\quad \textup{for}\quad
	\delta_K = \sqrt{K}
	\quad \textup{and}\quad
	\tK \triangleq K + \left(\sum_{j\in V}\bar{d}_j\right)\delta_K\, .
\end{align*}

\begin{lem}\label{lem:rhs-bound-general}
	In the general setting with finite buffers described in Section \ref{subsec:extend-barrier}, if the congestion functions $(f_j(\cdot))_{j\in V}$ satisfy the growth conditions defined below (Condition~\ref{cond:mirror_map}) with parameters $(\alpha,K_1,M_1,M_2)$, then for $K\geq K_1$,
	\begin{align}\label{eq:rhs-bound-general}
		\cV_2 + \cV_3 + \cV_4 \leq  M_2 \frac{1}{\tK}\, ,
	\end{align}
	where $\cV_2,\cV_3,\cV_4$ are defined in eq. \eqref{eq:one-step-lyap-finite-buffer} and $\tK$ is defined above.
\end{lem}

\begin{cond}[Growth condition for congestion functions]\label{cond:mirror_map}
	We say the congestion functions $(f_j(\cdot))_{j\in V}$ satisfy the growth condition with parameters $(\alpha,K_1, M_1, M_2) \in \mathbb{R}_{++}^4$ if the following holds:
	\begin{enumerate}
		\item For each $j\in V$, $f_j(\cdot)$ is strictly increasing and continuously differentiable. Moreover,
		\begin{enumerate}
			\item For any $K>K_1$, $f_j(\bar{q}_{j}) \leq f_k(\bar{q}_{k})$ (i)~for any $k\in V$ if $q_j=0$, and (ii)~for any $j\in V$ if $q_k=d_k$, $k\in\Vb$.
			\item For any $j,k\in V$, we have $f_j\left(\frac{\bar{d}_j}{\sum_{\ell\in V}\bar{d}_{\ell}}\right) = f_k\left(\frac{\bar{d}_k}{\sum_{\ell\in V}\bar{d}_{\ell}}\right)$.
		\end{enumerate}
		\item Define
		\begin{align*}
			\mathcal{B}(\bof)\triangleq \left\{
			\nq \in \Omega:\;
			\max_{j\in V}
			\left|
			f_j\left(\frac{\bar{d}_j}{\sum_{\ell\in V}\bar d_{\ell}}\right)
			-
			f_j(\bar{q}_j)
			\right|
			\leq
			4m
			\right\}\, .
		\end{align*}
		Denote $\bar{\mathcal{B}}(\bof)\triangleq \Omega\backslash \mathcal{B}(\bof)$.
		\begin{enumerate}
			\item  For any $K>K_1$, $\forall \nq \in \bar{\mathcal{B}}(\bof)$,
			\begin{align}
				&\alpha
				\left(
				\max_{j\in V}
				\left|
				f_j\left(\frac{\bar{d}_j}{\sum_{\ell\in V}\bar d_{\ell}}\right)
				-
				f_j(\bar{q}_j)
				\right|
				-
				2 m
				\right)^+ \nonumber\\
				\geq &
				\frac{1}{2 \tK}
				\max_{j\in V} \max_{\bar{q}\in\left[\bar{q}_{j}[t]-\frac{1}{\tK},\bar{q}_{j}[t]+\frac{1}{\tK}\right]}f'_j(\bar{q})
				+
				\mathds{1}\{q_j=0\textup{ or }d_j,\exists j\}
				\, .\label{eq:cond-eq-1}
			\end{align}
			\item Let $F(\nq)$ be the antiderivative of $\bof(\nq)\triangleq (f_j(\bar{q}_j))_{j\in V}$, we have $\sup_{\bq,\bq'\in\Omega}(F(\nq) - F(\nq')) \leq M_1$.
			\item
			We have
			$
			\sup_{\nq \in\mathcal{B}(\bof)} \max_{j\in V} \max_{\bar{q}\in\left[\bar{q}_{j}[t]-\frac{1}{\tK},\bar{q}_{j}[t]+\frac{1}{\tK}\right]}f'_j(\bar{q})
			\leq
			M_2.
			$
			\item If $\exists j\in V$ such that $q_j=0$ or $q_j=d_j$, then $\nq \in  \bar{\mathcal{B}}(\bof)$.
		\end{enumerate}
	\end{enumerate}
\end{cond}

\begin{proof}[Proof of Lemma \ref{lem:rhs-bound-general}]
	For $\nq \in \mathcal{B}(\bof)$, using Condition \ref{cond:mirror_map} point 2(d), we have $\cV_4=0$. By weak duality of SPP($\bphi^t$) we have $\cV_3\leq 0$. As a result, it follows from Condition \ref{cond:mirror_map} point 2(c), that
	\begin{align*}
		\cV_2 + \cV_3 + \cV_4\leq \
		&\frac{1}{2 \tK}\cdot \sup_{\nq \in \mathcal{B}(\bof)}
		\max_{j\in V} \max_{\bar{q}\in\left[\bar{q}_{j}[t]-\frac{1}{\tK},\bar{q}_{j}[t]+\frac{1}{\tK}\right]}f'_j(\bar{q})
		= M_2\cdot\frac{1}{\tK}\, .
	\end{align*}
	
	For $\nq \in \bar{\mathcal{B}}(\bof)$.
	We aim to use the inequality \eqref{eq:cond-eq-1} to prove the desired result. To do so, we first show that
	\begin{align*}
		\cV_3\leq\ -\alpha(\bphi^t)\cdot \left(\max_{j\in V}\left|f_j\left(\frac{\bar{d}_j}{\sum_{\ell\in V}\bar d_{\ell}}\right)
		-
		f_j(\bar{q}_j)\right| - 2m \right)^+\, .
	\end{align*}
	It was proved in Lemma \ref{lem:dual_subopt} that
	\begin{align*}
		\cV_3 \leq\
		-\alpha(\bphi^t)\cdot \left[\max_{j\in V}f_j (\bar{q}_j) - \min_{j\in V}f_j(\bar{q}_j) - 2 m\right]^+\, .
	\end{align*}
	Note that
	\begin{align*}
		&\max_{j\in V}\left|f_j\left(\frac{\bar{d}_j}{\sum_{\ell\in V}\bar d_{\ell}}\right)
		-
		f_j(\bar{q}_j)\right| \\
		\leq\ & \max\left\{\max_{j\in V}f_j(\bar{q}_j)
		-
		\min_{j\in V}f_j\left(\frac{\bar{d}_j}{\sum_{\ell\in V}\bar d_{\ell}}\right),
		\max_{j\in V}f_j\left(\frac{\bar{d}_j}{\sum_{\ell\in V}\bar d_{\ell}}\right)
		-
		\min_{j\in V} f_j(\bar{q}_j)
		\right\}\, .
	\end{align*}
	Note that there must exist $j^*\in V$ such that $\bar{q}_{j^*}\leq \frac{\bar{d}_{j^*}}{\sum_{\ell\in V}\bar{d}_{\ell}}$, hence $f_{j^*}(\bar{q}_{j^*}) \leq f_j\left(\frac{\bar{d}_{j^*}}{\sum_{\ell\in V}\bar{d}_{\ell}}\right)$.
	Because the congestion functions satisfy Condition \ref{cond:mirror_map} point 1(b), we have $f_j\left(\frac{\bar{d}_j}{\sum_{\ell\in V}\bar d_{\ell}}\right)$ has the same value for all $j\in V$, therefore
	\begin{align*}
		\max_{j\in V}f_j(\bar{q}_j)
		-
		\min_{j\in V}f_j\left(\frac{\bar{d}_j}{\sum_{\ell\in V}\bar d_{\ell}}\right)
		=\ &
		\max_{j\in V}f_j(\bar{q}_j) - f_{j^*}\left(\frac{\bar{d}_{j^*}}{\sum_{\ell\in V}\bar d_{\ell}}\right)\\
		\leq\ &\max_{j\in V}f_j(\bar{q}_j) - f_{j^*}(\bar{q}_{j^*})\\
		\leq\ &\max_{j\in V}f_j(\bar{q}_j) - \min_{j\in V}f_j(\bar{q}_j)\, .
	\end{align*}
	Similarly, we can show that
	\begin{align*}
		\max_{j\in V}f_j\left(\frac{\bar{d}_j}{\sum_{\ell\in V}\bar d_{\ell}}\right)
		-
		\min_{j\in V} f_j(\bar{q}_j)\leq\
		\max_{j\in V}f_j(\bar{q}_j) - \min_{j\in V}f_j(\bar{q}_j)\, .
	\end{align*}
	Combined, we have
	\begin{align*}
		\cV_3\leq\ -\alpha(\bphi^t)\cdot \left(\max_{j\in V}\left|f_j\left(\frac{\bar{d}_j}{\sum_{\ell\in V}\bar d_{\ell}}\right)
		-
		f_j(\bar{q}_j)\right| - 2m \right)^+\, .
	\end{align*}
	Plugging in Condition \ref{cond:mirror_map} point 2(a), we have for $\nq \in \bar{\mathcal{B}}(\bof)$,
	\begin{align*}
		\cV_2 + \cV_3 + \cV_4 \leq\ 0\, .
	\end{align*}
	
	Combine the above two cases, we conclude the proof.
\end{proof}

\subsubsection{Proof of Lemma \ref{lem:bound_rhs}}
We prove the result below that generalizes Lemma \ref{lem:bound_rhs} to the setting of finite buffer queues (Section \ref{subsec:extend-barrier}). By setting $V_b=\emptyset$ in Lemma \ref{lem:bound_rhs_jea}, we recover Lemma \ref{lem:bound_rhs}.

\begin{lem}\label{lem:bound_rhs_jea}
	Consider the congestion function \eqref{eq:inv_sqrt_mirror_map_gen}.
	Consider a set  $V$ of $m=|V|>1$ nodes, a subset $\Vb\subset V$ of buffer-constrained nodes with scaled buffer sizes $\bar{d}_j\in(0,1)$ (recall that we define $\bar{d}_j\triangleq 1$ for all $j\in V\backslash \Vb$) satisfying $\sum_{j\in V}\bar{d}_j>1$, and any $\{(\bphi^t,\cP,\cD)\}$ that satisfies Condition  \ref{cond:str_connect_jea}.
	Recall that $\alpha_{\min}=\min_{0\leq t\leq T}\alpha(\bphi^t)>0$.
	Then there exists $K_1 = \textup{poly}\left(m,\bar{\bd},\frac{1}{\alpha_{\min}}\right )$ such that for $K\geq K_1$,
	\begin{align*}
		\cV_2 + \cV_3 + \cV_4 \leq  M_2 \frac{1}{\tK}\, , \qquad \textup{for}\ M_2 = C\frac{\sqrt{m}}{\min_{j\in V}\bar{d}_j}\left(\frac{\sum_{j\in V}\bar{d}_j}{\min\{\sum_{j\in V}\bar{d}_j-1,1\}}\right)^{3/2} \, ,
	\end{align*}
	where $\cV_2,\cV_3,\cV_4$ were defined in \eqref{eq:one-step-lyap-finite-buffer}, $\tK$ was defined in \eqref{eq:normalized-queue-general}, and $C>0$ is a universal constant that is independent of $m$, $\bar{\bd}$, $K$, or $\alpha(\bphi^t)$.
\end{lem}
%
To prove Lemma \ref{lem:bound_rhs_jea}, it remains to be shown that the congestion function \eqref{eq:inv_sqrt_mirror_map_gen} satisfies Condition \ref{cond:mirror_map}.
\begin{lem}\label{lem:inv-sqrt-satisfy-cond}
	The congestion function \eqref{eq:inv_sqrt_mirror_map_gen} satisfies the growth conditions (Condition~\ref{cond:mirror_map}) with parameters $(\alpha_{\min},K_1,M_1,M_2)$ where
	$$
	K_1 = \textup{poly}\left(m,\bar{\bd},\frac{1}{\alpha_{\min}}\right)\, ,\quad  M_1=C m\, ,\quad M_2=C \frac{1}{\min_{j\in V}\bar{d}_j}\left(\frac{\sum_{j\in V}\bar{d}_j}{\min\{\sum_{j\in V}\bar{d}_j-1,1\}}\right)^{3/2}
	\sqrt{m}\, .
	$$
	Here $C$ is a universal constant that is independent of $m$, $\bar{\bd}$, $K$ and $\alpha_{\min}$.
\end{lem}
We delay the proof of Lemma \ref{lem:inv-sqrt-satisfy-cond} to Appendix \ref{append:validate-growth-cond}.
We are now ready to prove Lemma \ref{lem:bound_rhs_jea}.

\begin{proof}[Proof of Lemma \ref{lem:bound_rhs_jea}]
	Lemma \ref{lem:bound_rhs_jea} immediately follows from Lemma \ref{lem:rhs-bound-general} and Lemma \ref{lem:inv-sqrt-satisfy-cond}.
\end{proof}

\subsubsection{Validating Condition \ref{cond:mirror_map} for Congestion Function \eqref{eq:inv_sqrt_mirror_map_gen}}\label{append:validate-growth-cond}
In this section, we prove Lemma \ref{lem:inv-sqrt-satisfy-cond}. 
%
Recall the congestion function defined in \eqref{eq:inv_sqrt_mirror_map_gen}: let $\Vb\subset V$ be the subset of buffer-constrained nodes with scaled buffer sizes $\bar{d}_j\in(0,1)$, and
\begin{align*}
	f_j(\bar{q}_j) &= \sqrt{m} \cdot
	\Cb\cdot\left(
	\left(1 - \frac{\bar{q}_j}{\bar{d}_j}\right)^{-\frac{1}{2}}
	-
	\left(\frac{\bar{q}_j}{\bar{d}_j}\right)^{-\frac{1}{2}}
	-
	\Db
	\right)\, , &\forall j\in \Vb\,  \\
	f_j(\bar{q}_j) &= -\sqrt{m} \cdot
	\bar{q}_j^{-\frac{1}{2}}\, , &\forall j\in V\backslash \Vb\,
\end{align*}
Here $\Cb$ and $\Db$ are normalizing constants chosen as follows.
Define $\epsilon\triangleq \frac{\delta_K}{\tK}$ (recall that $\delta_K\triangleq \sqrt{K}$ and $\tK \triangleq K + \left(\sum_{j\in V}\bar{d}_j\right)\delta_K$).
Let $\hb(\bar{q})\triangleq (1-\bar{q})^{-\frac{1}{2}} - {\bar{q}}^{-\frac{1}{2}}$ and
$h(\bar{q})\triangleq - {\bar{q}}^{-\frac{1}{2}}$.
Define $\Cb\triangleq \frac{h(\epsilon)-h(1/\sum_{j\in V}\bar{d}_j)}{\hb(\epsilon)-\hb(1/\sum_{j\in V}\bar{d}_j)}$
and $\Db \triangleq \hb(1/\sum_{j\in V}\bar{d}_j) - \Cb^{-1}h(1/\sum_{j\in V}\bar{d}_j)$.
These definitions ensure that Condition \ref{cond:mirror_map} point 1(b) holds, and are useful in establishing Condition \ref{cond:mirror_map} point 1(a).

\begin{proof}[Proof of Lemma \ref{lem:inv-sqrt-satisfy-cond}]
	The proof of this lemma involves a lot of notations and computation.
	For readability, we use the following simplifying notation (with a slight abuse of notation): for $x_a,y_a\in\mathbb{R}_+$ where $a\in\mathcal{A}\subset\mathbb{Z}_+$, $\{x_a\}=O(\{y_a\})$ ($\{x_a\}=\Omega(\{y_a\})$, resp.) means that there exists a universal constant $C>0$ that does not depend on $m,K,\bar{\bd}$, or $\alpha_{\min}$ such that $x_a \leq C y_a$ ($x \geq C y_a$, resp.) for each $a\in\mathcal{A}$. We say $\{x_a\}=\Theta(\{y_a\})$ if $\{x_a\}=O(\{y_a\})$ and $\{x_a\}=\Omega(\{y_a\})$.
	Denote $\dsigma\triangleq \sum_{j\in V}\bar{d}_j$, $\dgap\triangleq \min\{1,\sum_{j\in V}\bar{d}_j - 1\}$, $\dmin\triangleq \min_{j\in V}\bar{d}_j$.
	Recall that $\bar{d}_j\in(0,1)$ for any $j\in \Vb$, and that $\dsigma>1$.
	
	\begin{itemize}
		\item  Point 1. It is not hard to see that the congestion functions $(f_j(\bar{q}_j))_{j\in V}$ are strictly increasing and continuously differentiable.
		For any $K>0$, we have $f_j(\bar{q}_j)=f_k(\bar{q}_k)$ for any $j,k\in V$ if $q_j=q_k=0$. As a result, if $q_j=0$, we have $f_j(\bar{q}_j)\leq f_k(\bar{q}_k)$ for any $k\in V$.
		It can be easily verified that Point 1(b) is also satisfied by any $K>0$.
		
		It remains to be shown that there exists $K_1<\infty$ such that for $K\geq K_1$, we have $f_j(\bar{q}_j)\leq f_k(\bar{q}_k)$ for any $j\in V$ if $q_k=d_k$ and $k\in \Vb$.
		To this end, if suffices to check the inequality $f_j(\bar{q}_j)\leq f_k(\bar{q}_k)$ for $q_j=d_j$, $q_k=d_k$ where $j\in V\backslash\Vb$ and $k\in \Vb$: In this case, we have $f_j(\bar{q}_j)\leq 0$; for $K= \Omega(\max\{\dsigma^2,\frac{\dsigma^2}{\dgap^2}\})$, we have $\Cb=\Theta(1)$, $\Db=O(\sqrt{\frac{\dsigma}{\dgap}})$ hence $f_k(\bar{q}_k)= \Omega(\sqrt{m}\frac{K^{1/4}}{\dgap^{1/2}})\geq 0$. Therefore point 1 is satisfied for  $K_1 = \Omega(\max\{\dsigma^2,\frac{\dsigma^2}{\dgap^2}\})=\Omega(\frac{\dsigma^2}{\dgap^2})$.
		\item Point 2(a). For $\bq$ such that $\nq\in\bar{\mathcal{B}}(\bof)$ and $0<q_j<d_j$ for any $j\in V$, we have, by definition of $\nq\in\bar{\mathcal{B}}(\bof)$,
		\begin{align*}
			\textup{LHS of \eqref{eq:cond-eq-1}}\geq 2m\alpha\, .
		\end{align*}
		On the other hand, we have for $K = \Omega(\frac{\dsigma^2}{\dgap^2})$, we have $\Cb=\Theta(1)$ hence
		\begin{align*}
			\textup{RHS of \eqref{eq:cond-eq-1}}= O\left(\frac{1}{K}\cdot \sqrt{m}\cdot  K^{3/4}\cdot \dmin^{-1}\dgap^{-3/2}\right)
		\end{align*}
		Here the RHS of \eqref{eq:cond-eq-1} is maximized when $q_j=0$ or $q_j=d_j$.
		Therefore \eqref{eq:cond-eq-1} holds for $K\geq K_1 = \Omega\left(\max\left\{\frac{\dsigma^2}{\dgap^2},\frac{1}{m^2 \alpha^4 \dmin^4\dgap^6}\right\}\right)$.
		For $\bq$ such that $\nq\in\bar{\mathcal{B}}(\bof)$ and $q_j=0$ or $d_j$ for some $j'\in V$, we have
		\begin{align*}
			\textup{LHS of \eqref{eq:cond-eq-1}}\geq \alpha \sqrt{m}\cdot \Omega\left(K^{1/4} - \sqrt{\dsigma}\right)\, ,
		\end{align*}
		which is obtained by plugging in $q_{j'}$.
		For $K=\Omega(\frac{\dsigma^2}{\dgap^2})$, we also have
		\begin{align*}
			\textup{RHS of \eqref{eq:cond-eq-1}}= O\left(\frac{1}{K}\cdot \sqrt{m}\cdot K^{3/4}\cdot \dmin^{-1} \dgap^{-3/2} + 1\right)\, ,
		\end{align*}
		Using the analysis above, for $K\geq \Omega\left(\frac{m^2}{\dmin^4\dgap^6}\right)$, the first term in the parentheses is $O(1)$.
		In this case we have $\textup{RHS of \eqref{eq:cond-eq-1}}= O(1)$.
		Therefore \eqref{eq:cond-eq-1} holds for $$K=\Omega\left(\max\left\{\frac{1}{m^2 \alpha^4 \dmin^4\dgap^6},\frac{m^2}{\dmin^4\dgap^6},\frac{1}{\alpha^2 \dgap^3},\frac{\dsigma^2}{\dgap^2}\right\}\right)\, .$$
		
		Combined, \eqref{eq:cond-eq-1} holds for $K_1 =\Omega\left(\max\left\{\frac{1}{m^2 \alpha^4 \dmin^4\dgap^6},\frac{m^2}{\dmin^4\dgap^6},\frac{1}{\alpha^2 \dgap^3},\frac{\dsigma^2}{\dgap^2}\right\}\right)$.
		\item Point 2(b). Note that for $K=\Omega(\frac{\dsigma^2}{\dgap^2})$,
		\begin{align*}
			&\sup_{\bq,\bq'\in \Omega^{K}}\left(
			F(\nq) - F(\nq')
			\right) \\
			\leq \ & \max\{\Cb,1\}\cdot O\left(\sqrt{m}\sup_{\bq,\bq'\in \Omega^{K}}\left(
			\sum_{j\in V}\sqrt{\bar{d}_j}\left(-\sqrt{\bar{q}_{j}}
			-
			\sqrt{\bar{d}_j-\bar{q}_{j}}
			+
			\sqrt{\bar{q}'_{j}}
			+
			\sqrt{\bar{d}_j-\bar{q}'_{j}}
			\right)
			\right)
			\right) \\
			\leq\ & O\left(\sqrt{m} \max_{\bq'\in \Omega^K}\sum_{j\in V}\sqrt{\bar{d}_j}\left(
			\sqrt{\bar{q}'_{j}}
			+
			\sqrt{\bar{d}_j-\bar{q}'_{j}}
			\right)
			\right)\\
			=\ & O(m)\, .
		\end{align*}
		Hence
		$
			M_1 = O(m)\, .
		$
		\item Point 2(c). Note that for $\nq\in\mathcal{B}_{\bof}$, we have $\bar{q}_j =\Theta\left(\frac{\bar{d}_j}{\dsigma}\right)$, hence
		\begin{align*}
			M_2 = \max_{\nq\in\mathcal{B}_{\bof}}\max_{j\in V} \max_{\bar{q}\in\left[\bar{q}_{j}[t]-\frac{1}{\tK},\bar{q}_{j}[t]+\frac{1}{\tK}\right]}f'_j(\bar{q})
			\leq
			\frac{1}{\dmin}\left(\frac{\dsigma}{\dgap}\right)^{3/2}
			O(\sqrt{m})\, .
		\end{align*}
		For the special case where $\Vb=\emptyset$ hence $\bar{d}_j=1$ for all $j\in V$, we have $\dsigma=m$, $\dmin=1$, $\dgap=1$ and $M_2 = O(m^2)$.
		\item Point 2(d). Note that for $\nq\in\mathcal{B}_{\bof}$, we have $\bar{q}_j =\Theta\left(\frac{\bar{d}_j}{\dsigma}\right)$, hence point 2(d) holds.
	\end{itemize}
\end{proof}

\section{Bounding the Variation Gap: Proofs in Section \ref{sec:bound-variation-gap}} \label{append:proof-main-theorems}
\yka{I deleted the assertions that Theorem 1 is proved here. That theorem was already proved in the main paper, and our job was only to prove the three lemmas.}

\begin{proof}[Proof of Proposition \ref{lem:SPPavg-vs-avgofSPPt}]
	\emph{Step 1 (Flow decomposition of $\textup{SPP}(\bar{\bphi)}$).}
	Let $\widebar{\bz}$ be an optimal solution to $\textup{SPP}(\bar{\bphi)}$. 
	For the $t$'s such that $\widebar{\bz}$ is feasible for $\textup{SPP}(\bphi^t)$ we are done. 
	Now consider any $t$ such that $\widebar{\bz}$ is \emph{not} feasible for $\textup{SPP}(\bphi^t)$. 
	Using the standard flow decomposition approach \citep[see, e.g.,][the interested reader can also find the flow decomposition argument in the proof of Lemma~\ref{lem:flow-decomp-jea-z} above]{williamson2019network}, the flow $\widebar{\bz}$ can be decomposed into flows along directed cycles, since it satisfies the flow balance constraints \eqref{eq:fluid_linear_flow_bal-z}:
	directed cycles $\cC$  carrying flow $f_{\cC} > 0$ in the decomposition take the form $\cC= \big((j_1,\tau_1,j_2),(j_2,\tau_2, j_3),\cdots,(j_{s},\tau_s, j_{s+1}=j_1)\big)$ where the nodes $j_1, j_2, \dots, j_s$ are distinct from each other, and for each $r=1, 2, \dots, s$, there is a flow from $j_r$ to $j_{r+1}$ due to demand type $\tau_r$. We have
	\begin{align}
		\widebar{z}_{j\tau k} = \sum_{\cC \ni (j, \tau, k)} f_{\cC} \qquad \textup{for all } \tau \in \cT, j \in \cP(\tau), k \in \cD(\tau) \, .
		\label{eq:cycle-decomp}
	\end{align}
	(Note that the number of cycles in the decomposition is bounded above by $\sum_{\tau \in \cT} |\cP(\tau)| |\cD(\tau)|$, but our argument will not be affected by the number of cycles. In fact our argument can handle an infinity of demand types by replacing sums with integrals.)
	
	\medskip
	\emph{Step 2 (Obtain a feasible flow of $\textup{SPP}(\bphi^t)$).}
	Starting from the flow $\widebar{\bz}$ and the associated cycle decomposition \eqref{eq:cycle-decomp}, we reduce the flows $(f_{\cC})$ along the cycles via the following iterative process, in order to obtain $\bz^t$ which is feasible for the problem $\textup{SPP}(\bphi^t)$:
	
	Consider each demand type $\tau \in \cT$ in turn and do the following. Define the (current) arrival rate violation as
	$$\delta_\tau \triangleq \left ( \sum_{\cC} f_{\cC} \cdot \textup{count}(\cC, \tau) \ - \  \phi^{t}_{\tau} \right )^+ \, $$
	where $\textup{count}(\cC, \tau)$ is the number of times demand type $\tau$ appears in cycle $\cC$.  If $\delta_{\tau} = 0$ do nothing. If $ \delta_{\tau} > 0$, reduce the flows in cycles containing $\tau$ sufficiently that after the reduction $\sum_{\cC} f_{\cC} \cdot \textup{count}(\cC, \tau) = \phi^{t}_{\tau}$ holds  (the reduction can be divided arbitrarily between the different cycles containing $\tau$; subject to the constraints that no cycle-flow should increase and no cycle-flow should go below zero). Note that the payoff loss resulting from this reduction is bounded above by $\delta_{\tau} m$ since each cycle length is at most $m$ (since no node is repeated in a cycle), the $w$s are assumed to be bounded by 1, and the total reduction in cycle flows is at most $\delta_{\tau}$.
	This simple process maintains the following properties:
	\begin{compactitem}[leftmargin=*]
		\item The flow balance constraint \eqref{eq:fluid_linear_flow_bal-z} is satisfied throughout.
		\item Cycle-flows are non-increasing during the process. Cycle-flows never drop below zero.
		\item For all demand types which have already been processed so far, the arrival rate constraint is satisfied. Formally: During the process, denote the current value of the right-hand side of \eqref{eq:cycle-decomp} by $z_{j \tau k}$. Then $\sum_{j \in \cP(\tau), k \in \cD(\tau)}z_{j \tau k} =  \sum_{\cC} f_{\cC} \cdot \textup{count}(\cC, \tau) \leq \phi^{t}_{\tau}$ for all demand types $\tau $ which have already been processed.
	\end{compactitem}
	
	In particular, at the end of the process, we arrive at flows $\bz^t$ which are feasible for $\textup{SPP}(\bphi^t)$. 
	
	\medskip
	\emph{Step 3 (Bound the payoff loss).}
	It remains to show that the payoff lost due to the reduction in flows is bounded by $m\eta_T$.
	
	Since flows are non-increasing and the initial flows are feasible for $\textup{SPP}(\bar{\bphi})$, we have that $\delta_\tau  \leq \left ( \widebar{\phi}_{ \tau} - \phi^{t}_{\tau} \right )_+$ for all $\tau \in \cT$. Since the payoff lost while processing demand type $\tau$ is bounded above by $\delta_\tau m$ (as we argued above), the total loss in payoff lost is then bounded above by
	\begin{align*}
		W^{\textup{SPP}(\bar{\bphi)}}-
		W^{\textup{SPP}(\bphi^t)}
		\leq
		m \sum_{\tau \in \cT} \delta_\tau 
		\leq m \sum_{\tau \in \cT} \left ( \widebar{\phi}_{ \tau} - \phi^{t}_{\tau} \right )^+ 
		\leq m \lVert \bphi^{t} - \widebar{\bphi}\rVert_1 \, .
	\end{align*}
	We further bound the RHS from above by:
	\begin{align*}
		m\lVert \bphi^{t} - \widebar{\bphi}\rVert_1
		= 
		m\left\lVert \frac{1}{T}\sum_{s=1}^{T} (\bphi^{t} - \bphi^{s})\right\rVert_1
		\leq
		m\frac{1}{T}\sum_{s=1}^{T}\lVert  \bphi^{t} - \bphi^{s}\rVert_1
		\leq
		m\frac{1}{T}\cdot T \cdot (\eta T)
		=
		m\eta T\, ,
	\end{align*}
	where the last inequality follows from the definition of \emph{$\eta$-slowly varying demand} (Definition \ref{defn:variation-budget}).
	Finally, we sum over $t$ and obtain that
	\begin{align}
		W^{\textup{SPP}(\bar{\bphi)}}
		-
		\frac{1}{T}\sum_{t=1}^{T} W^{\textup{SPP}(\bphi^t)} 
		\leq   m \eta T \, .
	\end{align}
	The proposition follows.
\end{proof}


\section{Proof of Theorem \ref{thm:main_regret}}\label{append:proof-main-regret}

In this section, we state and prove the following Theorem.
\begin{thm}[General result for finite-buffer setting]\label{thm:general_mirror_map}
	Consider a set $V$ of $m \! \triangleq \! |V|>1$ nodes, a subset $\Vb \subseteq V$ of buffer-constrained nodes with scaled buffer sizes $\bar{d}_j \in (0,1) \ \forall j \in \Vb$ satisfying $\sum_{j\in V}\bar{d}_j>1$.
	Recall that $\alpha_{\min}=\min_{0\leq t\leq T}\alpha(\bphi^t)>0$. Consider any congestion functions $(f_j(\cdot))_{j \in V}$ that satisfy Condition \ref{cond:mirror_map} with parameters $(\alpha = \alpha_{\min},K_1, M_1, M_2) \in \mathbb{R}_{++}^4$ for any $0\leq t\leq T$. 
	Then for any horizon $T$, any $K\geq K_1$, and any sequence of demand arrival rates $(\bphi^t)_{t=0}^{T-1}$ which varies $\eta$-slowly (for some $\eta \in [0,2]$) and pickup and dropoff neighborhoods $\cP$ and $\cD$ such that $\{(\bphi^t, \cP, \cD)\}_{t\leq T}$ satisfy Condition~\ref{cond:str_connect_jea}, we have
	\begin{align*}
		L^{\textup{MBP}}_T\leq
		(4\sqrt{M_1 m} + 2M_1)\cdot \left( \frac{K}{T} + \sqrt{\eta K} \right )
		+
		M_2
		\frac{1}{K}\, .
	\end{align*}
\end{thm}

Note that Theorem \ref{thm:general_mirror_map} generalizes both Theorem \ref{thm:main_regret} and Theorem \ref{thm:main-result-finite-buffer}. 
We will only prove Theorem \ref{thm:general_mirror_map}, and Theorems \ref{thm:main_regret} and \ref{thm:main-result-finite-buffer} will follow as corollaries since the congestion function \eqref{eq:inv_sqrt_mirror_map_gen} satisfies Condition \ref{cond:mirror_map} (Lemma \ref{lem:inv-sqrt-satisfy-cond}).

\begin{proof}[Proof of Theorem~\ref{thm:general_mirror_map}]
	
	
	First note that the following claim analogous to Proposition \ref{prop:main_regret_DeltaT} holds for the setting stated in this theorem. The claim can be proved by simply repeating the proof of Proposition \ref{prop:main_regret_DeltaT} and replacing Lemma \ref{lem:bound_rhs} by Lemma \ref{lem:rhs-bound-general}.
	
	\medskip
	\emph{Claim.} Consider the setting in Theorem \ref{thm:general_mirror_map}. Then for any congestion function $\bof(\cdot)$ that satisfy Condition \ref{cond:mirror_map} with parameters $(\alpha,K_1,M_1,M_2)$, the following result holds.
	For any  $K\geq K_1$, and any $0<\Delta_T<T$ the following guarantees hold for Algorithm~\ref{alg:mirror_bp}
	\begin{align}\label{eq:LT-bound-general}
		L^{\textup{MBP}}_T\leq
		M_1 \frac{K}{\Delta_T} + M_2\frac{1}{K}
		+
		\Delta_T m\eta\, .
	\end{align}
	
	\medskip
	It remains to choose $\Delta_T$ appropriately, i.e., to divide the horizon $T$ into intervals of appropriate length. Note that the bound on per period loss \eqref{eq:LT-bound-general} is minimized for $\Delta_T = T_*  = \sqrt{(M_1/m) (K/\eta)}$, which makes the first and third terms equal. This observation will guide our choice of $\Delta_T$.
	
	If $T \leq T_*$, we set $\Delta_T = T$ and we immediately have
	\begin{align}
		L^{\textup{MBP}}_{T} \leq  \frac{K}{T}  2 M_1 + M_2 \frac{1}{K} \qquad \forall T< T_* \, ,
		\label{eq:LT-bound-small-T}
	\end{align}
	since the first term is larger than the third term in \eqref{eq:LT-bound-general}. If $T > T_*$ then we divide $T$ into $\lceil T/T_* \rceil$ intervals of equal length (up to rounding error). In particular, each interval has length $\Delta_T \in [T_*/2,T_*]$, the first term is again larger than the third term in \eqref{eq:LT-bound} and so the per period loss in each interval  is bounded above by
	$$\frac{K}{\Delta_T}  2M_1 + M_2 \frac{1}{K} \leq \frac{K}{T_*/2}  2M_1 + M_2 \frac{1}{K} = \sqrt{\eta K}  4 \sqrt{M_1 m} + M_2 \frac{1}{K}\, .$$
	Since this bound holds for each interval, it holds for the full horizon of length $T$, i.e.,
	\begin{align}
		L^{\textup{MBP}}_{T} \leq  \sqrt{\eta K}  4 \sqrt{M_1 m} + M_2 \frac{1}{K} \qquad \forall T\geq T_* \, .
		\label{eq:LT-bound-large-T}
	\end{align}
	Combining \eqref{eq:LT-bound-small-T} and \eqref{eq:LT-bound-large-T}, we obtain that for any $K \geq K_1 \triangleq  \max(K_2, K_3)$ and any horizon $T$, we have
	$$L^{\textup{MBP}}_{T} 
	\leq  4\sqrt{\eta K} \sqrt{M_1 m} + M_2 \frac{1}{K } + \frac{K}{T}  2M_1 
	\leq (4\sqrt{M_1 m} + 2M_1) \left(\frac{K}{T} + \sqrt{\eta K}\right ) +  M_2 \frac{1}{K }\, .$$ 
	We obtain the bound claimed in the theorem.
\end{proof}

\begin{proof}[Proof of Theorem \ref{thm:main-result-finite-buffer}]
	In Lemma \ref{lem:inv-sqrt-satisfy-cond} we show that the congestion function \eqref{eq:inv_sqrt_mirror_map_gen} satisfies the growth conditions (Condition~\ref{cond:mirror_map}) with parameters $(\alpha_{\min},K_1,M_1,M_2)$ where
	$$
	K_1 = \textup{poly}\left(m,\bar{\bd},\frac{1}{\alpha_{\min}}\right)\, ,\quad  M_1=C m\, ,\quad M_2=C \frac{1}{\min_{j\in V}\bar{d}_j}\left(\frac{\sum_{j\in V}\bar{d}_j}{\min\{\sum_{j\in V}\bar{d}_j-1,1\}}\right)^{3/2}
	\sqrt{m}\, .
	$$
	Here $C$ is a universal constant that is independent of $m$, $\bar{\bd}$, $K$ and $\alpha_{\min}$.
	Plugging in the above constants to Theorem \ref{thm:general_mirror_map}, we obtain the desired result.
\end{proof}

\begin{proof}[Proof of Theorem \ref{thm:main_regret}]
	Setting $V_b=\emptyset$ in Theorem \ref{thm:main-result-finite-buffer}, and recall that $\bar{d}_j=1$ for $j\in V\backslash V_b$, we obtain the desired result.
\end{proof}



\pqa{End here. And one para on the special case.}

\section{Results for the Joint Pricing-Assignment Setting}\label{appen:jpa}
In this section, we outline some of the key lemmas and propositions that lead to a proof of Theorem \ref{thm:main_regret_jpa}.
Since most of the proofs are almost identical to their JEA counterpart, we omit most of the proofs or only provide an outline if necessary to avoid repetition.

\subsection{The Static Planning Problem for JPA} 
The static planning problem (SPP) in the JPA setting is
\begin{align}
	\textup{maximize}_{\bx}\
	&\sum_{\tau \in \cT}\
	\phi_{\tau}
	\left(
	r_{\tau}\left(\sum_{j\in \cP(\tau),k\in\cD(\tau)}x_{j\tau k}\right)
	-
	\sum_{j\in \cP(\tau),k\in\cD(\tau)}c_{j\tau k}\cdot x_{j\tau k}
	\right) \label{eq:fluid-jpa-obj}\\
	\textup{s.t.}\ &
	\sum_{\tau \in \cT}\phi_{\tau}\sum_{j\in \cP(\tau),k\in\cD(\tau)}x_{j\tau k}(\boe_j - \boe_k) = \bzero
	\hspace{4.5cm}
	\textup{(flow balance)} \label{eq:fluid-jpa-flow-bal}\\
	&\sum_{j\in \cP(\tau),k\in\cD(\tau)}x_{j\tau k}\leq 1\, ,\ x_{j\tau k}\geq 0\quad \forall j,k\in V\, ,\ \tau \in \cT
	\hspace{1.2cm}
	\textup{(demand constraint)}\, .\label{eq:fluid-jpa-dmd-constr}
\end{align}

We have the following result that is analogous to Proposition \ref{prop:fluid_ub}.

\begin{prop}\label{prop:fluid_upper_bound_jpa}
	For any horizon $T< \infty$, any $K$ and any starting state $\bq[0]$, the finite and infinite horizon average payoff $W^*_T$ and $W^*$ in the JPA setting are upper bounded as
	\begin{align*}
		W^*_T \leq \Wspp + m \frac{K}{T} \, ,
		\qquad
		W^* \leq \Wspp \, .
	\end{align*}
	Here $\Wspp$ is the optimal value of SPP \eqref{eq:fluid-jpa-obj}-\eqref{eq:fluid-jpa-dmd-constr}.
\end{prop}

\begin{rem}
In ride-hailing applications, one can consider a natural setting where the platform offers (fixed) discounts to customers to compensate for their inconvenience when the pick-up and drop-off locations are different from what the customers want.
This setting can be incorporated into our JPA model by interpreting $\{c_{j\tau k}\}$'s as such compensations. 
For example, for $\tau=(j',k')$, the platform can set $c_{j\tau k}=0$ when $j=j',k=k'$, and $c_{j\tau k}>0$ for other $j\in\cP(j'),k\in\cD(k')$. 
The platform can set the discount $c_{j\tau k}$'s through offline learning.
\end{rem}

\subsection{Lyapunov Analysis for JPA}\label{subsec:lyap-jpa}
For JPA setting, we have the following lemma which is analogous to Lemma \ref{lem:one_step_lyap}.
\begin{lem}\label{lem:opt_gap_decomp_jpa}
	Consider congestion function $f(\cdot)$ that is strictly increasing and continuously differentiable. 
	We have the following decomposition:
	\begin{align}
		W^* - \mathbb{E}[v^{\textup{MBP}}[t]|\nq[t]]
		\leq &\ \tK\left(F (\nq[t]) - \mathbb{E}[F (\nq[t+1])|\nq[t]]\right)
		+
		\frac{1}{2 \tK} \max_{j\in V} f_j'(\bar{q}_j[t]) \label{eq:opt_gap_decomp_jpa}\\
		&+ \left(\Wspp
		- \gJPA(\bof(\nq[t])) \right)
		+
		\mathds{1}\left\{q_{j}[t]=0,\exists j\right\}\, ,
		\nonumber
	\end{align}
	where $\gJPA(\by)$ is defined in \eqref{eq:partial_dual_jpa}.
	%
\end{lem}

\begin{proof}[Proof Sketch]
	The proof is analogous to Lemma \ref{lem:one_step_lyap}.
	To use the strong duality argument, we prove below that $\gJPA(\cdot)$ defined in \eqref{eq:partial_dual_jpa} is indeed the partial dual function of the SPP \eqref{eq:fluid-jpa-obj}-\eqref{eq:fluid-jpa-dmd-constr}.
	Then because the primal problem is a concave optimization problem with linear constraints, strong duality must hold.
	
	Let $\by$ be the Lagrange multipliers corresponding to constraints \eqref{eq:fluid-jpa-flow-bal}. We have
	\begin{align*}
		\gJPA(\by) =\; &
		\max_{\sum_{j \in \cP(\tau),k\in\cD(\tau)}x_{j\tau k}\leq 1,x_{j\tau k} \geq 0}
		\sum_{\tau \in \cT}
		\phi_{\tau}
		\left(
		r_{\tau}\left(\sum_{j\in \cP(\tau),k\in\cD(\tau)}x_{j\tau k}\right) \right. \\
		&\ \left.
		+
		\sum_{j\in \cP(\tau),k\in\cD(\tau)}\left(-c_{j\tau k} + y_j - y_k
		\right)x_{j\tau k}
		\right)\\
		=\; &\sum_{\tau \in \cT}
		\phi_{\tau}
		\max_{\sum_{j \in \cP(\tau),k\in\cD(\tau)}x_{j\tau k}\leq 1,x_{j\tau k} \geq 0}
		\left(
		r_{\tau}\left(\sum_{j\in \cP(\tau),k\in\cD(\tau)}x_{j\tau k}\right) \right. \\
		&\ \left. +
		\sum_{j\in \cP(\tau),k\in\cD(\tau)}\left(
		-c_{j\tau k} + y_j - y_k
		\right)x_{j\tau k}
		\right)\, .
	\end{align*}
	Let $\mu_{\tau}=\sum_{j \in \cP(\tau),k\in\cD(\tau)}x_{j\tau k}$, we have
	\begin{align*}
		\gJPA(\by) =\; &\sum_{\tau\in \cT}
		\phi_{\tau}
		\max_{0\leq\mu_{\tau}\leq 1}
		\max_{\sum_{j \in \cP(\tau),k\in\cD(\tau)}x_{j\tau k}= \mu_{\tau},x_{j\tau k} \geq 0}
		\left(
		r_{\tau}\left(\mu_{\tau}\right)
		+
		\sum_{j\in \cP(\tau), k\in \cD(\tau)}\left(
		-c_{j\tau k} + y_j - y_k
		\right)x_{j\tau k}
		\right)\\
		=\; &\sum_{\tau \in \cT}
		\max_{0\leq\mu_{\tau}\leq 1}
		\left(
		r_{\tau}\left(\mu_{\tau}\right)
		+
		\mu_{\tau}\max_{j\in\cP(\tau),k\in\cD(\tau)}	
		\left(-c_{j\tau k} + y_j - y_k\right)
		\right)\, .
	\end{align*}
\end{proof}

\subsection{Proof of Theorem \ref{thm:main_regret_jpa}}
The following lemma is the counterpart of Lemma \ref{lem:dual_subopt} for the JPA setting.
\begin{lem}\label{lem:local_polyhedral_jpa}
	Consider congestion function $f(\cdot)$ that is strictly increasing and continuously differentiable, and any $\bphi$ with connectedness $\alpha(\bphi,\cP,\cD)>0$.
	We have
	\begin{align*}
		\gJPA(\by) - \Wspp\geq\
		\alpha(\bphi,\cP,\cD)\cdot \left[\max_{j\in V}y_j - \min_{k\in V}y_k - 2m \right]^+\, ,
	\end{align*}
	where $\Wspp$ is the value of SPP \eqref{eq:fluid-jpa-obj}-\eqref{eq:fluid-jpa-dmd-constr}, and $\alpha(\bphi,\cP,\cD)$ is defined in \eqref{eq:strong-connectivity-jea}.
\end{lem}

\begin{proof}[Proof Sketch]
	The proof is a direct extension of the proof of Lemma \ref{lem:dual_subopt}.
	The key observation is that: if $y_j-y_k\geq 2 \geq 2\max_{j,k\in V,\tau\in\cT}|c_{j\tau k}|+\bar{p}$, then for any $\tau\in\cP^{-1}(j)\cap \cD^{-1}(k)$ we have
	\begin{align*}
		\textup{argmax}_{\{0\leq\mu_{\tau}\leq 1\}}
		\left(
		r_{\tau}(\mu_{\tau})
		+
		\mu_{\tau}\cdot
		\max_{j\in\cP(\tau),k\in\cD(\tau)}
		\left(
		-c_{j\tau k} + y_j - y_k
		\right)
		\right)&=1\, ,
	\end{align*}
	for any $\tau\in\cP^{-1}(k)\cap \cD^{-1}(j)$ we have:
	\begin{align*}
		\textup{argmax}_{\{0\leq\mu_{\tau}\leq 1\}}
		\left(
		r_{\tau}(\mu_{\tau})
		+
		\mu_{\tau}\cdot
		\max_{k\in\cP(\tau),j\in\cD(\tau)}
		\left(
		-c_{k\tau j} + y_k - y_j
		\right)
		\right)&=0\, .
	\end{align*}
\end{proof}

\begin{proof}[Proof sketch for Theorem \ref{thm:main_regret_jpa}]
	The proof is a direct extension of the proof of Theorem \ref{thm:main_regret}, and follows from Lemmas \ref{lem:opt_gap_decomp_jpa}, \ref{lem:local_polyhedral_jpa}, and the JPA counterpart of Lemma~\ref{lem:bound_rhs_jea} (which is almost identical to Lemma~\ref{lem:bound_rhs_jea}, and was hence omitted).
	We bound $M_1$ using Lemma \ref{lem:inv-sqrt-satisfy-cond}.
\end{proof}

\section{Appendix to Section~\ref{subsec:numerics-ridehailing}}\label{appen:subsec:numerics-ridehailing}

In Appendix \ref{subsec:sims_setup}, we provide 
a full description of our simulation environment and the benchmarks we employ.
In Appendix \ref{subsec:sims-large-network}, we investigate the performance of MBP policy when the network size (i.e., number of locations) grows large.

\yka{Need to update the organization of the section, and summarize the findings for large networks in the main paper.}

\subsection{Simulation Setup}\label{subsec:sims_setup}
Throughout the numerical experiments, we use the following model primitives.
\begin{figure}[!t]
	\centering
	\includegraphics[width=0.4\linewidth]{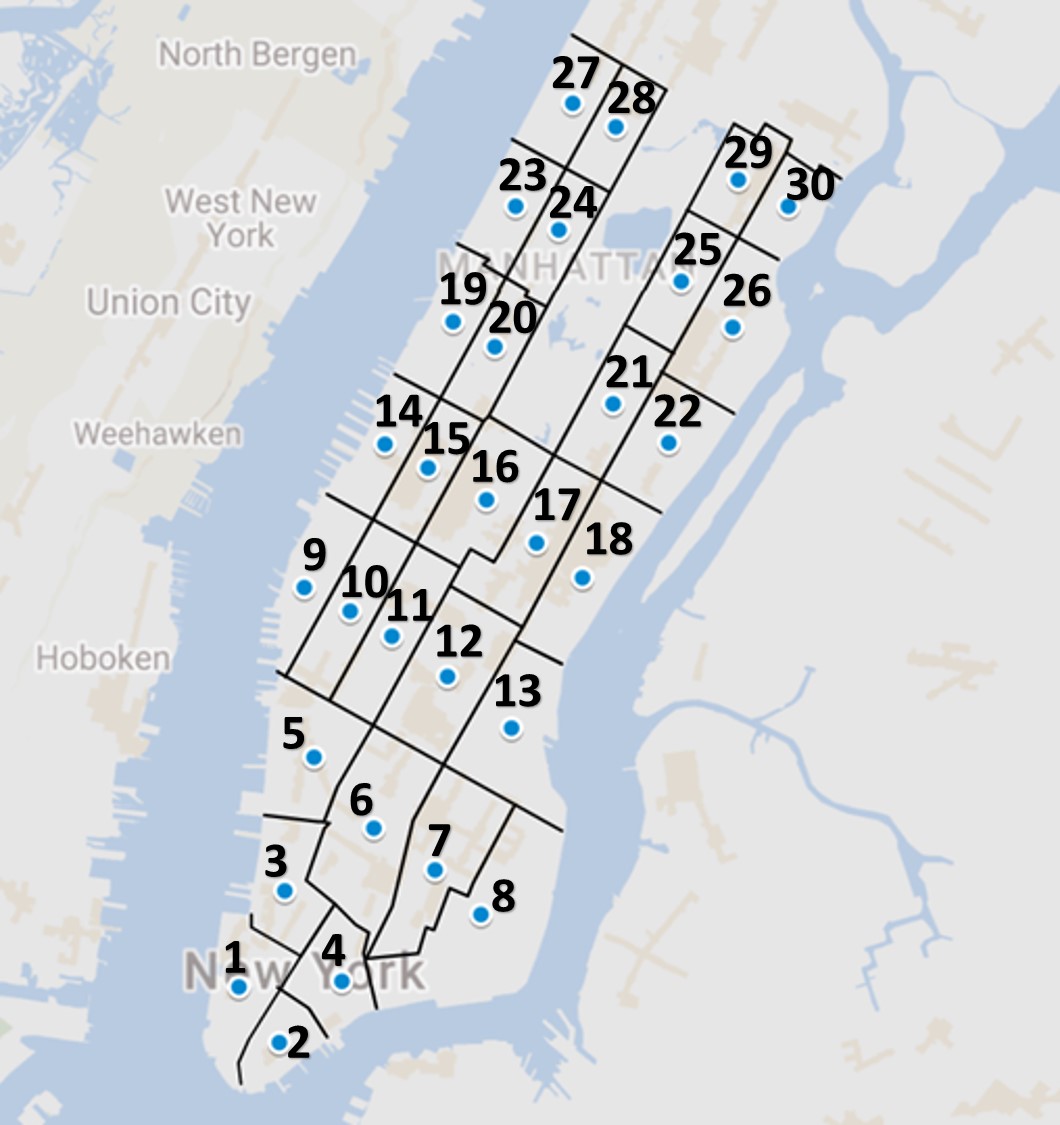}
	\caption{A $30$ location model of Manhattan below $110$-th street, excluding the Central Park. (tessellation is based on \cite{buchholz2015spatial})
	}\label{fig:nyc30}
\end{figure}
%
\begin{compactitem}[wide, labelwidth=!, labelindent=0pt]
	\item \textit{Payoff structure.}
	In many scenarios, ride-hailing platforms take a commission proportional to the trip fare, which increases with trip distance/duration.
	Motivated by this, we present results for $w_{ijk}$ set to be the travel time from $j$ to\footnote{We tested a variety of payoff structures, and found that our results are robust to the choice of $\bw$. One set of tests was to generate $100$ random payoff vectors $\bw$, with each $w_{ijk}$ drawn i.i.d. from Uniform(0,1); we found that the results obtained are similar.} 
	$k$.
	
	\item \textit{Graph topology.}
	We consider a $30$-location model of Manhattan below $110$-th street excluding Central Park (see Figure \ref{fig:nyc30}), as defined in \citet{buchholz2015spatial}.
	We let pairs of regions which share a non-trivial boundary be pickup compatible with each other, e.g., regions $23$ and $24$ are compatible but regions $23$ and $20$ are not.
	
	\item \textit{Demand arrival process, and pickup/service times.}
	We consider a stationary demand arrival process, whose rate is the average decensored demand from 8 a.m. to 12 p.m. estimated in~\cite{buchholz2015spatial}. 
	\yka{This appendix needs to be fixed up. It seems to have been copy-pasted from the main paper for now.}
	This period includes the morning rush hour and has significant imbalance of demand flow across geographical locations (for many customers the destination is in Midtown Manhattan).\footnote{We also simulated the MBP and greedy policy with time-varying demand arrival rates, where the demand arrival rate is estimated (from the real data) for every $5$ min interval.
		Our MBP policy still significantly outperforms the greedy policy.
		\yka{Relocated this footnote and deleted the last sentence about how it's hard to compute $\Wspp$ with time-varying demand.}}
	We estimate travel times between location pairs using Google Maps.\footnote{We extract the pairwise travel time between region centroids (marked by the dots in Figure \ref{fig:nyc30}) using Google Maps, denoted by $\hat{D}_{ij}$'s ($i,j=1,\cdots,30$).
		We use $\hat{D}_{jk}$ as service time for customers traveling from $j$ to $k$.
		For each customer at $j$ who is picked up by a supply from $i$ we add a pickup time~\footnote{We use the inflated $\hat{D}_{ij}$'s as pickup times to account for delays in finding or waiting for the customer.} of $\tilde{D}_{ij} = \max\{\hat{D}_{ij},2 \textup{ minutes}\}$.
		The average travel time across all demand is $13.1$ minutes, and the average pickup time is about $4$ minutes (it is policy dependent).}
	
	\item \textit{Number of cars, and steady state upper bound.}
	\begin{compactitem}[wide, labelwidth=!, labelindent=0pt]
		\item \textit{Excess supply. } We use as a baseline the fluid requirement $K_{{\rm \tiny fl}}$ on number of cars needed to achieve optimal payoff.
		A simple workload conservation argument (using Little's Law) gives the fluid requirement as follows.
		Applying Little's Law, if the optimal solution $\bz^*$ of SPP \eqref{eq:fluid_linear_obj-z}-\eqref{eq:fluid_linear_dmd_constr-z} is realized as the average long run assignment, the mean number of cars who are currently occupied, i.e. serving or picking up customers is \yka{deleted ``at least''}
		$
		\sum_{j,k\in V}\sum_{i \in \cP(j)}D_{ijk} \cdot z^*_{ijk}\, ,
		$ for $D_{ijk} \triangleq \tilde{D}_{ij} + \hat{D}_{jk}$,
		where $\tilde{D}_{ij}$ is the pickup time from $i$ to $j$ and $\hat{D}_{jk}$ is the travel time from $j$ to $k$.
		In our case, it turns out that $K_{{\rm \tiny fl}} = 7,307$.
		We use $1.05\times K_{{\rm \tiny fl}}$ as the total number of cars in the system to study the excess supply case, i.e., there are $5\%$ extra (idle) cars in the system beyond the number needed to achieve the $\Wspp$ benchmark.
		\item  \textit{Scarce supply. }When the number of cars in the system is fewer than the fluid requirement, i.e., $K= \kappa K_{{\rm \tiny fl}}$ for $\kappa<1$, no policy can achieve a steady state performance of $\Wspp$.
		A tighter upper bound on the steady state performance is then the value of the SPP \eqref{eq:fluid_linear_obj-z}-\eqref{eq:fluid_linear_dmd_constr-z} with the additional supply constraint
		\begin{align*}
		\sum_{j,k\in V}\sum_{i \in \cP(j)}D_{ijk} \cdot z_{ijk}\leq K\, .
		\end{align*}
		We denote the value of this problem for $K = \kappa K_{{\rm \tiny fl}} $ by $\Wspp({\kappa})$. We study the case $\kappa = 0.75$ as an example of scarce supply. For our simulation environment, it turns out that $\Wspp(0.75) \approx 0.86 \Wspp$, i.e., $0.86 \Wspp$ is an upper bound on the per period payoff achievable in steady state.
	\end{compactitem}
\end{compactitem}

\smallskip
We compare the performance of our MBP-based policy against the following policies:
\begin{compactenum}[label=\arabic*.,leftmargin=*]
	\item \textit{Static (fluid-based) policy.} The fluid-based policy is a static randomization based on the solution to the SPP, given exactly correct demand arrival rates \citep*[see, e.g., ][]{banerjee2016pricing,ozkan2016dynamic}: Let $\bz^*$ be a solution of SPP. When a type $(j,k)$ demand arrives at location $j$, the randomized fluid-based policy dispatches from location $i\in\cP(j)$ with probability $z^*_{ijk}/\phi_{jk}$.
	
	\item \textit{Multi-hop Utility-Delay Optimal Algorithm (UDOA).} UDOA was proposed by \cite{neely2006super}. The policy uses an exponential Lyapunov function to avoid empty queues. The multi-hop UDOA can be viewed as a special case of MBP policy with exponential congestion function $f(q)=  \omega\cdot (e^{\omega (q-q_0)} - e^{\omega (q_0-q)})$, where $\omega,q_0>0$ are positive constants.\footnote{The parameters $\omega,q_0$ we use in the simulations are found using grid search. We set $\omega,q_0$ as the one that achieve the best performance in the grid search.} \yka{How do you set  in your simulations?}
	
	\item \textit{Deficit MaxWeight (DMW) policy.} The DMW policy was proposed by \cite{jiang2009stable} to deal with underflow problems in certain open networks. 
	DMW is a modification of the vanilla Backpressure policy: instead of using lengths of physical queues as congestion costs as in vanilla BP, DMW uses lengths of virtual queues as congestion costs. 
	The dynamics of virtual queues differ from that of physical queues each time the underflow constraint binds: a unit of ``virtual supply unit'' is relocated, but the physical queue lengths do not change.
	(Virtual queue lengths are allowed to be negative.)
\end{compactenum}

\subsection{MBP Policy in Large Networks}\label{subsec:sims-large-network}
Recall that in Corollary \ref{cor:stationary-main-result}, the steady-state optimality gap of MBP is shown to be $O\left(\frac{m^2}{K}\right)$ for congestion function \eqref{eq:inv_sqrt_mirror_map}.
Compared with the $O\left(\frac{m}{K}\right)$ bound for the fluid-based policy proved in \cite{banerjee2016pricing}, our bound for MBP has the same dependence on $K$ but worse dependence on $m$.\footnote{For the simulations of MBP in Section \ref{subsec:sims-large-network}, we use the congestion function $f(q)=-m^{-1/2}\cdot q^{-1/2}$, which is slightly different from \eqref{eq:inv_sqrt_mirror_map}.
This congestion function guarantees that certain ``invariances'' hold, where $f(\frac{1}{m})\equiv 1$ as the network grows large.
Although it has poorer performance guarantee compared with \eqref{eq:inv_sqrt_mirror_map} (in Appendix \ref{append:bound-sims-large-network} we show that its steady-state optimality gap is $O(m^4/K)$), it actually performs better in the simulations.} A natural question is whether the worse dependence on $m$ reflects poorer performance, or if it is a proof artifact.
The following numerical experiment studies this question.

We compare the performance of our MBP policy with the fluid-based policy in \cite{banerjee2016pricing} for different values of fleet size $K$ and network size (i.e., number of locations) $m$. 
%
To study the effect of network size $m$ on the performance of both policies, we consider the entry control setting {where there is no assignment flexibility and $\cP(i,j)=i$, $\cD(i,j)=j$.}
We ``grow'' the network as follows: 
We start with the 30-location network of Manhattan described above. 
We then split each location into $n$ ``child'' locations, resulting in a larger graph with $30n$ locations.
We define the demand arrival rates in the new network as follows: Suppose $j$, $k$ are locations in the original 30-location network, and the type $(j,k)$ demand arrives with rate $\phi_{jk}$. 
Then for $j'$, $k'$ that are the child locations of $j$ and $k$ respectively, we define the arrival rate of type $(j,k')$ in the new network to be $\phi_{jk}/n^2$.
{
Such simulation design allows us to focus on the impact of network size $m$ on policy performance.}
{Since we are interested in steady-state performance, for each experiment we simulate a $20$-hour period and only record the average per customer payoff in the last two hours.}
\yka{clarify that we consider a setting where there is no assignment flexibility (recall that the entry control only model has been removed from the paper). clarify that growing the network in this way, without any assignment flexibility, constrains the ability of the platform to find a car for a customer, and hence hurts performance in a finite system. On the other hand, the benchmark $Wopt$ in fact does not vary with $n$, right?}

\textbf{Results.} The simulation results are shown in Figure \ref{fig:sims-stationary-large-network}. 
{The results demonstrate that MBP consistently outperforms the fluid-based policy in steady state across different choices of $m$ and $K$.
Also, the steady-state suboptimality of MBP appears to scale as $m/K$ (and not $m^2/K$, which was the scaling of our formal upper bound on the optimality gap). Note that $K/m$ is the number of cars per location. 
To make it visually apparent to the reader that the optimality gap depends on $K/m$, in each subfigure of Figure~\ref{fig:sims-stationary-large-network}, we vary $m$ while holding $K/m$ fixed.}
\yka{In all experiments where there are at least 10 cars per location, MBP captures more than $90\%$ of $\Wopt$ in steady state, suggesting that the dependence of the regret on $m$ is in fact linear.} 

\begin{figure}[h]
        \centering
        \subfigure{
            \includegraphics[width=0.4\textwidth]{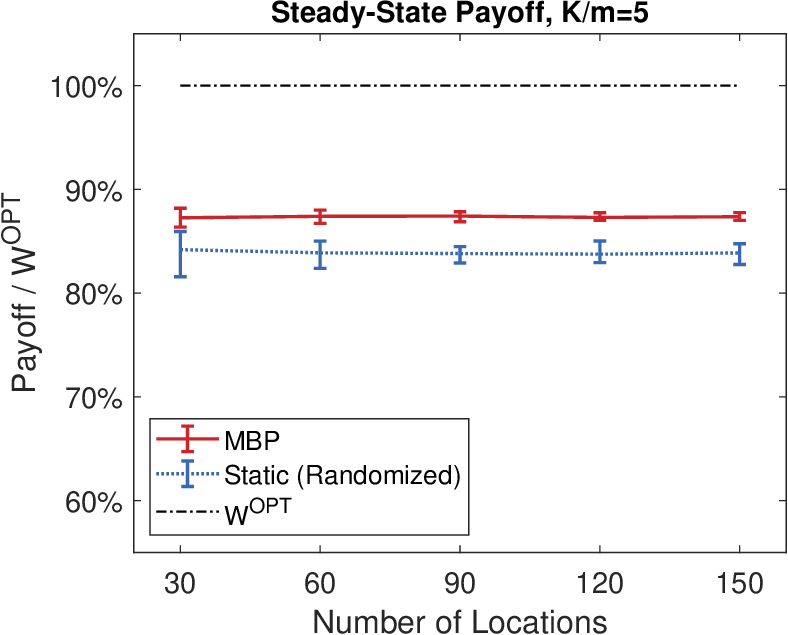}
            }
         \subfigure{
            \includegraphics[width=0.4\textwidth]{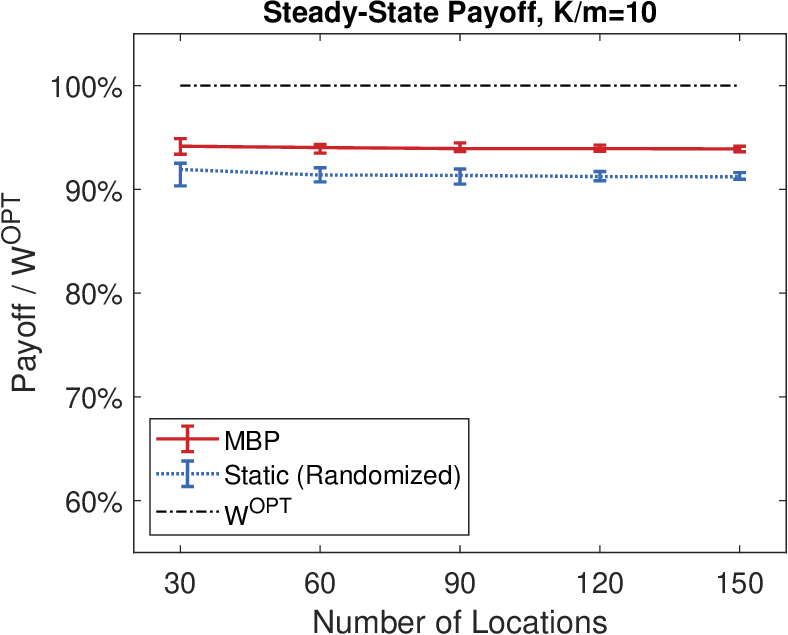}
            }
            \par\bigskip
        \subfigure{
            \includegraphics[width=0.4\textwidth]{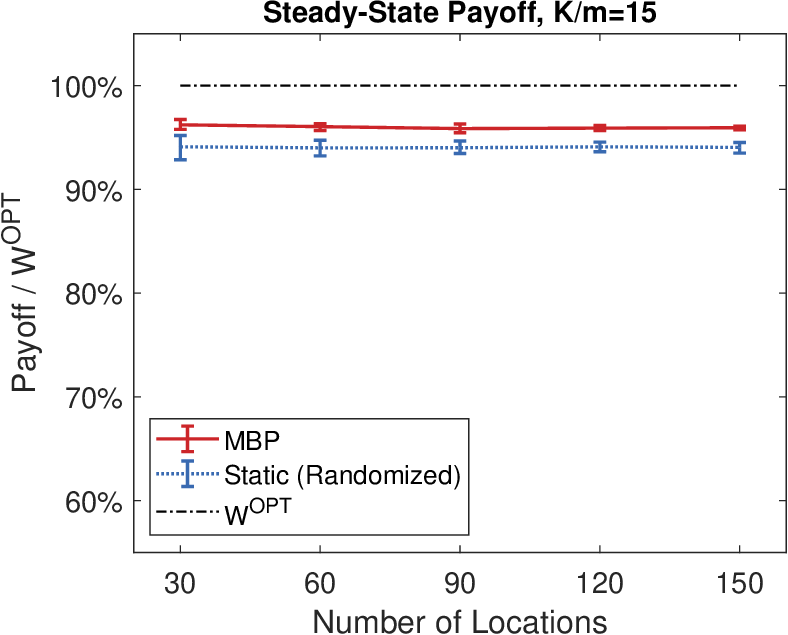}
            }
        \subfigure{
            \includegraphics[width=0.4\textwidth]{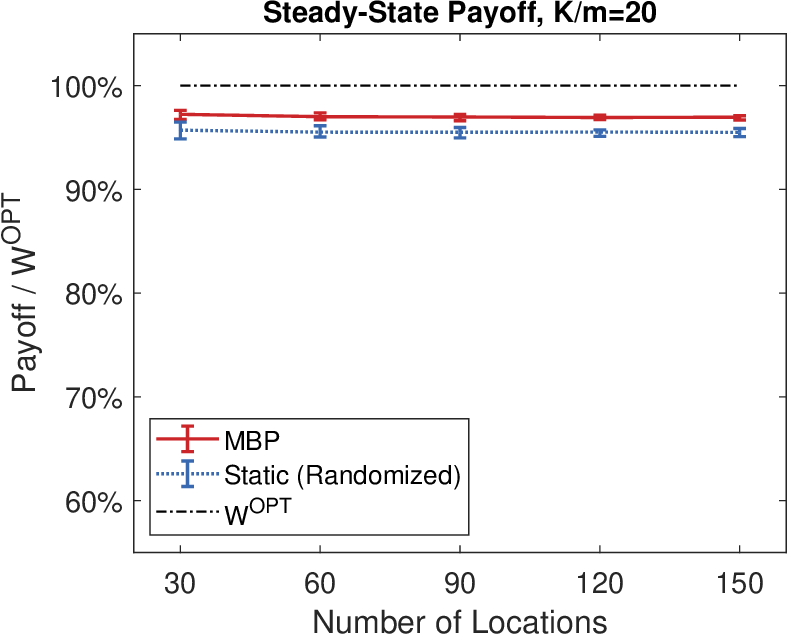}
            }
        \caption{Steady-state per-period payoff under the MBP policy and the static (fluid-based) policy, under values of $K/m$ (number of cars per location). 
        For each data point we generate 100 sample paths and plot the 90\% confidence interval.
        The results suggest that for both MBP and the fluid-based policy, their steady-state suboptimality scales as $K/m$ for this problem instance, and that MBP outperforms the fluid-based policy.
        }
        \label{fig:sims-stationary-large-network}
    \end{figure}

\section{Application to Scrip Systems}\label{sec:scrip-systems}
\yka{Updated this section.}
In this section, we illustrate the application of our model to scrip systems.
A scrip system is a nonmonetary trade economy where agents use scrips (tokens, coupons, artificial currency) to exchange services.
These systems are typically implemented when monetary transfer is undesirable or impractical.
For example, \cite{agarwal2019market} suggest that in kidney exchange, to align the incentives of hospitals, the exchange should deploy a scrip system that awards points to hospitals that submit donor-patient pairs to the central exchange, and deducts points from hospitals that conduct transplantations.
Another well-known example is Capitol Hill Babysitting Co-op \citep[][see also \citealt{johnson2014analyzing}]{sweeney1977monetary}, where married couples pay for babysitting services by another couples with scrips.
A key challenge in these markets is the design of the
admission-and-provider-selection rule:
If an agent is running low on scrip balance, should they be allowed to request services? If yes, and if there are several possible providers for a trade, who should be selected for service?

We introduce a natural model of a scrip system with multiple agents and heterogeneous services, where agents exchange scrips (i.e., artificial currency) for services.
There is a central planner who tries to maximize social welfare by making decisions over whether a trade should occur when a service request arises, and if so, who the service provider should be. The setting is seen to be a special case of the joint entry-assignment (JEA) setting studied in Section \ref{sec:model};
yielding a simple MBP control rule that comes with the guarantee that it asymptotically maximizes social welfare.


\subsection{Model of Scrip Systems}
We now describe a model of a service exchange (i.e., a scrip system). 
Consider an economy with a finite number of agents indexed by $j \in V$.
There are finitely many types of service types $\Sigma$ indexed by $\sigma \in \Sigma$. A demand type $\tau = (j, \sigma)$ is specified by the requestor $j \in V$ along with the requested service type $\sigma \in \Sigma$, i.e., the set of demand types $\cT \subseteq V \times \Sigma$. If the demand is served, the requestor pays a scrip to the service provider. Accordingly, for each demand type $\tau=(j,\sigma)$,  we define  the \emph{compatible} set of agents who can serve it as $\cD(\tau) \subseteq V \backslash \{j\}$.
We again consider a slotted time model, where in each period exactly one service request arises, with demand type drawn i.i.d. from the distribution\footnote{Time-varying demand arrival rates can be seamlessly handled since they are permitted in the JEA setting; we work with stationary arrival rates only for the sake of brevity.} $\bphi = (\phi_\tau)_{\tau \in \cT}$. 
There are a fixed number $K$ of scrips in circulation, distributed among the agents. For each $\tau =(j, \sigma) \in \cT$, serving a demand type $\tau=(j, \sigma)$ generates payoff 
$w_{j \sigma}$.

Observe that our model here is a special case of the JEA setting.\footnote{This can be seen as follows: For each demand type $\tau \in \cT$, the compatible set of service providers $\cD(\tau)$ is identified with the ``dropoff neighborhood'' for $\tau$. The ``pickup neighborhood'' is a singleton set consisting of the requestor $\cP( \tau) = \{ j\} $. Finally,  for each $k \in \cD(\tau)$ we define the payoff $w_{j\tau k} \triangleq w_{j \sigma}$. The primitives $V, \cP, \cD, \bphi$ and $(w_{j\tau k})_{\tau = (j, \sigma) \in \cT, k \in \cD(\tau)}$ fully specify the JEA setting.} 

\medskip
\emph{Comparison with the model in Johnson et al. \citep{johnson2014analyzing}.}
The work \cite{johnson2014analyzing} consider the case where there is only \emph{one} type of service which \emph{all} agents can provide, and requests arrive at the same rate from all agents.
One one hand, we significantly generalize their model by considering heterogeneous service types, general compatibility structures, and asymmetric service request arrivals. They obtain an optimal rule for the symmetric fully connected setting, whereas we develop an asymptotically optimal control rule for the general setting.
On the other hand, we only focus on the central planner setting, and leave the incentives of agents for future work (see the remarks in Section \ref{subsec:scrip_mbp}).

\subsection{MBP Control Rule and Asymptotic Optimality}\label{subsec:scrip_mbp}
Since the model above is a special case of the JEA setting, we immediately obtain an MBP control rule for scrip systems that achieves asymptotic optimality as a special case of Algorithm~\ref{alg:mirror_bp} and Theorem~\ref{thm:main_regret}. This control rule is specified in Algorithm \ref{alg:mbp_scrip} below.
The congestion function $f(\cdot)$ can again be chosen flexibly; we state our formal guarantee for the congestion function in \eqref{eq:inv_sqrt_mirror_map}.  
Denote the normalized number of scrips (defined in \eqref{eq:normalized_queue}) in the possession of agent $i$ by $\bar{q}_i$.

\medskip
\begin{algorithm}[H]\label{alg:mbp_scrip}
	\SetAlgoLined
	{At the start of period $t$, the central planner receives a request from agent $j$ for service type $\sigma$, i.e., demand type $\tau=(j,\sigma)$ arises.}
	
	\eIf{$w_{j\sigma} + f(\bar{q}_j[t])-\min_{k\in \cD(\tau)} f(\bar{q}_k[t])\geq 0$ \textup{\textbf{and}} $\bar{q}_i[t] > 0$}{
		$k^*\leftarrow \textup{argmin}_{k\in\cD(\tau)} f(\bar{q}_k[t])$,\\
		Let agent $k^*$ provide the service to $j$, and agent $j$ gives one scrip to agent $k^*$ \;
	}{
		Reject the service request from agent $j$\;
	}
	\caption{MBP Admission-and-provider-selection rule for scrip systems}
\end{algorithm}
\medskip

Theorem~\ref{thm:main_regret} immediately
implies the following performance guarantee for Algorithm \ref{alg:mbp_scrip}. 
%

%

%
\begin{cor}\label{prop:regret_scrip}
	Consider a set of $m$ agents and any demand type distribution and compatibilities $(\bphi,\cP,\cD)$ (where $\cP$ is identity) that satisfy Condition \ref{cond:str_connect_jea}. Then there exists $K_1 = \textup{poly}\left(m,\frac{1}{\alpha(\bphi, \cP,\cD)}\right )$ and a universal $C>0$ that does not depend on $m$, $K$ or $\alpha(\bphi,\cP,\cD)$, such that for the congestion function $f(\cdot)$ defined in \eqref{eq:inv_sqrt_mirror_map}, for any $K\geq K_1$, the following guarantee holds for Algorithm \ref{alg:mbp_scrip}
	\begin{align*}
		L^{\textup{MBP}}_T\leq
		M_1\cdot \frac{K}{T}
		+
		M_2
		\cdot
		\frac{1}{K}\, ,
		\qquad\textup{and}\qquad
		L^{\textup{MBP}} \leq
		M_2
		\cdot
		\frac{1}{K}\, , \qquad \textup{for} \ M_1 \triangleq Cm \ \textup{and} \ M_2 \triangleq C m^2 \,.
	\end{align*}
\end{cor}
%
%

A few remarks on the model and results are in order:
\begin{compactenum}[1.,wide, labelwidth=!, labelindent=0pt]
	\item \emph{Necessity of declining trades.} By considering a more general setting than in \cite{johnson2014analyzing}, we obtain qualitatively different insights on the optimal control rule by central planner.
	In \cite{johnson2014analyzing}, it is optimal for the central planner to always approve trades, and let the agent with fewest scrips be the service provider.
	In our general setting, however, in many cases the central planner has to decline a non-trivial fraction of the trades to sustain flow balance of scrips in the system (constraint \eqref{eq:fluid_flow_bal}).\footnote{For example, consider a setting with two agents $j_1$ and $j_2$. Denote the demand type requested by $j_1$ as $\tau_1$  (this demand type can be served by $j_2$) and similarly define $\tau_2$.
		Under the mild condition $\phi_{\tau_1} \neq \phi_{\tau_2}$, the planner will be forced to decline a positive fraction of requests.}
	When a trade is approved, our policy also chooses the compatible trade partner with the fewest scrips as service provider.
	\item \emph{Incentives}. Our analysis of scrip systems is meant to illustrate the versatility of MBP control policies, hence we only focused on the central planner setting.
	It would be interesting to study the MBP control rule in the decentralized setting where the agents recommended to be potential trading partners can decide whether to trade, but that is beyond the scope of the current paper. (At a high level, we expect that agents will have an incentive to provide service whenever requested by the MBP policy as long as (i) agents are sufficiently patient, and (ii) agents benefit from trading, i.e., agents derive more value from receiving service than the cost they incur from providing service.)
	%
	
\end{compactenum}


\section{Additional Proofs and Examples}\label{append:additional-proofs}
\subsection{Greedy policy typically incurs $\Omega(1)$ loss}\label{append:greedy}
	Consider a network with three nodes $V=\{1,2,3\}$. Each demand type $\tau=(i,j)$ corresponds to an origin-destination pair, and that $\cP(i,j)=\{i\}$, $\cD(i,j)=\{j\}$.
	The demand arrival probabilities are: $\phi_{12}=\epsilon$, $\phi_{23} = \frac{1}{3} + \epsilon$, $\phi_{21} = \phi_{32} = \frac{1}{3} - \epsilon$ (where $0<\epsilon <\frac{1}{6}$), and payoffs are: $w_{23}=w>0$, $w_{12}=w_{21}=w_{32}=\frac{w}{2}$.
	Let $\bx^*$ be the optimal solution to the SPP($\bphi$) \eqref{eq:fluid_obj}-\eqref{eq:fluid_dmd_constr}. By inspection, $\bx^*$ should induce the maximum circulation in each of the two cycles \textup{1---2---1} and \textup{2---3---2}, 
	hence $x_{12}^*=x_{32}^*=1$,
	$x_{21}^*=\frac{\epsilon}{\frac{1}{3}-\epsilon}$, $x_{23}^*=\frac{\frac{1}{3}-\epsilon}{\frac{1}{3}+\epsilon}$.
	We know that there exists a policy whose performance approaches the value of the SPP as $K \to \infty$ \citep*{banerjee2016pricing}.
	We will prove by contradiction that the greedy policy incurs an $\Omega(1)$ loss for this example, by showing that its payoff per period is $\Omega(1)$ below the value of the SPP. Consider the steady state under the greedy policy.
	Suppose the loss is vanishing, i.e., all but an $o(1)$ fraction of type $(1,2)$ and type $(3,2)$ demand are served.
	Suppose a $\gamma$ fraction of the time there is a supply unit present at node $2$. As a result, since the greedy policy is being used, a $\gamma$ fraction of demands of type $(2,1)$ are served, and a $\gamma$ fraction of demands of type $(2,3)$ are served. Flow-balance at nodes $1$ and $3$, respectively, implies that we have $(\frac{1}{3}-\epsilon)\gamma=\epsilon-o(1)$, $(\frac{1}{3}+\epsilon)\gamma=\frac{1}{3}-\epsilon-o(1)$.
	However, these two equations cannot both be satisfied as $K \to \infty$ unless $\epsilon=\frac{1}{9}$.
	We infer that the greedy policy incurs an $\Omega(1)$ loss in this network for any $\epsilon\in(0,\frac{1}{6})$, $\epsilon\neq \frac{1}{9}$.

\subsection{Failure of standard backpressure}\label{append:bp_fails}

In this section, we show that BP can fail in dealing with no-underflow constraints. 
Recall that BP uses a linear function of normalized queue lengths, i.e., $\by[t] = c\cdot \bar{\bq}[t]$ (where $c>0$), as shadow prices.
Intuitively, this could cause issues since $\bar{\bq}$ is subject to (physical) state space constraints, but the shadow prices $\bar{\by}$ can take any value since they are the Lagrange multipliers of equality constraints \eqref{eq:fluid_flow_bal}. Indeed, we construct an example such that for all $\bar{\bq} \in \Omega$, the corresponding shadow prices $c\cdot \bar{\bq}$ are suboptimal for the dual problem \eqref{eq:partial_dual}, when the proportionality constant $c$ is not chosen to be sufficiently large.

\begin{examp}[BP is far from optimal if $c$ is not large enough]\label{examp:bp-fail}
	Consider the network introduced in Appendix~\ref{append:greedy}.\yka{Any constraint on $c$?}
	%
	Suppose the platform employs backpressure where the shadow prices are taken to be proportional to (normalized) queue lengths $\by[t] = c\cdot\nq[t]$ with $c<\frac{3}{2}w$.\yka{changed from $c\bq/K$ to be consistent with Sec 4.2.}
	
	Let $\by^*$ be the optimal dual variables in \eqref{eq:partial_dual}. By complementary slackness we have that the set of dual optima are $\by^*$ which satisfy
	\begin{align*}
		\frac{w}{2} + y_1^* - y_2^*\geq 0\, ,\quad
		\frac{w}{2} + y_2^* - y_1^* =0\, , \quad
		w+y_2^*-y_3^*=0\, ,\quad
		\frac{w}{2}+y_3^*-y_2^*\geq 0\, .
	\end{align*}
	Hence $\by^*$ takes the form $\by^*=(y_1^*,y_1^*-\frac{w}{2},y_1^*+\frac{w}{2})$ for arbitrary $y_1^* \in \mathbb{R}$.
	Let $\nq^* \triangleq \by^*/c$ be the queue lengths corresponding to the optimal dual variables in \eqref{eq:partial_dual} with the additional constraint that the normalized queue lengths sum to $1$.
	Simple algebra yields $\nq^* = (\frac{1}{3}, \frac{2c-3w}{6c}, \frac{2c+3w}{6c})$.
	Because $c<\frac{3}{2}w$ we have $\bar{q}^*_2 < 0$, and so $\nq^*$ lies outside the normalized state space $\nq^* \notin \Omega$. Hence, the $\nq[t]$ will never converge to $\nq^*$ and BP is far from optimal.
\end{examp}

Even if the platform uses BP with sufficiently large $c$ to ensure that there are optimal dual variables $\by^*$ which correspond to a queue length vector, i.e., $\by^* = c\cdot \nq$ for some $\nq\in\Omega$, we construct an example in  where the existing analysis of BP still fails. 
This is because in the analysis of BP one needs to show that the Lyapunov function decreases in expectation in each period, which roughly means that the queue lengths are ``pushed'' towards the interior of state space under BP. However, in our example we show that the Lyapunov drift could be positive at certain states due to underflow constraints.
%
%


\begin{examp}[BP has positive Lyapunov drift at a certain state]\label{examp:bp-lyapunov-fail}
	Again consider the example in Appendix~\ref{append:greedy} and let $c\geq \frac{3}{2}w$.
	A typical analysis of BP is based on establishing that the ``drift'' defined by
	\begin{align*}
		\mathbb{E}\Big [ \, \lVert \nq[t+1] - \nq^* \rVert_2^2 \, \Big| \, \nq[t] \, \Big ]
		-
		\lVert \nq[t] - \nq^*\rVert_2^2
	\end{align*}
	is strictly negative when $\lVert \nq[t] - \nq^*\rVert_2 = \Omega(1)$.
	Suppose at time $t$ we have\footnote{The integrality of the components of $\bq[t]$ is non-essential, hence we assume all components of $\bq[t]$ are integers. Also, here we take the normalized queue lengths to be defined as $\nq[t] \triangleq \bq[t]/K$ to simplify the expressions.} $\nq[t] = (\frac{2}{3}, 0 ,\frac{1}{3})$; in particular, queue 2 is empty.
	Note that at $\nq[t]$, BP can only fulfill the demand going from $1$ to $2$ and from $3$ to $2$
	because of the no-underflow constraint.
	Straightforward calculation shows that the ``drift'' is positive for large enough $K$ if $\epsilon < \frac{w}{2c + 3w}$. \yka{The threshold I found was $\epsilon < w/(4c)$ and that the drift is positive for large enough $K$, not for arbitrary $K$. Please check. Maybe I'm missing something.} 
	%
\end{examp}

\subsection{Additional intuitions of why MBP alleviates the underflow problem}\label{appen:additional-intuition}
Standard Backpressure could run into two issues due to the underflow constraints: (i) The queue lengths corresponding to the optimal dual variables lie outside of the state space; (ii) The Lyapunov drift could be positive at certain ``boundary states'', i.e., states where some of the queues are empty.
		
		In the analysis, we show that the underflow problem is provably alleviated by MBP policies with an appropriately chosen congestion function. We provide some high-level intuitions below.
		Consider the MBP policy with congestion function $\bof(\cdot)$ given in \eqref{eq:inv_sqrt_mirror_map} and normalized queue lengths $\bar{\mathbf{q}}$ defined in \eqref{eq:normalized_queue}.
		
		For issue (i), first note the following crucial fact in our closed network setting, which is that we can add any constant to each coordinate of $\by^*$ and we will get another dual optimum.
		Under standard BP, the range of congestion functions $c\bar{q}_j$ is a bounded interval $[0,c]$ since $\bar{q}_j\in[0,1]$, hence it is possible that all dual optima lie outside of the range of $c\nq$.
		In contrast, $f(\bar{\mathbf{q}})$ maps $[0,1]$ to $(-\infty,-\sqrt{m})$ as $K\to\infty$, hence the range of $f(\bar{\mathbf{q}})$ is only right-bounded.
		Intuitively, for an optimal dual vector $\by^*$, even if $\by^*\notin \Omega$, there could exist $\nq\in\Omega$ such that $\by^* = \bof(\nq)$.
		Indeed, we can show that for large enough $K$, we can always find $\bar{\mathbf{q}}\in\Omega$ such that $\bof(\bar{\mathbf{q}})$ corresponds to 
an optimal dual vector. Under MBP, the normalized queue length converges to that $\mathbf{q}$.
\yka{[YK: Isn't it crucial that in our closed network setting we can add any constant to each coordinate of $\by^*$ and we will get another dual optimum? The current answer seems incomplete and doesn't clearly convey how the problem is resolved.]}
\yka{The following logic works only if you assume $\delta_K$ is $0$. It would be better to be careful in explaining how the argument extends to our choice of $\delta_K$.}

For issue (ii), first note that this problem only occurs when there exists an empty queue.
		At these states, the dual-suboptimality at $\bof(\bar{\mathbf{q}})$ is large (for the reasons that we provide below), which creates a negative Lyapunov drift that ``pushes'' $\bof(\bar{\mathbf{q}})$ towards the optimal dual variable.
		This corresponds to the intuition that MBP is more aggressive in preserving supply units in near-empty queues compared to BP, making the system less likely to violate the no-underflow constraints.
		This intuition is formalized by Lemma \ref{lem:dual_subopt}, where we show that the dual-suboptimality increases with $\max_j f_j(\bar{q}_j) - \min_j f_j(\bar{q}_j)$.
		Because $f(0)\approx -\sqrt{m} K^{1/4}$, , and that $\max_j f_j(\bar{q}_j)\geq f(1/m)$, we have that $\max_j f_j(\bar{q}_j) - \min_j f_j(\bar{q}_j)$ is indeed large (of order $\Omega(\sqrt{m} K^{1/4})$) when there are empty queues, hence the dual-suboptimality is also large.

\subsection{Performance guarantee of MBP with the congestion function used in Section \ref{subsec:sims-large-network}}\label{append:bound-sims-large-network}
In the numerical experiments in Section \ref{subsec:sims-large-network}, we use the following congestion:
\begin{align}
	f(\bar{q}_j) \triangleq -\frac{1}{\sqrt{m}}\cdot\bar{q}_j^{-\frac{1}{2}}\, , \label{eq:inv_sqrt_mirror_map_new}
\end{align}
instead of \eqref{eq:inv_sqrt_mirror_map}.
In the following, we prove that the steady-state suboptimality of MBP with congestion function \eqref{eq:inv_sqrt_mirror_map_new} is $O\left(\frac{m^4}{K}\right)$.
Here we assume that all buffers have infinite capacity, i.e., $V_b = \emptyset$, which is the case for the simulations in Section \ref{subsec:sims-large-network}.

\begin{prop}\label{prop:stationary-improved}
	Consider the setting in Theorem \ref{thm:main_regret}.
	When the demand arrivals are stationary ($\eta=0$), for any $K\geq K_1$, the following infinite-horizon guarantee holds for Algorithm~\ref{alg:mirror_bp} with congestion function \eqref{eq:inv_sqrt_mirror_map_new}:
	%
	\begin{align*}
		L^{\textup{MBP}} \leq
		M_2
		\frac{1}{K}\, , \qquad \textup{for} \ M_2 = Cm^4 \, .
	\end{align*}
\end{prop}

As in the proof of Theorem \ref{thm:main_regret}, the proof of this proposition boils down to verifying the growth conditions in Condition \ref{cond:mirror_map} for congestion function \eqref{eq:inv_sqrt_mirror_map_new}.

\begin{lem}\label{lem:inv-sqrt-new-satisfy-cond}
	The congestion function \eqref{eq:inv_sqrt_mirror_map_new} satisfies the growth conditions (Condition~\ref{cond:mirror_map}) with parameters $(\alpha_{\min},K_1,M_1,M_2)$ where
	$$
	K_1 = \textup{poly}\left(m,\frac{1}{\alpha_{\min}}\right)\, ,\quad M_1=C\, ,\quad M_2=C m^4\, .
	$$
	Here $C$ is a universal constant that is independent of $m$, $K$ and $\alpha_{\min}$.
\end{lem}

We first prove Proposition \ref{prop:stationary-improved} using Lemma \ref{lem:inv-sqrt-new-satisfy-cond}.

\begin{proof}[Proof of Proposition \ref{prop:stationary-improved}]
It follows from Lemma \ref{lem:inv-sqrt-new-satisfy-cond} that Lemma \ref{lem:rhs-bound-general} holds with $M_1 = C, M_2=Cm^4$. 
The rest of the proof are exactly the same as that of Theorem \ref{thm:main_regret}.
\end{proof}

Now we prove Lemma \ref{lem:inv-sqrt-new-satisfy-cond}.
\begin{proof}[Proof of Lemma \ref{lem:inv-sqrt-new-satisfy-cond}]
We verify Condition \ref{cond:mirror_map} for congestion function \eqref{eq:inv_sqrt_mirror_map_new} point by point below.
\begin{itemize}
		\item  It is straightforward to see that point 1 holds.
		\item Point 2(a). For $\bq$ such that $\nq\in\bar{\mathcal{B}}(\bof)$ and $q_j>0$ for any $j\in V$, we have, by definition of $\nq\in\bar{\mathcal{B}}(\bof)$,
		\begin{align*}
			\textup{LHS of \eqref{eq:cond-eq-1}}\geq 2m\alpha\, .
		\end{align*}
		On the other hand, we have for $K = \Omega(m^2)$, we have
		\begin{align*}
			\textup{RHS of \eqref{eq:cond-eq-1}}= O\left(\frac{1}{K}\cdot m^{-1/2}\cdot  (m\cdot K^{-1/2})^{-3/2}\right)
		\end{align*}
		Here the RHS of \eqref{eq:cond-eq-1} is maximized when $q_j=0$.
		Therefore \eqref{eq:cond-eq-1} holds for $K\geq K_1 = \Omega\left(\max\left\{m^2,\frac{1}{m^{12} \alpha^4}\right\}\right)$.
		For $\bq$ such that $\nq\in\bar{\mathcal{B}}(\bof)$ and $q_j=0$ for some $j'\in V$, we have
		\begin{align*}
			\textup{LHS of \eqref{eq:cond-eq-1}}\geq \alpha \cdot \Omega\left(m^{-1/2}\cdot (m\cdot K^{-1/2})^{-1/2} - m\right)\, ,
		\end{align*}
		which is obtained by plugging in $q_{j'}$.
		For $K=\Omega(m^2)$, we also have
		\begin{align*}
			\textup{RHS of \eqref{eq:cond-eq-1}}= O\left(\frac{1}{K}\cdot m^{-1/2}\cdot  (m\cdot K^{-1/2})^{-3/2} + 1\right)\, ,
		\end{align*}
		Using the analysis above, for $K\geq \Omega\left(1\right)$, the first term in the parentheses is $O(1)$.
		In this case we have $\textup{RHS of \eqref{eq:cond-eq-1}}= O(1)$.
		Therefore \eqref{eq:cond-eq-1} holds for $$K=\Omega\left(\max\left\{m^8,\frac{1}{m^{12}\alpha^4}\right\}\right)\, .$$
		
		%
		\item Point 2(b). Note that for $K=\Omega(m^2)$,
		\begin{align*}
			&\sup_{\bq,\bq'\in \Omega^{K}}\left(
			F(\nq) - F(\nq')
			\right) \\
			\leq \ &  O\left(m^{-1/2}\sup_{\bq,\bq'\in \Omega^{K}}
			\sum_{j\in V}\left(
			\sqrt{\bar{q}'_{j}}
			-\sqrt{\bar{q}_{j}}
			\right)
			\right) \\
			\leq\ & O\left(m^{-1/2} \max_{\bq'\in \Omega^K}\sum_{j\in V}
			\sqrt{\bar{q}'_{j}}
			\right)\\
			=\ & O(1)\, .
		\end{align*}
		Hence
		$
			M_1 = O(1)\, .
		$
		\item Point 2(c). Note that for $\nq\in\mathcal{B}_{\bof}$, we have $\bar{q}_j =\Omega\left(m^{-3}\right)$, hence
		\begin{align*}
			M_2 = \max_{\nq\in\mathcal{B}_{\bof}}\max_{j\in V} \max_{\bar{q}\in\left[\bar{q}_{j}[t]-\frac{1}{\tK},\bar{q}_{j}[t]+\frac{1}{\tK}\right]}f'_j(\bar{q})
			\leq
			m^{-1/2}\cdot m^{9/2}
			=
			O(m^4)\, .
		\end{align*}
		\item Point 2(d). Note that for $\nq\in\mathcal{B}_{\bof}$, we have $\bar{q}_j =\Omega\left(m^{-3}\right)$, hence point 2(d) holds for $K=\Omega(m^8)$.
	\end{itemize}
\end{proof}

\section{MBP Executes Dual Stochastic Mirror Descent on the SPP} 
\label{appen:MBP-as-MD}

\textbf{Review of the interpretation of BP as dual stochastic subgradient descent.} Rich dividends have been obtained by treating the (properly scaled) current queue lengths $\bq$ as the dual variables $\by$, resulting in the celebrated backpressure (BP, also known as MaxWeight) control policy, introduced by  \citet{tassiulas1992stability}\pqa{Note to self: There is a citation formatting problem, which I fixed for Nicholson submissions, etc, but not this version.}, see also, e.g., \cite{stolyar2005maximizing,eryilmaz2007fair}. 
Formally, BP sets the current value of $\by$ to be proportional to the current normalized queue lengths, i.e., $\by[t] = c\cdot\bar{\bq}[t]$ for some $\bar{\bq} \in \Omega$ defined, e.g., as in \eqref{eq:normalized_queue}, and some $c> 0$ and greedily maximizes the inner problem in \eqref{eq:partial_dual} 
for every origin $j$ and destination $k$, i.e.,
\begin{align}
x_{j^*\tau k^*}^{\textup{BP}}[t]=
\left\{
\begin{array}{ll}
1 & \textup{ if } (j^*,k^*)=\textup{argmax}_{j\in\cP(\tau),k\in\cD(\tau)}w_{j\tau k} +
c \bar{q}_j[t] - c \bar{q}_k[t]\geq 0\textup{ and }q_j[t]>0\, ,\\
0 & \textup{ otherwise}\, .
\end{array}
\right.
\label{eq:VBP-defn}
\end{align}
The BP policy can be viewed as a \textit{stochastic subgradient descent (SGD)} algorithm on the dual problem \eqref{eq:partial_dual}, when the current state is in the \textit{interior} of the state space, i.e., when $q_j>0$ for all $j\in V$ \citep{huang2009delay}.
To see this, denote the subdifferential (set of subgradients) of function $g^t(\cdot)$ at $\by$ as $\partial g^t(\by)$.
Observe that \emph{the expected change of queue lengths under BP is proportional to the negative of a subgradient of $g^t(\cdot)$ at $\by = c\cdot\bar{\bq}[t]$}, in particular 
\begin{align}
\hspace{-0.2cm}-\frac{\tK}{c} \mathbb{E}[\by[t+1]-\by[t]]
= - \mathbb{E}[\bq[t+1] -\bq[t] ]=
\sum_{\tau \in \cT}\phi_{\tau}\sum_{j\in\cP(\tau),k\in\cD(\tau)}
x_{j\tau k}^{\textup{BP}}[t](\boe_j - \boe_k)
\in\partial g^t(\by[t])
\label{eq:subgrad-VBP}\, ,
\end{align}
where the first equality follows from the definition $\by[t]= c\cdot\bar{\bq}[t]$ (and the definition of normalized queue length \eqref{eq:normalized_queue}) and second equality is just the expectation of the system dynamics \eqref{eq:system-dynamics-jea}. Here $\sum_{\tau \in \cT}\phi_{\tau}\sum_{j\in\cP(\tau),k\in\cD(\tau)}
x_{j\tau k}^{\textup{BP}}[t](\boe_j - \boe_k)
\ \in\ \partial g^t(\by[t])$ follows from direct observation.
Eq. \eqref{eq:subgrad-VBP} shows that the evolution of $\by[t]$ when $\bq[t]>0$ is exactly an iteration of SGD with step size $\frac{c}{\tK}$.

\textbf{MBP executes dual stochastic mirror descent.} 
A notable property of MBP (given in Algorithm \ref{alg:mirror_bp}) is that it executes stochastic mirror descent on the partial dual problem \eqref{eq:partial_dual} (with flow constraints dualized), with $\nq[t]$ given by \eqref{eq:normalized_queue} being the mirror point and the inverse mirror map being the (vector) congestion function $\mathbf{f}(\nq)\triangleq [f(\bar{q}_1),\cdots,f(\bar{q}_m)]^{\T}$.
Mathematically, if $\bq>0$, we have
\begin{align}
  - \mathbb{E}[\bq[t+1] -\bq[t] ]=
\sum_{\tau \in \cT}\phi_{\tau}\sum_{j\in\cP(\tau),k\in\cD(\tau)}
x_{j\tau k}^{\textup{MBP}}[t](\boe_j - \boe_k)
 \in \partial g^t(\by) \Big |_{\by = \bof(\bar{\bq}[t])}
\label{eq:subgrad-MBP}\, ,
\end{align}
where $\bx^{\textup{MBP}}[t]$ is the control defined in Algorithm \ref{alg:mirror_bp}; notice that the entry rule $\bx^{\textup{MBP}}[t]$ has the same form as that for BP \eqref{eq:VBP-defn} except that it uses a general congestion function $f(\bar{q}_j)$, leading to \eqref{eq:subgrad-MBP} for MBP via the same reasoning that led to \eqref{eq:subgrad-VBP} for BP.
Thus, MBP performs stochastic mirror descent on the partial dual problem \eqref{eq:partial_dual}, which generalizes the previously known fact that BP performs stochastic gradient descent.
A main advantage of mirror descent over gradient descent is that it can better capture the geometry of the state space via an appropriate choice of mirror map \citep[see, e.g.,][]{beck2003mirror}.
%
In our setting, the congestion function $\mathbf{f}(\nq)$ is the inverse mirror map and can be flexibly chosen.

\end{document}